\documentclass[12pt]{article}
\usepackage{amsfonts}
\usepackage{amssymb}
\usepackage{mathrsfs}
\usepackage{srcltx}
\usepackage{amsmath}
\usepackage{adjustbox} 
\usepackage{amsthm}
\usepackage{booktabs} 
\usepackage{diagbox} 
\usepackage{amstext}
\usepackage{amsopn}
\usepackage{graphicx}
\usepackage{caption}
\usepackage{subfigure}
\usepackage{grffile}
\usepackage{float}
\usepackage{epsfig}
\usepackage{subfig}
\usepackage{overpic}
\usepackage{epstopdf}
\usepackage{color}
\usepackage{cite}
\usepackage{indentfirst}
\usepackage{graphicx}
\usepackage{multirow}
\usepackage[pagewise]{lineno}
\graphicspath{{./figs/}}
\DeclareGraphicsExtensions{.eps}
\newtheorem{theorem}{Theorem}[section]
\newtheorem{lemma}[theorem]{Lemma}
\newtheorem{proposition}[theorem]{Proposition}

\theoremstyle{definition}

\numberwithin{equation}{section}

\newcommand{\ba}{\begin{array}}
	\newcommand{\ea}{\end{array}}

\begin{document}
	\date{}
	\title{ \bf\large{ Dynamic Coupling of Infiltration-Soil Moisture Feedback: Emergent Vegetation Patterns in a Water-Vegetation Model}\footnote{This work was partially supported by grants from National Science Foundation of China (12371503, 12071382), Natural Science Foundation of
Chongqing (CSTB2022NSCQ-MSX0284, CSTB2024NSCQ-MSX0992).}
	}
	\author{Juan Yan \textsuperscript{a},\ \ Xiaoli Wang \textsuperscript{a},\footnote{Corresponding Author. Email: wxl711@swu.edu.cn }\ \ Guohong Zhang \textsuperscript{a},\ \ Yuan Yuan\textsuperscript{b}
		\\
		{\small \textsuperscript{a}School of Mathematics and Statistics, Southwest
University, Chongqing 400715,  China\hfill{\ }}\\
{\small \textsuperscript{b}Department of Mathematics and Statistics, Memorial University of Newfoundland,\hfill{\ }}\\
\ \ {\small  St. John's,  Newfoundland, A1C 5S7, Canada \hfill{\ }}}
	
	\maketitle
	
	\noindent
	\begin{abstract}
		{  We present a modified water-vegetation  model to investigate the mechanistic relationship between infiltration-soil moisture feedback and vegetation pattern  in arid/semi-arid ecosystems. Employing Turing pattern formation theory, we drive conditions for  diffusion-induced instability and analyze  spatiotemporal dynamics  near Turing-Hopf bifurcation points. Our key findings include: (i) The system exhibits rich dynamics including multiple stable equilibria, supercritical/subcritical Hopf bifurcations, bubble loops of limit cycles and homoclinic bifurcations. 
(ii) The system admits Turing-Hopf bifurcation. Using normal form theory, we establish the existence of quasiperiodic solutions and mixed-mode oscillations near critical thresholds, providing a mathematical framework for predicting nonlinear ecological regime shifts. (iii) Soil moisture feedbacks govern critical transitions between three distinct ecosystem states: uniform vegetation covering, self-organized spatial patterns (labyrinth/gapped vegetation), and bare soil state, which demonstrates that soil moisture thresholds control the final state selection in this system.}

	\end{abstract}
	\section{Introduction}
Vegetation, often referred to as an \lq\lq ecosystem engineer", plays an irreplaceable role in maintaining ecological balance \cite{jones1994}. Through photosynthesis, vegetation provides energy and sustains ecosystems. Its ground cover mitigates wind erosion, stabilizes sand, and reduces soil erosion. Additionally, transpiration facilitates the transfer of soil moisture to the atmosphere, contributing to global hydrological balance \cite{Lemordant2018,gallagher2019,danielsen2005}.
However, global climate change and the expansion of human activities have led to severe degradation of vegetation ecosystems, intensifying the process of desertification. This issue is particularly pronounced in arid and semi-arid areas, where water scarcity further exacerbates vegetation loss \cite{d2013,sun2022,wang1999}.

Vegetation degradation during desertification does not occur uniformly; instead, it manifests in distinct spatial patterns such as stripes, spots, and labyrinths, each of which has profound ecological implications \cite{saco2007,sherratt2005,sun2018}. For instance, stripe-like vegetation patterns are commonly observed in semi-arid landscapes \cite{sherratt2012}. To elucidate the mechanisms driving these spatial patterns, researchers have developed various water-vegetation models.
In 1999, Klausmeier \cite{klausmeier1999} introduced a seminal vegetation-water model to explain the formation of vegetation stripes on slopes and irregular mosaics on flat terrain, emphasizing the role of nonlinear dynamics in shaping plant community structures. Von Hardenberg \cite{von2001} later extended Klausmeier model by incorporating root competition for water resources, enhancing its ecological realism. Gilad et al. \cite{gilad2007} formulated a mathematical model to analyze woody plant ecosystems in arid regions, capturing multiple feedback mechanisms between biomass and water availability. Recognizing that Klausmeier model did not account for water diffusion in the soil, Vander Stelt  \cite{van2013rise} incorporated this factor into Klausmeier model. Zelnik  \cite{zelnik2015gradual} simplified Gilad's model and utilized empirical data to investigate the Namibian fairy circle ecosystem, demonstrating that spatially expanding ecosystems undergo gradual pattern transitions. Additionally, numerous studies have examined the mechanisms underlying vegetation pattern formation, including feedback effects \cite{sun2022spatial,sun2018,wang2017interaction} and non-local interactions \cite{eigentler2018analysis,liang2022nonlocal,guo2024pattern,liu2025spatiotemporal,xue2020spatiotemporal,yuan2023}.

Among the various factors influencing vegetation pattern formation, soil moisture plays a crucial role \cite{duan2019spatial}. It not only regulates vegetation productivity but also impacts terrestrial carbon sequestration capacity \cite{chen2014using}. Research has shown that fluctuations in soil moisture significantly affect vegetation coverage, with higher moisture levels generally promoting denser vegetation growth \cite{nega2023investigating}. In semi-arid regions, soil moisture availability determines ecosystem sustainability, influencing vegetation recovery and ecosystem services \cite{yang2014response}.

To further explore the role of soil moisture in vegetation pattern formation, we propose an improved water-vegetation model as follows:
	 \begin{equation}\label{eq1.1}
		\begin{cases}
			\displaystyle \frac{\partial w}{\partial t}=d_1 \Delta w+R-aw-\delta wb, & x \in \Omega ,t>0,  \\
			\displaystyle \frac{\partial b}{\partial t}=d_2 \Delta b+\rho b(1-\frac{b}{w})-\mu(w,b)b, & x \in \Omega ,t>0,   \\
			\displaystyle \frac{\partial w}{\partial \nu}=\frac{\partial b}{\partial \nu}=0,&  x\in \partial \Omega, t>0,\\
			\displaystyle w(x,0)=w_0(x) > 0,b(x,0)=b_0(x) \ge 0,& x \in \Omega,
		\end{cases}
	\end{equation}
where $w$ and $b$ represent the densities of water and vegetation, respectively;  $d_1$ and $d_2$ are coefficients of diffusion rate of water and  vegetation, respectively;
 $R$ denotes the mean annual rainfall rate; $a$ is the evaporation rate of soil water;  $\delta$ stands for the consumption rate of water by the vegetation; the term $\mu(w,b)$ is the vegetation mortality rate; $\Omega$ represents the spatial domain, $\partial \Omega$ is the boundary of $\Omega$, and  $\bar{\Omega}= \Omega \cup \partial \Omega$
is the spatial domain that includes its boundary $\partial \Omega$;
$\nu$ stands for the outward normal vector at $\partial \Omega$. Here, the vegetation growth term is $\rho b(1-\frac{b}{w})$ inspired by the Holling-Tanner model\cite{tanner1975stability}, where $\rho$ represents the intrinsic growth rate of the vegetation biomass. Unlike standard logistic growth, the carrying capacity is not fixed but is dynamically determined by the available water $w(x,t)$. Biologically, this vegetation growth formulation captures the resource-dependent growth of biomass, reflecting how vegetation density is regulated by local water availability. Specially, when $w$ is low, the system supports less vegetation, while greater water availability allows for increased vegetation growth.

To capture the effects of the inﬁltration feedback \cite{gilad2007} and soil moisture simultaneously on vegetation, we consider the following specific form for the vegetation mortality rate:
\begin{equation}\label{dr}
\mu(w,b)=\frac{\mu}{1+\theta_1 b+\theta_2 w},
\end{equation}
 where $\mu$ indicates mortality due to human
factors, $\theta_1$ describes the infiltration feedback of vegetation, and $\theta_2$ represents the effect of soil moisture on vegetation. When $\theta_1=\theta_2=0$, the vegetation mortality rate \eqref{dr} is a constant death rate. When $\theta_1>0$ and $\theta_2=0$, the authors in \cite{wang2017interaction} have found the rich dynamics of a simple water-biomass model. When $\theta_1>0, \theta_2>0$,
the vegetation mortality function \eqref{dr} exhibits a monotonically decreasing relationship with respect to both vegetation biomass and water availability. This can be ecologically explained as follows. On one hand, vegetation enhances the surrounding soil environment by promoting water infiltration, thereby increasing water uptake and reducing mortality. On the other hand, higher water availability leads to elevated soil moisture levels, which in turn suppresses vegetation mortality.

It is well known that in many reaction-diffusion models, complex spatiotemporal patterns can emerge through Turing-Hopf (TH) bifurcations \cite{song2019turing,song2020spatio,wu2020spatiotemporal,yi2009bifurcation,xing2024turing,lv2024turing}. A TH bifurcation is a codimension-two bifurcation with the characteristic equation having a pair of simple
purely imaginary roots and a simple zero root. The normal form method is a powerful tool for analyzing the spatiotemporal dynamics of a system near a TH bifurcation point. The main idea is to transform the  partial differential equations at the equilibrium near the TH bifurcation point into a topologically conjugate normal form.
Faria et al. \cite{faria1995normal} computed normal forms for delay functional differential equations, while Jiang et al. \cite{jiang2020formulation} derived a third-order normal form for TH bifurcations in partial functional differential equations using center manifold theory. Song et al. \cite{song2016turing} developed a rigorous procedure for computing normal forms in general reaction-diffusion systems, which was later extended to incorporate diffusion effects \cite{song2017spatiotemporal}.

In this paper, we assume that all parameters in model \eqref{eq1.1}  with $\mu(w,b)$ given by \eqref{dr} are nonnegative. We first consider the stability of positive equilibria and the existence of Hopf bifurcations of model \eqref{eq1.1}.  By taking the diffusion coefficient of water ($d_1$) as the Turing bifurcation parameter, we identify Turing instability in a water-vegetation model with soil moisture. We then investigate the spatiotemporal dynamics near the TH bifurcation. Through numerical simulations, we uncover rich dynamical behaviors in the corresponding ODE system, including forward/backward transcritical bifurcations, saddle-node bifurcations
of equilibria and limit cycles, Hopf bifurcations, limit cycle bubble/heart, homoclinic bifurcation, and bistability phenomena—either between two equilibrium states or between an equilibrium and a limit cycle. The dynamical behavior of the system is simulated across different parameter regions, which has significant implications for ecological sustainability, early warning signals, and ecosystem management.
Furthermore, we observe that as soil moisture ($\theta_2$) decreases, the system may experience multiple transitions: from uniform vegetation to pattern formation, and eventually to bare soil.

The rest of the paper is organized as follows.
In Section \ref{2}, we analyze the existence and stability of nonnegative  equilibria.
In Section \ref{3}, we derive the conditions for Hopf bifurcation, Turing bifurcation, and provide the TH bifurcation point.
In section \ref{4}, we determine the dynamical classification of the system through the normal form of TH bifurcation.
In section \ref{5}, some numerical simulations are performed to verify our theoretical results and specific kinetic classification is
obtained.
Finally, in Section \ref{6}, we give some discussions and conclusions. For simplicity of notations, throughout this paper, we denote the spaces $X=\{(u,v)\in H^{2}(\Omega)\cap H^{2}(\Omega),\frac{\partial u}{\partial \nu}=\frac{\partial v}{\partial \nu}=0, u,v\in \partial \Omega\}$ and the formula  $X_{\mathbb{C}}:=X \bigoplus iX = \left\lbrace x_{1}+ix_{2} | x_{1},x_{2}\in X\right\rbrace$ defines the complexity of any space $X$.

\section{Constant positive equilibria and stability}\label{2}
In this section, we first give some results about the existence and stability of constant nonnegative equilibria and their stability of system \eqref{eq1.1}  with $\mu(w,b)$ given by \eqref{dr}.
\subsection{Positivity and boundedness}
Regardless of spatial diffusion factors, system \eqref{eq1.1}  with $\mu(w,b)$ given by \eqref{dr} reads:
	\begin{equation}\label{eq2.1}
		\begin{cases}
			\displaystyle \frac{dw}{d t}=R-aw-\delta wb =: f(w,b),  \\
			\displaystyle \frac{d b}{d t}=\rho b(1-\frac{b}{w})-\frac{\mu}{1+\theta_1 b+\theta_2 w}b=: g(w,b).
		\end{cases}
	\end{equation}
Obviously, $f(w, b)$ and $g(w, b)$ are analytic functions in the first quadrant $\mathbb{R}_+^2 = \{(w, b) \in \mathbb{R}^2 : w>0, b \ge 0\}$. Assume that all parameters $R, a, \delta, \rho, \mu, \theta_1, \theta_2$ are positive. First, we show the positivity and boundedness of the solutions to (\ref{eq2.1}), which indicates that the model \eqref{eq2.1} is biologically meaningful.

\begin{lemma}\label{lem2.1}
	Any solution of system \eqref{eq2.1} with initial data $w_0 > 0$, $b_0 \ge 0$ exists globally for all $t > 0$, remains nonnegative, and is uniformly bounded. More precisely, if $w_0 > 0, b_0 > 0$, then there exists a positive constant $M_b$ such that for all $t > 0$ any solution $(w(t), b(t))$ to \eqref{eq2.1} satisfies
\begin{equation*}
	\min\big\{w_0,	\displaystyle \frac{R}{a+\delta M_b}\big\} \le w(t) \le \max\big\{w_0,	\displaystyle \frac{R}{a}\big\},\,and \,\, 0 < b(t) \le M_b.
	\end{equation*}
\end{lemma}

\begin{proof}
	We first establish the  non-negativity of solutions. When $b_0 = 0$, the second equation in system \eqref{eq2.1} implies that $b(t) \equiv 0$. In this case, the first equation reduces to a linear ODE:
	$$\frac{dw}{dt} = R - a w,$$
	whose explicit solution is given by
	$$w(t) = \left(w_0 - \frac{R}{a}\right)e^{-a t} + \frac{R}{a},$$
	which remains strictly positive for all $t > 0$.
	When $b_0 > 0$, the positivity of $w(t)$ is ensured by $\lim\limits_{w \to 0^+} \frac{dw}{dt} = R > 0$. Meanwhile, $b(t)$ remains strictly positive for all $t > 0$ since
	$$b(t) = b_0 \exp\left( \int_0^t \left( \rho \left(1 - \frac{b(s)}{w(s)} \right) - \frac{\mu}{1 + \theta_1 b(s) + \theta_2 w(s)} \right) ds \right) > 0.$$
	
	Next, we prove the uniform boundedness of solutions. From the first equation of system \eqref{eq2.1}, we have
	$$\frac{dw}{dt} = R - a w - \delta w b \le R - a w.$$
	By comparison with the auxiliary equation $\frac{dz}{dt} = R - a z$, whose solution is bounded above by $R/a$, it follows that
	$$w(t) \le \max\left\{w_0, \frac{R}{a} \right\} =: M_w.$$
	Using this upper bound in the second equation, we obtain
	$$
	\frac{db}{dt} \le \rho b \left(1 - \frac{b}{w} \right) \le \rho b \left(1 - \frac{b}{M_w} \right),
	$$
	which is a logistic-type inequality. Applying the standard comparison principle yields
	$$
	b(t) \le \max\left\{b_0, M_w \right\} =: M_b.
	$$
	To establish a positive lower bound for $w(t)$, we observe that
	$$
	\frac{dw}{dt} \ge R - (a + \delta M_b) w.
	$$
	which yields
	$$
	w(t) \ge \min\left\{w_0, \frac{R}{a + \delta M_b} \right\} > 0.
	$$
	
	In conclusion, $w(t)$ and $b(t)$ are nonnegative and  uniformly boundedand for all $t > 0$. This guarantees the global existence and boundedness of solutions to system \eqref{eq2.1}.
\end{proof}

\subsection{Existence of equilibria}
Obviously, the system (\ref{eq2.1}) has a bare-soil equilibrium $E_0=(\frac{R}{a},0)$ for all parameters. Before investigating the existence of positive equilibria of system (2.1), we first present the following proposition, which can be directly obtained by Descarte's Rule of signs. As the argument is elementary, we omit the proof here.
\begin{proposition}\label{prop2.2}
	Consider the quartic equation
$$m(x) = x^4 + m_3 x^3 + m_2 x^2 + m_1 x + m_0=0,$$
	where $m_3 > 0 $ and the signs of $ m_2, m_1, m_0 $ are arbitrary. The existence of positive roots of the equation $m(x)=0$ is summarized in the following table.
\begin{table}[htbp]
	\centering
	\caption{The existence of positive roots of the equation $m(x)=0$.}
	\label{tab2.2}
	\begin{adjustbox}{width=0.98\textwidth}
		\small
		\begin{tabular}{lccccccc}
			\toprule
			$ m_2 $ & $ m_1 $ & $ m_0 $ & $ m'(\bar{b}) $ & $ m(b_{11}) $ & $ m(b_{12}) $ & Number of Positive Roots & Illustration \\
			\midrule
			$+$ & $+$ & $-$ &  &  &  & one &Fig. \ref{fig:2.0}(a)A1\\
			$+$ & $-$ & $-$ &  &  &  & one &Fig. \ref{fig:2.0}(a)A2\\
			$-$ & $-$ & $-$ &  &  &  & one &Fig. \ref{fig:2.0}(a)A2\\
			$-$ & $+$ & $-$ & $+$ &  &  & one &Fig. \ref{fig:2.0}(a)A1\\
			$-$ & $+$ & $-$ & $-$ & $-$ &  & one &Fig. \ref{fig:2.0}(a)A3\\
			$-$ & $+$ & $-$ & $-$ & $+$ & $+$ & one &Fig. \ref{fig:2.0}(a)A4\\
			\midrule
			$+$ & $-$ & $+$ &  &  & $-$ & two &Fig. \ref{fig:2.0}(a)A5\\
			$-$ & $-$ & $+$ &  &  & $-$ & two &Fig. \ref{fig:2.0}(a)A5\\
			$-$ & $+$ & $+$ & $-$ &  & $-$ & two &Fig. \ref{fig:2.0}(a)A6\\
			\midrule
			$-$ & $+$ & $-$ & $-$ & $+$ & $-$ & three &Fig. \ref{fig:2.0}(a)A7\\
			\bottomrule
		\end{tabular}
	\end{adjustbox}
	
	\begin{minipage}{0.98\textwidth}
		\small
		\textbf{Note:} Here, $ b_{11} $ and $ b_{12} $ denote the two positive roots (if they exist) of $ m'(x) = 0$, with $ b_{11} < b_{12} $, and $ \bar{b} $ is the unique positive root of $ m''(x) = 0 $ when $ m_2 < 0 $.
	\end{minipage}
\end{table}
\end{proposition}



\begin{figure}[!h]
	\centering
	\subfigure[]{\includegraphics[width=0.475\textwidth]{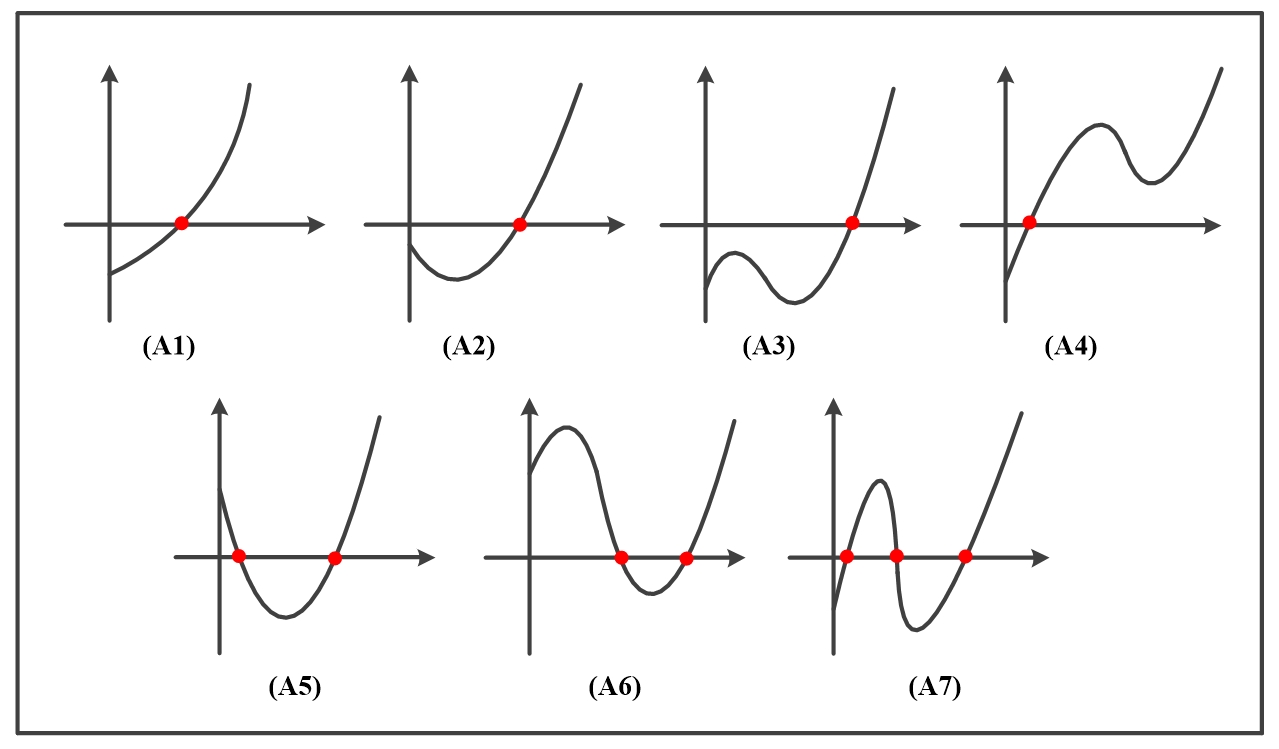}} \hspace{5pt}
	\subfigure[]{\includegraphics[width=0.35\textwidth]{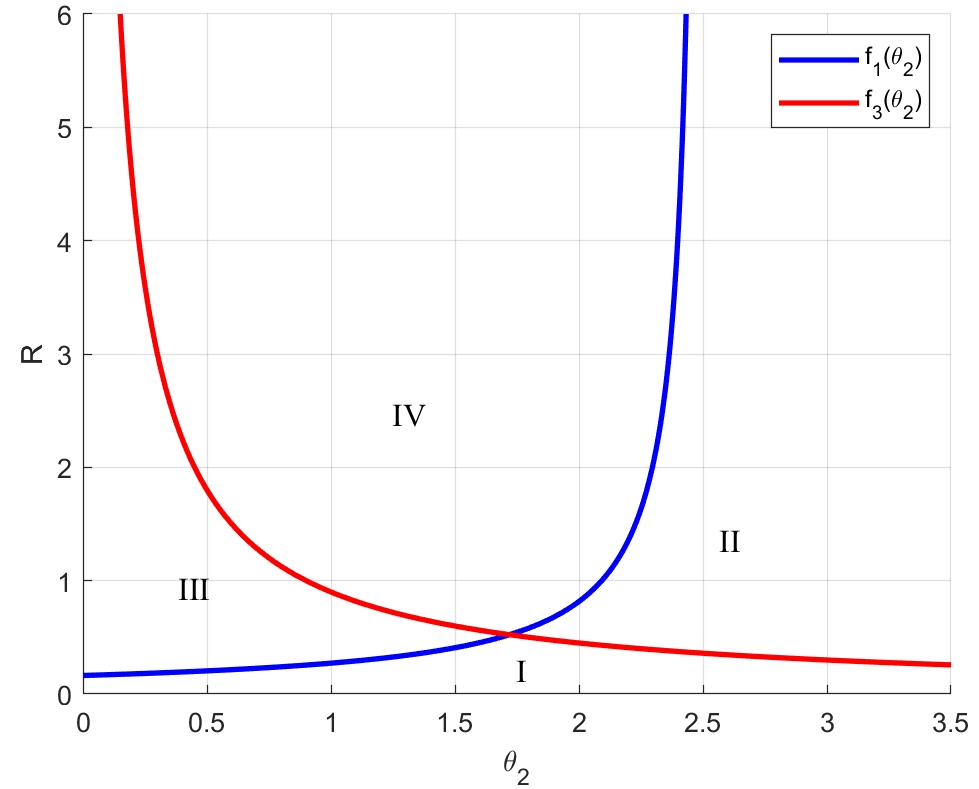}}
	\caption{(a) Possible shapes of the graph of $m(x)$. (b) Illustration of the parameter subregions in the $\theta_2-R$
		plane with $a=0.1,\,\, \delta=0.12,\,\, \rho=1,\,\, \mu=10, \theta_1=2.5.$}
	\label{fig:2.0}
\end{figure}

 Now we discuss the existence of positive equilibria of system (\ref{eq2.1}). A positive equilibrium $E^{*}=(w^{*},b^{*})$ satisfies $w^{*} = \frac{R}{a + \delta b^{*}}$ and $h(b^*)=0$, where
\begin{equation}\label{eq2.3}
			\begin{aligned}
	 h(b) &:= \frac{\rho}{R} \left( R - ab - \delta b^{2} \right) - \frac{\mu (a + \delta b)}{\left( 1 + \theta_1 b \right) \left( a + \delta b \right) + \theta_2 R}\\
&=-\frac{\rho \delta^2 \theta_1}{R[(1+\theta_1 b)(a+\delta b)+\theta_2 R]}F(b).
\end{aligned}
\end{equation}
Here,
\begin{center}
	$F(b) = b^4+C_3b^3+C_2b^2+C_1b+C_0$,
\end{center}
with
\begin{center}
	$\begin{aligned}
C_0& = \frac{\mu aR - R^2 \rho \theta_2 - Ra\rho}{\rho \delta^2 \theta_1},\quad\quad\quad\quad
C_1 = \frac{\rho (R\theta_2 a - R\theta_1 a + a^2 - R\delta) + \mu R \delta}{\rho \delta^2 \theta_1}, \quad \\
C_2 &= \frac{R\delta \theta_2 - R\delta \theta_1 + 2a\delta + \theta_1 a^2}{\delta^2 \theta_1}, \quad\ \ 		
		C_3 = \frac{2a\theta_1 + \delta}{\delta \theta_1}> 0.
	\end{aligned}$
\end{center}
Setting $ C_2 =0, C_1 =0, C_0 = 0 $, respectively,  we obtain
\begin{center}
	$\begin{aligned}
		R&=\frac{2a\delta+\theta_1 a^2}{\delta(\theta_1-\theta_2)}=:f_1(\theta_2),\\
		R&=\frac{\rho a^2}{\rho \theta_1 a+\rho \delta-\rho\theta_2 a-\mu \delta}=:f_2(\theta_2),\\
		R&=\frac{a(\mu - \rho)}{\rho  \theta_2}=:f_3(\theta_2).
	\end{aligned}$
\end{center}

With parameters $a, \delta, \rho, \mu, \theta_1$ fixed, we plot the curves $f_1(\theta_2), f_2(\theta_2), f_3(\theta_2)$ in the $\theta_2-R$ plane. Assume that $\mu>\rho$. In a particular case, the curves $f_1(\theta_2)$ and $f_3(\theta_2)$ divide the first quadrant of the parameter plane into four subregions (see Fig. \ref{fig:2.0}(b)), as the curve $f_2(\theta_2)$ is not in the the first quadrant. We denote these subregions as I, II, III, and IV, respectively, and define them as follows:

\begin{center}
	$\begin{aligned}
		\mathrm{I}=&\left\{(\theta_2,R):0<R<\min\left\{f_1(\theta_2),f_3(\theta_2)\right\}\right\},\\
		\mathrm{II}=&\left\{(\theta_2,R):0<\theta_2<\theta_1,f_3(\theta_2)<R<f_1(\theta_2)\right\}\cup\left\{(\theta_2,R):\theta_2>\theta_1,R>f_3(\theta_2)\right\},\\
		\mathrm{III}=&\left\{(\theta_2,R):0<\theta_2<\theta_1,f_1(\theta_2)<R<f_3(\theta_2)\right\},\\
		\mathrm{IV}=&\left\{(\theta_2,R):0<\theta_2<\theta_1,R>\max\left\{f_1(\theta_2),f_3(\theta_2)\right\}\right\}.
	\end{aligned}$
\end{center}

The existence of positive equilibria in each of these regions is as follows.
\begin{theorem}\label{thm2.2}
	With parameters $a, \delta, \rho, \mu, \theta_1$ fixed, let $f_1(\theta_2)$, $f_2(\theta_2)$, $f_3(\theta_2)$, and the subregions $\mathrm{I}-\mathrm{IV}$ be defined as above. Then the existence of positive equilibria for system~\eqref{eq2.1} is characterized as follows:
\begin{description}
	\item[(i)] For $(\theta_2, R) \in \mathrm{I} $, there exists no positive equilibria.
	\item[(ii)] For $(\theta_2, R) \in \mathrm{II}$, there exists a unique  positive equilibrium
	$	E_1^{*} = \left( \frac{R}{a + \delta b_1^{*}},\ b_1^{*} \right),$ where $b_1^{*}$ is the unique positive root of $F(b)$.
	\item[(iii)] For $(\theta_2, R) \in \mathrm{III}$, if $F'(\bar{b}) < 0$ and $F(b_{12}) < 0$, then there exist two  positive equilibria $E_{2j}^{*} = \left( \frac{R}{a + \delta b_{2j}^{*}},\ b_{2j}^{*} \right),  j = 1,2,$ where $0 < b_{21}^{*} < b_{22}^{*}$ are consecutive positive roots of $F(b)$.
	  \item[(iv)] For $(\theta_2, R) \in \mathrm{IV}$, the following two cases may occur:
	  \begin{enumerate}
			\item[(a)] A unique positive equilibrium $E_1^{*}$ exists if either
			  $F'(\bar{b}) > 0$, or $\left \{\begin{aligned}
				&F'(\bar{b}) < 0,\\
				&F(b_{11}) < 0,\end{aligned}\right.$ or
		$\left \{\begin{aligned}
				&F'(\bar{b}) < 0,\\
				&F(b_{11}) > 0,\\
				&F(b_{12}) > 0,\\
			\end{aligned}\right.$ holds.
			    \item[(b)] Three positive equilibria $E_{3j}^{*} = \left( \frac{R}{a + \delta b_{3j}^{*}},\ b_{3j}^{*} \right) ( j = 1,2,3)$
				exist if $\left \{\begin{aligned}
					&F'(\bar{b}) < 0,\\
					&F(b_{11}) > 0,\\
					&F(b_{12})< 0,\\
				\end{aligned}\right.$
					where $0 < b_{31}^{*} < b_{32}^{*} < b_{33}^{*}$ are the ordered positive roots of $F(b)$.
				  \end{enumerate}	
\end{description}
Here, $ b_{11}, b_{12}  ( b_{11} < b_{12}) $denote the two positive roots (if they exist) of $ F'(b) = 0$, and $ \bar{b} $ is the unique positive root of $ F''(b) = 0 $ when $ C_2 < 0 $.
\end{theorem}

For other cases, the existence of positive equilibria of system \eqref{eq2.1} can be analyzed similarly and we omit it here.

\subsection{Local stability of equilibria}

From Theorem \ref{thm2.2}, the system (\ref{eq2.1}) may have at most four non-negative equilibria. In this subsection we investigate the local stability of these equilibria. First we have the following results for  $E_0=(\frac{R}{a},0)$.

\begin{theorem}\label{thm2.3}
	Assume that all parameters are positive and define
	\begin{equation}\label{eq2.5}
		R^{*}:=\frac{a(\mu-\rho)}{\rho \theta_2}.
	\end{equation}
 If $R<R^{*}$,	then the bare-soil equilibrium $E_0=(\frac{R}{a},0)$ is locally asymptotically stable, and  unstable if $R>R^{*}$.
\end{theorem}
\begin{proof}
	The Jacobian matrix of system (\ref{eq2.1}) at the bare-soil equilibrium $E_0 = \left(\frac{R}{a}, 0\right)$ is
$J\left(\frac{R}{a}, 0\right) =
	\begin{pmatrix}
		-a &  -\frac{\delta R}{a} \\
		0 &  \rho - \frac{\mu a}{a + \theta_2 R}
	\end{pmatrix}$,
whose eigenvalues are $\lambda_1=-a$ and $\lambda_2=\rho - \frac{\mu a}{a + \theta_2 R}$. Thus $E_0$ is locally asymptotically stable if $R<R^{*}$ and unstable if $R>R^{*}$.
\end{proof}

In the following, we discuss the local stability of the positive equilibria obtained in Theorem \ref{thm2.2}.
The following signs of $F'(b)$ at the simple zero may need
\begin{equation}\label{eq2.4}
	 F'(b_1^{*})>0, F'(b_{21}^{*})<0, F'(b_{22}^{*})>0,  F'(b_{31}^{*})>0, F'(b_{32}^{*})<0, F'(b_{33}^{*})>0,
\end{equation}
which can be easily obtained from the characteristics of the function $F(b)$ illustrated in Fig. \ref{fig:2.0}(a).

Recall the function $h(b)$ defined in (\ref{eq2.3}), then we have
$$h'(b^{*}) = \frac{\mu \theta_1(a+\delta b^{*})^2-\mu \delta \theta_2 R}{[(1+\theta_1 b^{*})(a+\delta b^{*})+\theta_2 R]^2}-\frac{\rho}{R}(a+2\delta b^{*}),$$
 which yields
$$h'(b^{*}) = -\frac{\rho \delta^2 \theta_1 }{R[(1+\theta_1 b^{*})(a+\delta b^{*})+\theta_2 R]}F'(b^{*}).
$$

The associated Jacobian matrix at the equilibrium $E^{*}=(w^{*},b^{*})$ is given by
\begin{equation}\label{eq2.7}
J\left(w^{*}, b^{*}\right) =
\begin{pmatrix}
	a_{11} & a_{12} \\
	a_{21} & a_{22}
\end{pmatrix},
\end{equation}
where
\begin{equation}\label{eq2.8}
	\begin{aligned}
		a_{11} & = -a - \delta b^{*}<0, & a_{21} & = \frac{\rho b^{*2}}{w^{*2}} + \frac{\mu \theta_2 b^{*}}{(1+\theta_1 b^{*} + \theta_2 w^{*})^2}>0, \\
		a_{12} & = -\delta w^{*}<0, & a_{22} & = \frac{\mu \theta_1 b^{*}}{(1+\theta_1 b^{*} + \theta_2 w^{*})^2} - \frac{\rho b^{*}}{w^{*}}.
	\end{aligned}
\end{equation}
Then we have
\begin{center}
	$\begin{aligned}
	\text{Det}(J(E^{*})) = & \, \delta w^{*} \left[ \frac{\rho b^{*2}}{w^{*2}} + \frac{\mu \theta_2 b^{*}}{(1+\theta_1 b^{*} + \theta_2 w^{*})^2} \right]  - (a + \delta b^{*}) \left[ \frac{\mu \theta_1 b^{*}}{(1+\theta_1 b^{*} + \theta_2 w^{*})^2} - \frac{\rho b^{*}}{w^{*}} \right] \\
		= & \, b^{*}(a + \delta b^{*}) \left\{\frac{\mu \delta \theta_2 R - \mu \theta_1 (a + \delta b^{*})^2}{[(1+\theta_1 b^{*})(a+\delta b^{*}) + \theta_2 R]^2} + \frac{\rho (a + 2\delta b^{*})}{R}\right\}\\
		= & \, -b^{*}(a + \delta b^{*}) h'(b^{*}) \\
		= & \, \frac{\rho \delta^2 \theta_1  b^{*}(a + \delta b^{*})}{R[(1+\theta_1 b^{*})(a+\delta b^{*})+\theta_2 R]}F'(b^{*}).
	\end{aligned}$
\end{center}
This shows that $sign\,\,\text{Det}(J(E^{*})) = -sign\,\,h'(b^{*}) = sign\,\, F'(b^{*})$. Then we get from (\ref{eq2.4}) that
\begin{center}
	$\begin{aligned}
		\text{Det}(J(E_{\,1}^{*})) & > 0, \quad
		\text{Det}(J(E_{21}^{*})) & < 0, \quad
		\text{Det}(J(E_{22}^{*})) & > 0, \\
		\text{Det}(J(E_{31}^{*})) & > 0, \quad
		\text{Det}(J(E_{32}^{*})) & < 0, \quad
		\text{Det}(J(E_{33}^{*})) & > 0.
	\end{aligned}$
\end{center}
Therefore,  when the equilibria $E_{21}^{*}$ and $E_{32}^{*}$ exist, they are saddles and unstable, while the other four equilibria are locally asymptotically stable if $\text{Tr}(J(E^{*}))<0$ for the corresponding equilibrium.

Next, we discuss the sign of  $\text{Tr}(J(E^{*}))$. Since $w^{*}$ and $b^{*}$ satisfy $\rho (1-\frac{b^{*}}{w^{*}})-\frac{\mu}{1+\theta_1 b^{*}+\theta_2 w^{*}}=0$, we have
\begin{center}
	$\begin{aligned}
		\text{Tr}(J(E^{*})) &= - a - \delta b^{*} + \frac{\mu \theta_1 b^{*}}{(1+\theta_1 b^{*} + \theta_2 w^{*})^2} - \frac{\rho b^{*}}{w^{*}} \\
		&= -\frac{a+\delta b^{*}}{R[(1+\theta_1 b^{*})(a+\delta b^{*}) + \theta_2 R]}
		\left\{ (R+\rho b^{*})[(1+\theta_1 b^{*})(a+\delta b^{*}) + \theta_2 R] \right. \\
		& \qquad \left. - \frac{\mu R \theta_1 b^{*}(a+\delta b^{*})}{(1+\theta_1 b^{*})(a+\delta b^{*}) + \theta_2 R} \right\} \\
		&= -\frac{a+\delta b^{*}}{R[(1+\theta_1 b^{*})(a+\delta b^{*}) + \theta_2 R]}
		\left\{ (R+\rho b^{*})[(1+\theta_1 b^{*})(a+\delta b^{*}) + \theta_2 R] \right. \\
		& \qquad \left. + \rho \theta_1 b^{*} (\delta b^{*2} + ab^{*} - R) \right\} \\
		& =:-P(R,b^{*})T(R,b^{*}),
	\end{aligned}$
\end{center}
where $P(R,b^{*})=\frac{a+\delta b^{*}}{R[(1+\theta_1 b^{*})(a+\delta b^{*}) + \theta_2 R]}>0$ and $T(R,b^{*})=\theta_2 R^2+K_1 R+K_0$  with $K_1=\delta \theta_1 b^{*2}+(a \theta_1+\delta+\rho \theta_2-\rho \theta_1)b^{*}+a, K_0=\rho b^{*}(a+\delta b^{*})(1+2\theta_1 b^{*})>0.$  It is easy to see that $Tr(J(E^*)) < 0$ if and only if one of the
following conditions holds:
\begin{center}
	$\begin{aligned}
		&\mathbf{(H1)}  \ \ \ \ K_1 \ge 0;\\
		&\mathbf{(H2)}  \ \ \ \ K_1 <0 \ \ \text{and} \ \  T_{min}=T(-\frac{K_1}{2\theta_2},b^*)>0;\\
		&\mathbf{(H3)}  \ \ \ \ (1)\ K_1 <0\ \text{and} \  T_{min}=T(-\frac{K_1}{2\theta_2},b^*)\le 0 \ \text{and}\ (2) \  0<R<R_1 \ \text{or} \ R>R_2,
	\end{aligned}$
\end{center}
where
\begin{equation}\label{eq2.9}
	R_{1,2}=\frac{-K_1\pm \sqrt{K_1^2-4 \theta_2 K_0}}{2 \theta_2}
\end{equation} are the two positive roots of $T(R,b^*)$.
%
%
%

In summary, we have derived the following results regarding the local stability of the positive equilibria of system (\ref{eq2.1}).

\begin{theorem}\label{thm2.5}
	Assume that all parameters are positive. Then  the following results holds.
\begin{description}
  \item[(i)] If exists, any of the positive equilibria $E_{21}^{*}$ and $E_{32}^{*}$ of (\ref{eq2.1}) is a saddle point, and thus unstable.
  \item[(ii)] If exists, any of the positive equilibria $E_{1}^{*}$, $E_{22}^{*}$, $E_{31}^{*}$ and $E_{33}^{*}$ is locally asymptotically stable, provided that one of the conditions $\mathbf{(H1)}$, $\mathbf{(H2)}$, or $\mathbf{(H3)}$ is satisfied. Here, $b^{*}$ corresponds to  $b_{1}^{*}, \, b_{22}^{*},\,  b_{31}^{*},$ or $\,b_{33}^{*}.$
\end{description}
\end{theorem}

\section{Bifurcation analysis}\label{3}
\subsection{Hopf bifurcations}
Hopf bifurcation describes the process in which a stable equilibrium loses its stability and gives rise to a periodic solution as a system parameter varies. It is a fundamental mechanism behind the emergence of oscillatory behavior in dynamical systems. To investigate the conditions under which system (\ref{eq2.1}) exhibits such periodic dynamics, we analyze the occurrence of Hopf bifurcations below.

From the above analysis, we know that there exist values of $R$ near $R_i (i=1, 2)$ such that the corresponding characteristic matrix has a pair of conjugate complex roots, defined by  $\sigma(R)=\alpha(R)\pm i \beta(R)$, which
satisfies $$\alpha(R_i)=0,\quad \beta(R_i)=\sqrt{Det(J(E^{*})},$$
 where $R_i (i=1, 2)$ are defined in \eqref{eq2.9} and $E^{*}$ may be $E^{*}_1, E^{*}_{22}, E^{*}_{31}$, or $E^{*}_{33}$.
Note that
\begin{equation}\label{h}
H(R,b^*):=h(b^{*}) = \frac{\rho}{R} \left( R - ab^{*} - \delta b^{*2} \right) - \frac{\mu (a + \delta b^{*})}{\left( 1 + \theta_1 b^{*} \right) \left( a + \delta b^{*} \right) + \theta_2 R}=0,
\end{equation}
where $b^{*}=b^{*}(R)$. Differentiating \eqref{h} with respect to $R$, we have
$$
 \frac{d b^{*}(R)}{d R}=-\frac{\partial H}{\partial R}\,\big/\,\frac{\partial H}{\partial b^{*}}=-\frac{\partial H}{\partial R}\,\big/\,h'(b^{*}).
$$
On the other hand, a direct calculation shows that
$$\frac{\partial H}{\partial R}=\frac{\rho b^{*}(a+\delta b^{*})}{R^2}+\frac{\mu \theta_2(a+\delta b^{*})}{[(1+\theta_1 b^{*})(a+\delta b^{*})+\theta_2 R]^2}>0.
$$
This combining with $h'(b^{*})<0$ implies that
$$\frac{d b^{*}(R)}{d R}>0.$$

Let $b_1=b^{*}(R_1)$ and $b_2=b^{*}(R_2)$. Then we conclude that $b_1<b_2$. According to the expression of  $T(R,b^{*})$, we obtain
$$
\frac{\partial T(R,b^{*})}{\partial R}\bigg|_{R=R_1} < 0, \,\,
\frac{\partial T(R,b^{*})}{\partial R}\bigg|_{R=R_2} > 0, \,\,
\frac{\partial T(R,b^{*})}{\partial b}\bigg|_{b=b_1} < 0, \,\,
\frac{\partial T(R,b^{*})}{\partial b}\bigg|_{b=b_2} > 0.
$$
Therefore, we can derive the transversality condition
\begin{align*}
	\alpha '(R_i) &= \frac{1}{2} \frac{\partial Tr(J(E^{*}))}{\partial R} \bigg|_{R=R_i} \\
	&= -\frac{1}{2} P(R, b^{*}) \left( \frac{\partial T(R, b^{*})}{\partial R} + \frac{\partial T(R, b^{*})}{\partial b} \, \frac{d b^{*}(R)}{d R} \right) \bigg|_{R=R_i} \neq 0, \quad i=1,2.
\end{align*}
Specifically, $\alpha '(R_1)>0$ and $\alpha '(R_2)<0$. Consequently, we conclude that a Hopf bifurcation occurs at $R=R_i, i = 1, 2.$
\begin{theorem}\label{thm3.1}
	Assume that $R=R_i, i=1,2$  are defined  as in (\ref{eq2.9}). Then system (\ref{eq2.1}) undergoes  Hopf bifurcations at $R=R_1$ and $R=R_2$ provided these values exist.
\end{theorem}

\subsection{Turing instability induced by diffusion and Turing$-$Hopf bifurcation}
In 1952, Turing \cite{Turing} demonstrated that a system of coupled reaction-diffusion (R-D) equations could be used to describe pattern formation in biological systems. Turing's theory shows that the interplay between chemical reactions and diffusion may cause a locally stable equilibrium to become unstable in the presence of diffusion, leading to the spontaneous formation of a spatially periodic stationary structure. This kind of instability is called Turing instability or diffusion-driven instability.

In this part, we proceed with the analysis of the diffusive water-vegetation model (\ref{eq1.1}) with \( \mu(w,b) \) given by (\ref{dr}). Our goal is to determine whether the spatially uniform steady state becomes unstable under diffusion. Specifically, we examine the conditions for Turing instability in the one-dimensional spatial domain \( \Omega = (0, l \pi) \).


Let $(w^{*},b^{*})$ denote a stationary homogeneous solution of system (\ref{eq1.1}). The linearized operator of system  (\ref{eq1.1}) at $(w^{*},b^{*})$ is given by
\begin{center}
	$L =
	\begin{pmatrix}
		 d_1\frac{\partial}{\partial x^2}+a_{11} & a_{12} \\
		a_{21} &  d_2\frac{\partial}{\partial x^2}+a_{22}
	\end{pmatrix},$
\end{center}
where $a_{ij}$ are defined as in (\ref{eq2.8}).
It is well known that the eigenvalue problem
$$
-\Phi''=\rho \,\Phi,\,x \in (0,l \pi),\,\Phi'(0)=\Phi'(l \pi)=0,
$$
has eigenvalues $\rho_k={k^2}/{l^2},\,k=0,\,1,\,2,$..., with corresponding eigenfunctions $\Phi_k(x)=\text{cos}(\frac{k}{l} x) $. Let
\begin{center}
	$\begin{pmatrix}
		 \phi \\ \psi
	\end{pmatrix} =  \sum\limits_{k=0}^{\infty}
	\begin{pmatrix}
		 \phi_k \\ \psi_k
	\end{pmatrix} \text{cos}\frac{k}{l} x$
\end{center}
denote an eigenfunction of $L$ corresponding to the eigenvalue $\lambda$, i.e., $L(\phi, \psi)^T = \lambda (\phi, \psi)^T$. Then a straightforward analysis  shows that $L_k(\phi_k, \psi_k)^T = \lambda (\phi_k, \psi_k)^T$ for $k \in \mathbb{N}_0$, where $L_k$ is defined as
\begin{center}
	$L_k =
	\begin{pmatrix}
		 -d_1\frac{k^2}{l^2}+a_{11} & a_{12} \\
		a_{21} &  -d_2\frac{k^2}{l^2}+a_{22}
	\end{pmatrix}.$
\end{center}
The characteristic equation of $L_k$ in $\lambda$ is
\begin{equation}\label{eq3.1}
	\Delta_{k}(\lambda)=\lambda^2-T_k \,\lambda +J_k=0,\quad k=0,\,1,\,2,\cdots\ ,
\end{equation}
where
\begin{equation}\label{eq3.2}
	\begin{aligned}
		T_k &= -(d_1+d_2)\frac{k^2}{l^2} + T_0 = -(d_1+d_2)\frac{k^2}{l^2} + \text{Tr}(J(E^{*})), \\
		J_k &= d_1 d_2 \frac{k^4}{l^4} - (d_1 a_{22} + d_2 a_{11}) \frac{k^2}{l^2} + J_0 \\
		&= d_1 d_2 \frac{k^4}{l^4} - (d_1 a_{22} + d_2 a_{11}) \frac{k^2}{l^2} + \text{Det}(J(E^{*})).
	\end{aligned}
\end{equation}

From Theorem \ref{thm2.5}, we can conclude that in the absence of diffusion (i.e., $d_1=d_2=0$), when any of the conditions $\mathbf{(H1)}, \mathbf{(H2)}, \mathbf{(H3)}$ is satisfied, $E^{*}$ is stable. Here, $E^{*}$ may be $E_{1}^{*}$, $E_{22}^{*}$, $E_{31}^{*}$, or $E_{33}^{*}$. This implies that
\begin{equation}\label{eq3.3}
	T_0=\text{Tr}(J(E^{*}))<0,\quad J_0=\text{Det}(J(E^{*}))>0.
\end{equation}
Therefore,  for each positive integer $k$ , we have $T_k<0$. To make $E^{*}$ unstable, we must find at least one positive integer $k$ such that $J_k < 0$, in which case Turing instability may occur.

Based on the above analysis, for the positive steady state $E^{*}$, the key to Turing instability in system (\ref{eq1.1}) is to find the critical wave number $k_{*}$ such that $J_{k_{*}} = 0$, while $J_k > 0$ for $k \neq k_{*}$.
\begin{theorem}\label{thm3.2}
	Assume that one of the conditions $\mathbf{(H1)}$, $\mathbf{(H2)}$, $\mathbf{(H3)}$ is satisfied. Then there is no diffusion-driven Turing instability if one of the following conditions holds:\\
$\mathbf{(H4)}$\ \ \ \ \ \ \ \ \ \ \ \ \ \ \ \ \ \ \ \ \ \ \ \ \ \ \ \ \ \ \ \   $ a_{22} \le 0$;\\
$\mathbf{(H5)}$\ \ \ \ \ \ \ \ \ \ \ \ \ \ \ \ \ \ \ \ \ \ \ \ \ \ \ \ \ \ \ \   $a_{22} > 0 \,\, \text{and} \,\, d_1 \le -\frac{d_2 a_{11}}{a_{22}}. $
\end{theorem}
\begin{proof}
	It is straightforward to verify from (\ref{eq3.2}) that
	$d_1a_{22}+d_2a_{11} \le 0$ if and only if  $\mathbf{(H4)}$ or $\mathbf{(H5)}$ holds. Moreover, since
	$J_0 > 0$, it follows that if either $ \mathbf{(H4)}$ or $\mathbf{(H5)}$ is satisfied, then
	$J_k > 0$ for all positive integers $k$.
\end{proof}
\begin{theorem}\label{thm3.3}
	Assume that \\
$\mathbf{(H6)}$\ \ \ \ \ \ \ \ \ \ \ \ \ \ \ \ \ \ \ \ \ \ \ \ \ \ \ \ \ \ \ \   $a_{22} > 0 \,\, \text{and} \,\, d_1 > -\frac{d_2 a_{11}}{a_{22}}$\\
is satisfied. Then for the diffusive water-vegetation model (\ref{eq1.1}) with $\mu(w,b)$ given by \eqref{dr}, we have the following results.
	\begin{description}
   \item[(i)]  When either  $\mathbf{(H1)}$ or $\mathbf{(H2)}$ is satisfied, the boundary of the stability region in the $ R - d_1 $ plane consists of the critical Turing bifurcation curve $ d_1 = d_1(R, k_{*}^{2}) $, and no TH bifurcation occurs in this case.
   \item[(ii)] When $\mathbf{(H3)}(1)$ is satisfied, the boundary of the stability region in the $ R - d_1 $ plane consists of Hopf bifurcation curves $ H_j $, which are determined by $ R = R_i $, $ i = 1, 2 $, and the critical Turing bifurcation curve $ d_1 = d_1(R, k_{*}^{2}) $. In this scenario, a $(k_*, 0)-$mode TH bifurcation can occur at the interaction points $ (R_i, d_1^{*}) $, where $ d_1^{*} = d_1(R_i, k_{*}^{2}) $. The function $ d_1(R_i, k_{*}^{2}) $ is determined by the following equation (\ref{eq3.4}), and $ k_* $ is determined by the following equation (\ref{eq3.6}).
 \end{description}
\end{theorem}
\begin{proof}
We begin by analyzing the conditions for Turing bifurcation. For $ E^{*} = (w^{*}, b^{*}) $, from $ J_k = 0 $, we obtain the following expression for $ d_1 $:
\begin{equation}\label{eq3.4}
	d_1 = d_1(R, k^2):=\frac{l^2(B_1 k^2 + B_2)}{k^2 (d_2 k^2 - B_3)}, \,\, \text{for} \,\, d_2 k^2 \neq B_3,
\end{equation}
where $ B_1 = d_2 a_{11} < 0 $, $ B_2 = (a_{12} a_{21} - a_{11} a_{22}) l^2 < 0 $, and $ B_3 = a_{22} l^2 > 0 $.

Let $ z = k^2 $, then
$
d_1(R, z) = \frac{l^2(B_1 z + B_2)}{z(d_2 z - B_3)}.
$
Next, by differentiating $ d_1(R, z) $ with respect to $ z $, we obtain:
$$
\frac{\partial d_1}{\partial z}(R, z) = l^2 \frac{-d_2 B_1 z^2 - 2d_2 B_2 z + B_2 B_3}{z^2 (d_2 z - B_3)^2}.
$$

To analyze the sign of $ J_k $ and determine the range of the wave number $ k $, we consider the following two cases.

\textbf{Case I}: $ z > \frac{B_3}{d_2} $.

In this case, the behavior of $ J_k $ is as follows:
$$
J_k
\begin{cases}
	< 0, & \quad d_1 < d_1(R, k^2), \\
	= 0, & \quad d_1 = d_1(R, k^2), \\
	> 0, & \quad d_1 > d_1(R, k^2).
\end{cases}
$$
Additionally, since
$$
-d_2 B_1 z^2 - 2d_2 B_2 z + B_2 B_3 > -\frac{B_1 B_3^2}{d_2} - B_2 B_3 > 0,
$$
it follows that
$
\frac{\partial d_1}{\partial z}(R, z) > 0,
$
which implies that the function $ d_1(R, z) $ is monotonically increasing for $ z > \frac{B_3}{d_2} $. Furthermore, we know that $ d_1(R, z) < 0 $ and
$
\lim\limits_{z \to +\infty} d_1(R, z) = 0.
$
Thus, for any $ R > 0 $ and $ d_1 > 0 $, we conclude that $ J_k > 0 $. This means that system (\ref{eq1.1}) does not branch from $ E^{*} $ when the wave number $ k^2 > \frac{B_3}{d_2} $.

\textbf{Case II}: $ z < \frac{B_3}{d_2} $.

In this case, the behavior of $ J_k $ is as follows:
$$
J_k
\begin{cases}
	> 0, & \quad d_1 < d_1(R, k^2), \\
	= 0, & \quad d_1 = d_1(R, k^2), \\
	< 0, & \quad d_1 > d_1(R, k^2).
\end{cases}
$$
By direct calculation, we have
$$
\frac{\partial d_1}{\partial z}(R, z)
\begin{cases}
	< 0, & \quad 0 < z < z_{*}, \\
	> 0, & \quad z > z_{*},
\end{cases}
$$
where
\begin{equation}\label{eq3.5}
	\begin{aligned}
		z_{*} &= \frac{-d_2 B_2 - \sqrt{d_2^2 B_2^2 + d_2 B_1 B_2 B_3}}{d_2 B_1} \\
		&= l^2 \frac{a_{11} a_{22} - a_{12} a_{21} - \sqrt{-a_{12} a_{21} (a_{11} a_{22} - a_{12} a_{21})}}{d_2 a_{11}}.
	\end{aligned}
\end{equation}
Thus, $ d_1(R, z) $ attains its minimum at $ z = z_{*} $ for fixed $ R $. Additionally, we have $ d_1(R, z) > 0 $, and
$$
\lim_{z \to 0^{+}} d_1(R, z) = \lim_{z \to \left( \frac{B_3}{d_2} \right)^{-}} d_1(R, z) = +\infty.
$$
Refer to Fig. \ref{fig0}, which shows the relationship between $ d_1(R, k^2) $ and $ k^2 $. In summary, the range of wave numbers for which Turing instability occurs is $ k^2 < \frac{B_3}{d_2} $. When $ d_1 > d_1(R, k^2) $, we have $ J_k < 0 $. Therefore, we aim to find the wave number that minimizes $ d_1(R, k^2) $. Since the wave number $ k $ takes positive integer values, $ d_1(R, k^2) $ achieves its minimum at $ k = k_{*} $. The value of $ k_{*} $ is given by
\begin{equation}\label{eq3.6}
	k_{*} =
	\begin{cases}
		k_0, & \quad d_1(R, k_0^2) \le d_1(R, (k_0 + 1)^2), \\
		k_0 + 1, & \quad d_1(R, k_0^2) > d_1(R, (k_0 + 1)^2).
	\end{cases}
\end{equation}
Here, $ k_0 = \left[ \sqrt{z_{*}} \right] $, where $ \left[\cdot\right] $ denotes the integer part function.
\begin{figure}[!t]
	\centering
	\includegraphics[width=6cm,height=5.5cm]{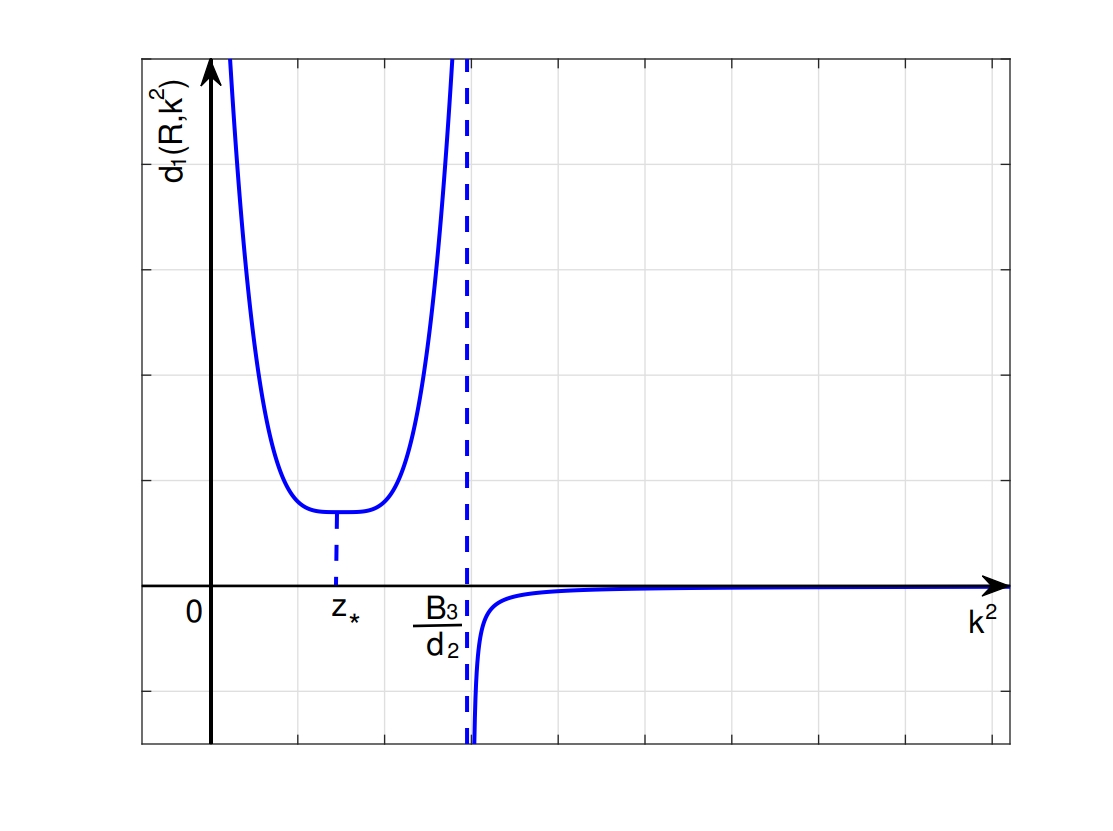}
	\caption{The schematic of $d_1(R,k^2)$ with respect to $k^2$.}
	\label{fig0}
\end{figure}

Additionally, based on Theorem \ref{thm2.5}, when either condition $\mathbf{(H1)}$ or $\mathbf{(H2)}$ is satisfied, there is no Hopf bifurcation for the diffusive water-vegetation model (\ref{eq1.1}). This, along with the above discussions
on Turing bifurcation, confirms conclusion $\mathbf{(i)}$.

When condition $\mathbf{(H3)}(1)$ is satisfied,, we also find that $T_0 > 0$ for $R_1<R<R_2$ and $T_0 < 0$ for $R<R_1$ or $R>R_2$. Therefore, there is no Hopf bifurcation for $R<R_1$ or $R>R_2$ in the $R-d_1$ plane. Consequently, $R = R_1$ and $R = R_2$ are two critical Hopf bifurcation curves. This, along with the above discussions on Turing bifurcation and Theorem \ref{thm3.1}, completes the proof of conclusion $\mathbf{(ii)}$.
\end{proof}

\section{Normal forms for Turing$-$Hopf bifurcation}\label{4}
What kinds of dynamical behavior will the system present near the TH points $(R_{i}, d_1^{*}) (i=1,2)$? In this section, we aim to determine the dynamical classification of system (\ref{eq1.1}) near these bifurcation points. We will derive the normal forms of the TH bifurcation for system (\ref{eq1.1}) at the positive equilibrium $E^{*}=(w^{*},b^{*})$. It is important to note that both $w^{*}$ and $b^{*}$ depend on $R$, and $E^{*}$ will change with the small perturbation of $R$, which increases the difficulty of calculating the normal form.

First, we make the change of variables  to translate $(w^{*},b^{*})$ into the origin by $\overline{w}=w-w^{*}, \overline{b}=b-b^{*}$
and drop the bars. Then, we can rewrite (\ref{eq1.1}) as follows:
	 \begin{equation}\label{eq4.1}
	\begin{cases}
		\displaystyle \frac{\partial w}{\partial t}=d_1 \Delta w+R-a(w+w^{*})-\delta (w+w^{*})(b+b^{*}),\\
		\displaystyle \frac{\partial b}{\partial t}=d_2 \Delta b+\rho (b+b^{*})(1-\frac{b+b^{*}}{w+w^{*}})-\frac{\mu(b+b^{*})}{1+\theta_1 (b+b^{*})+\theta_2 (w+w^{*})}.
	\end{cases}
\end{equation}
For convenience, we introduce the perturbation parameters $\varepsilon_1, \varepsilon_2$ by setting
$$R=R^{*}+\varepsilon_1,\ \ \ d_1 = d_1^{*} + \varepsilon_2,$$
 where $R^{*}=R_{i} (i=1,2)$.
Since $\varepsilon_1 \ll 1$, ignoring higher-order terms, $w^{*}(\varepsilon_1)$ and $b^{*}(\varepsilon_1)$ in $E^{*}$ can be expressed as
$$w^{*}(\varepsilon_1) \approx w^{*}+\varepsilon_1 \tilde{w},\quad b^{*}(\varepsilon_1) \approx b^{*}+\varepsilon_1 \tilde{b},$$
where
$$\tilde{w}=\frac{a+\delta b^{*}-\delta R^{*} \tilde{b}}{(a+\delta b^{*})^2},$$
$$\tilde{b}=\frac{\rho\delta  \left(\theta_1-\theta_2\right) b^{*2}+\rho\left[\delta +\left(\theta_1-\theta_2\right) a\right] b^*+2\rho R^* \theta_2+a\rho  -\mu  \left(b^* \delta +a\right)}{M},$$
with
$$
\begin{aligned}
	M = &4 b^{*3} \rho\delta^{2} \theta_1+\rho\left(6 a \delta  \theta_1+3 \delta^{2}\right) b^{*2}+\rho\left[\left(-2 R^* \theta_1+2 R^* \theta_2+4 a\right) \delta +2 a^{2} \theta_1\right] b^*\\
	& -\rho R^* \delta -a\rho \left[-a+R^* \left(\theta_1-\theta_2\right)\right]  +R^* \delta  \mu.
\end{aligned}
$$
Then, system (\ref{eq4.1}) becomes
\begin{equation}\label{eq4.2}
	\begin{cases}
		\begin{aligned}
			\frac{\partial w}{\partial t} &= (d_1^{*} + \varepsilon_2)\Delta w + A_{11}w + A_{12}b + A_{13}w^2 + A_{14}wb + A_{15}b^2 + A_{16}w^3 \\
			&\quad +A_{17}w^2b +  A_{18}wb^2 + A_{19}b^3,
		\end{aligned} \\
		\begin{aligned}
			\frac{\partial b}{\partial t} &= d_2 \Delta b + A_{21}w + A_{22}b + A_{23}w^2 + A_{24}wb + A_{25}b^2  + A_{26}w^3 \\
			&\quad+ A_{27}w^2b + A_{28}wb^2 + A_{29}b^3,
		\end{aligned}
	\end{cases}
\end{equation}
where
\begin{center}
	$\begin{aligned}
		A_{11}&=-a-\delta b^{*}(\varepsilon_1)=-a-\delta b^{*}+\varepsilon_1(-\delta \tilde{b}),\\
		A_{12}&=-\delta w^{*}(\varepsilon_1)=-\delta w^{*}+\varepsilon_1(-\delta \tilde{w}),\\
		A_{14}&=-\delta,\\
		A_{21}&=\frac{\rho b^{*2}(\varepsilon_1)}{w^{*2}(\varepsilon_1)}+\frac{\mu \theta_2 b^{*}(\varepsilon_1)}{\left[1+\theta_1 b^{*}(\varepsilon_1)+\theta_2 w^{*}(\varepsilon_1)\right]^2}\\
		&=\frac{\rho b^{*2}}{w^{*2}}+\frac{\mu \theta_2 b^{*}}{(1+\theta_1 b^{*}+\theta_2 w^{*})^2}+\varepsilon_1 (2 \tilde{w} A_{23} |_{\varepsilon_1=0}+ \tilde{b} A_{24}|_{\varepsilon_1=0}),\\
		A_{22}&=\rho (1-\frac{2 b^{*}(\varepsilon_1)}{w^{*}(\varepsilon_1)})-\frac{\mu}{1+\theta_1 b^{*}(\varepsilon_1)+\theta_2 w^{*}(\varepsilon_1)}+\frac{\mu \theta_1 b^{*}(\varepsilon_1)}{\left[1+\theta_1 b^{*}(\varepsilon_1)+\theta_2 w^{*}(\varepsilon_1)\right]^2}\\
		&=\frac{\mu \theta_1 b^{*}}{(1+\theta_1 b^{*} + \theta_2 w^{*})^2} - \frac{\rho b^{*}}{w^{*}}+\varepsilon_1 ( \tilde{w} A_{24} |_{\varepsilon_1=0}+2 \tilde{b} A_{25}|_{\varepsilon_1=0}),\\
		A_{23}&=-\frac{\rho b^{*2}(\varepsilon_1)}{w^{*3}(\varepsilon_1)}-\frac{\mu \theta_2^2 b^{*}(\varepsilon_1)}{\left[1+\theta_1 b^{*}(\varepsilon_1)+\theta_2 w^{*}(\varepsilon_1)\right]^3},
	\end{aligned}$
\end{center}
\begin{center}
	$\begin{aligned}
		A_{24}&=\frac{2 \rho b^{*}(\varepsilon_1)}{w^{*2}(\varepsilon_1)}+\frac{\mu \theta_2}{\left[1+\theta_1 b^{*}(\varepsilon_1)+\theta_2 w^{*}(\varepsilon_1)\right]^2}-\frac{2 \mu \theta_1 \theta_2 b^{*}(\varepsilon_1)}{\left[1+\theta_1 b^{*}(\varepsilon_1)+\theta_2 w^{*}(\varepsilon_1)\right]^3},\\
		A_{25}&=-\frac{ \rho }{w^{*}(\varepsilon_1)}+\frac{\mu \theta_1}{\left[1+\theta_1 b^{*}(\varepsilon_1)+\theta_2 w^{*}(\varepsilon_1)\right]^2}-\frac{ \mu \theta_1^2  b^{*}(\varepsilon_1)}{\left[1+\theta_1 b^{*}(\varepsilon_1)+\theta_2 w^{*}(\varepsilon_1)\right]^3},\\
		A_{26}&=\frac{ \rho b^{*2}(\varepsilon_1)}{w^{*4}(\varepsilon_1)}+\frac{ \mu  \theta_2^3 b^{*}(\varepsilon_1)}{\left[1+\theta_1 b^{*}(\varepsilon_1)+\theta_2 w^{*}(\varepsilon_1)\right]^4},\\
		A_{27}&=-\frac{2 \rho b^{*}(\varepsilon_1)}{w^{*3}(\varepsilon_1)}-\frac{\mu \theta_2^2}{\left[1+\theta_1 b^{*}(\varepsilon_1)+\theta_2 w^{*}(\varepsilon_1)\right]^3}+\frac{3 \mu \theta_1 \theta_2^2 b^{*}(\varepsilon_1)}{\left[1+\theta_1 b^{*}(\varepsilon_1)+\theta_2 w^{*}(\varepsilon_1)\right]^4},\\
		A_{28}&=\frac{\rho}{w^{*2}(\varepsilon_1)}-\frac{2 \mu \theta_1 \theta_2}{\left[1+\theta_1 b^{*}(\varepsilon_1)+\theta_2 w^{*}(\varepsilon_1)\right]^3}+\frac{3 \mu \theta_1^2 \theta_2 b^{*}(\varepsilon_1)}{\left[1+\theta_1 b^{*}(\varepsilon_1)+\theta_2 w^{*}(\varepsilon_1)\right]^4},\\
		A_{29}&=-\frac{ \mu \theta_1^2 (1+\theta_2 w^{*}(\varepsilon_1))}{\left[1+\theta_1 b^{*}(\varepsilon_1)+\theta_2 w^{*}(\varepsilon_1)\right]^4},\\
		A_{13}&=A_{15}=A_{16}=A_{17}=A_{18}=A_{19}=0.
	\end{aligned}$
\end{center}

It follows from Theorem \ref{thm3.3} that for system (\ref{eq1.1}), when $\varepsilon_1 = \varepsilon_2 = 0$, $\Delta_0(\lambda) = 0$ has a pair of purely
imaginary roots $\pm i \omega_0$ with $\omega_0=\sqrt{J_0}$; $\Delta_{k_{*}}(\lambda) = 0$
has a zero root and a negative real root $\lambda=T_{k_{*}}$;
and if $k \neq 0,k_{*}$, all roots of $\Delta_k(\lambda) = 0$ have negative real parts. Therefore, for \eqref{eq4.1} a $(k_*,0)-$mode TH bifurcation occurs near the origin  at
$(\varepsilon_1, \varepsilon_2)=(0,0)$, that is, $(0,0)$  is the TH
singularity in the perturbation plane of $\varepsilon_1$ and $\varepsilon_2$.

Similarly to \cite{an2018turing}, for system (\ref{eq4.2}) we define $U=(w,b)^T\in X_{\mathbb{C}}$,
$$
D(\varepsilon)=\begin{pmatrix}
	d_1^*+\varepsilon_2 & 0 \\
	0 & d_2
\end{pmatrix}, \quad
L(\varepsilon)=\begin{pmatrix}
	A_{11} & A_{12} \\
	A_{21} & A_{22}
\end{pmatrix}
$$
and
$$
G(U,\varepsilon_1,\varepsilon_2)=\begin{pmatrix}
A_{13}w^2 + A_{14}wb + A_{15}b^2 + A_{16}w^3+A_{17}w^2b +  A_{18}wb^2 + A_{19}b^3\\
	A_{23}w^2 + A_{24}wb + A_{25}b^2  + A_{26}w^3 + A_{27}w^2b + A_{28}wb^2 + A_{29}b^3
\end{pmatrix}.
$$
Then system (\ref{eq4.2}) can be written as an abstract
equation in phase space $X_{\mathbb{C}}$:
\begin{equation}\label{cq}
\frac{d U}{dt}=D(\varepsilon)\Delta U+L(\varepsilon)U+G(U,\varepsilon_1,\varepsilon_2).
\end{equation}	
For simplicity, we rewrite $D(\varepsilon)$ and $L(\varepsilon)$ as
$$
D(\varepsilon)=D_0 + D_1^{(1,0)}\varepsilon_1+D_1^{(0,1)}\varepsilon_2,
$$
$$
L(\varepsilon)=L_0 + L_1^{(1,0)}\varepsilon_1+L_1^{(0,1)}\varepsilon_2.
$$
Thus, we can obtain $D_0,  D_1^{(1,0)}, D_1^{(0,1)}, L_0,  L_1^{(1,0)}$ and $L_1^{(0,1)}$ as follows:
$$
D_0=\begin{pmatrix}
	d_1^* & 0 \\
	0 & d_2
\end{pmatrix}, \quad D_1^{(1,0)}=\begin{pmatrix}
	0 & 0 \\
	0 & 0
\end{pmatrix}, \quad D_1^{(0,1)}=\begin{pmatrix}
	1 & 0 \\
	0 & 0
\end{pmatrix},
$$
$$
L_0=\begin{pmatrix}
	l_{0,11} & l_{0,12} \\
	l_{0,21} & l_{0,22}
\end{pmatrix}, \quad L_1^{(1,0)}=\begin{pmatrix}
	l_{1,11} & l_{1,12} \\
	l_{1,21} & l_{1,22}
\end{pmatrix}, \quad L_1^{(0,1)}=\begin{pmatrix}
	0 & 0 \\
	0 & 0
\end{pmatrix},
$$
with
\begin{gather*}
	\begin{aligned}
		l_{0,11}&=-a-\delta b^{*}, \\
		l_{0,21}&=\frac{\rho b^{*2}}{w^{*2}}+\frac{\mu \theta_2 b^{*}}{\left(1+\theta_1 b^{*}+\theta_2 w^{*}\right)^2}, \\
		l_{1,11}&=-\delta \tilde{b}, \\
		l_{1,21}&=2 \tilde{w} A_{23} |_{\varepsilon_1=0}+ \tilde{b} A_{24}|_{\varepsilon_1=0},
	\end{aligned}
	\qquad
	\begin{aligned}
		l_{0,12}&=-\delta w^{*},\\
		l_{0,22}&=\frac{\mu \theta_1 b^{*}}{(1+\theta_1 b^{*} + \theta_2 w^{*})^2} - \frac{\rho b^{*}}{w^{*}},\\
		l_{1,12}&=-\delta \tilde{w},\\
		l_{1,22}&=\tilde{w} A_{24} |_{\varepsilon_1=0}+2 \tilde{b} A_{25}|_{\varepsilon_1=0}.
	\end{aligned}
\end{gather*}

Let
$$
\mathcal{M}_k(\lambda)=\begin{pmatrix}
	\lambda + d_1^*\frac{k^2}{l^2}-l_{0,11}& -l_{0,12}\\
	-l_{0,21}&\lambda+d_2 \frac{k^2}{l^2}-l_{0,22}
\end{pmatrix}.
$$
Then, by a straightforward calculation, we can
obtain that $p_0 \in \mathbb{C}^2$ and $p_{k_*}\in \mathbb{R}^2$ are the eigenvectors associated with the eigenvalues $i\omega_0$ and 0,
respectively, and $q_0 \in \mathbb{C}^2$ and $q_{k_*}\in \mathbb{R}^2$ are the
corresponding adjoint eigenvectors, where
$$
p_0=\begin{pmatrix}
	p_{01}\\
	p_{02}
\end{pmatrix}=\begin{pmatrix}
	1\\
	\frac{i\omega_0 - l_{0,11}}{l_{0,12}}
\end{pmatrix},\quad
q_0=\begin{pmatrix}
	q_{01}\\
	q_{02}
\end{pmatrix}=D_1\begin{pmatrix}
	1\\
	\frac{i\omega_0 - l_{0,11}}{l_{0,21}}
\end{pmatrix},
$$
$$
p_{k_*}=\begin{pmatrix}
	p_{k_*1}\\
	p_{k_*2}
\end{pmatrix}=\begin{pmatrix}
	1\\
	\frac{d_1^*\frac{k_*^2}{l^2}-l_{0,11}}{l_{0,12}}
\end{pmatrix},\quad
q_{k_*}=\begin{pmatrix}
	q_{k_*1}\\
	q_{k_*2}
\end{pmatrix}=D_2\begin{pmatrix}
	1\\
	\frac{d_1^*\frac{k_*^2}{l^2}-l_{0,11}}{l_{0,21}}
\end{pmatrix},
$$
with
$$
D_1=\frac{l_{0,21}l_{0,12}}{l_{0,21}l_{0,12}+(i\omega_0 - l_{0,11})^2},\quad
D_2=\frac{l_{0,21}l_{0,12}}{l_{0,21}l_{0,12}+(d_1^*\frac{k_*^2}{l^2}-l_{0,11})^2},
$$
such that
$$
\langle\Psi_1,\Phi_1\rangle = I_2, \quad \langle\Psi_2,\Phi_2\rangle = 1,
$$
where $I_2$ is an $2\times 2$ identity matrix and
$$
\Phi_1 = (p_0,\overline{p_0}), \quad \Phi_2=p_{k_*},
$$
$$
\Psi_1=\text{col}(q_0^T,\overline{q_0}^T), \quad \Psi_2=q_{k_*}^T.
$$

Following the techniques in \cite{song2016turing}, by a recursive transformation, we can obtain that the normal form of \eqref{cq} for TH bifurcation reads as
\begin{equation}\label{eq4.3}
	\dot{z}=Bz +
	\begin{pmatrix}
		B_{11}\varepsilon_1z_1 + B_{21}\varepsilon_2z_1\\
		\overline{B}_{11}\varepsilon_1z_1 + \overline{B}_{21}\varepsilon_2z_1\\
		B_{13}\varepsilon_1z_3 + B_{23}\varepsilon_2z_3
	\end{pmatrix}+
	\begin{pmatrix}
		B_{210}z_1^2z_2 + B_{102}z_1z_3^2\\
		\overline{B}_{210}z_1^2z_2 + \overline{B}_{102}z_1z_3^2\\
		B_{111}z_1z_2z_3 + B_{003}z_3^3
	\end{pmatrix}+O(|z||\varepsilon|^2),
\end{equation}
where
$$
B_{11}=q_0^T L_1^{(1,0)}p_0,\ B_{21}=q_0^T L_1^{(0,1)}p_0,
$$
$$
B_{13}=q_{k_*}^T\left(-\frac{k_*^2}{l^2}D_1^{(1,0)}p_{k_*}+L_1^{(1,0)}p_{k_*}\right),\ B_{23}=q_{k_*}^T\left(-\frac{k_*^2}{l^2}D_1^{(0,1)}p_{k_*}+L_1^{(0,1)}p_{k_*}\right),
$$
and
$$
B_{210}=C_{210}+\frac{3}{2}(D_{210}+E_{210}),\ B_{102}=C_{102}+\frac{3}{2}(D_{102}+E_{102}),
$$
$$
B_{111}=C_{111}+\frac{3}{2}(D_{111}+E_{111}),\ B_{003}=C_{003}+\frac{3}{2}(D_{003}+E_{003}).
$$
Here, the expressions of $C_{ijk}, D_{ijk}$ and $E_{ijk}$ are similar to those in \cite{liu2023turing} (see Appendix A).

Through the variable transformation $z_1=v_1-iv_2, z_2=v_1+iv_2$ and $z_3=v_3$, the normal form (\ref{eq4.3}) can be expressed in real coordinates $\omega$. Subsequently, by using $v_1=\rho cos \Theta$,$v_2=\rho sin \Theta$ and $v_3=s$, it can be rewritten in cylindrical coordinates. Truncating at third-order terms and removing the azimuthal term, ultimately, we obtain
\begin{equation}\label{eq4.4}
		\begin{cases}
		\dot{\rho}=\nu_1(\varepsilon)\rho+\kappa_{11}\rho^3+\kappa_{12}\rho s^2,\\
		\dot{s}=\nu_2(\varepsilon)s+\kappa_{21}\rho^2s+\kappa_{22}s^3,
	\end{cases}
\end{equation}
where
$$\nu_1(\varepsilon)=\text{Re}(B_{11})\varepsilon_1+\text{Re}(B_{21})\varepsilon_2,\quad \nu_2(\varepsilon)=B_{13}\varepsilon_1 + B_{23}\varepsilon_2,$$
$$\kappa_{11}=\text{sign}(\text{Re}(B_{210})),\quad
	\kappa_{12}=\frac{\text{Re}(B_{102})}{|B_{003}|},\quad
	\kappa_{21}=\frac{B_{111}}{|\text{Re}(B_{210})|},\quad
	\kappa_{22}=\text{sign}(B_{003}).$$
\section{Numerical simulations}\label{5}
In this section, we shall carry out some numerical simulations to supplement our previous theoretical analysis.

\subsection{The rich dynamics of the ODE system}
 Firstly, we  fix the parameters as
$
a=0.1,\ \delta=0.12,\ \rho=1,\ \mu=10,\ \theta_1=2.5$ and vary $\theta_2, R$
 to investigate how the parameters $\theta_2, R$ affect the dynamics of  system (\ref{eq2.1}).

\begin{figure}[!b]
	\centering
	\includegraphics[width=0.45\textwidth]{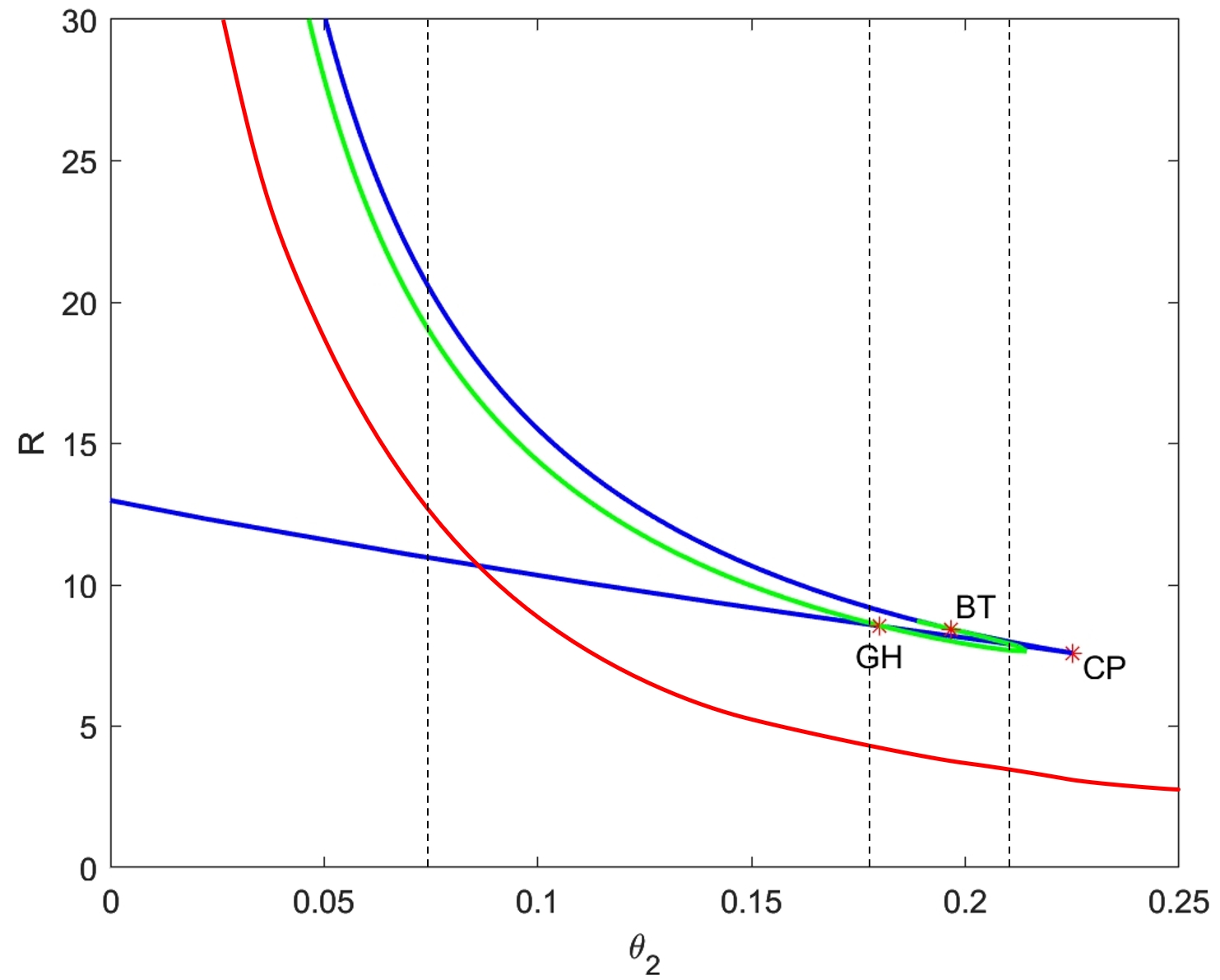}
	\caption{The green curve represents the Hopf bifurcation curve, the blue curve represents the saddle-node bifurcation curve and the red curve is  the transcritical
bifurcation curve $R=\frac{a(\mu-\rho)}{\rho \theta_2}$. The dashed lines from left to right correspond to $\theta_2 = 0.075$, $\theta_2 = 0.18$, and $\theta_2 = 0.21$, respectively. Other parameters: $
a=0.1,\ \delta=0.12,\ \rho=1,\ \mu=10,\ \theta_1=2.5$. }
	\label{fig:5.1.31}
\end{figure}
  For this case,  system (\ref{eq2.1}) may have up to three positive equilibria, $ E_{31}^* $, $ E_{32}^* $, and $ E_{33}^* $. Fig. \ref{fig:5.1.31} illustrates  the two-parameter bifurcation diagrams on the $\theta_2-R$ plane. In Fig. \ref{fig:5.1.31}, the green curve represents the Hopf bifurcation curve, the blue curve represents the saddle-node bifurcation curve and the red curve is  the transcritical
bifurcation curve $R=\frac{a(\mu-\rho)}{\rho \theta_2}$.  The intersection point of the
saddle-node curve  and the Hopf curve is the Bogdanov-Takens bifurcation
point marked by \lq\lq BT\rq\rq, and the  point marked by \lq\lq GH\rq\rq is the generalized Hopf bifurcation point at which
 the Hopf bifurcation changes from subcritical to
supercritical. We observe that: (1) The system may undergo saddle-node bifurcations at $ R_{LP}^1 $ and $ R_{LP}^2 $ ($ R_{LP}^1 < R_{LP}^2 $), causing $ E_{32}^* $ to always be a saddle and therefore unstable, as also established by Theorem \ref{thm2.5}; (2)
The critical value $ R^* $ (the forward transcritical bifurcation point) can be either below or above $ R_{LP}^1 $, resulting in distinct dynamical behaviors; (3) There may exist one or two Hopf bifurcation points ($R_1$ and $R_2$), which may be subcritical or
supercritical resulting in  either stable or unstable limit cycles.

\begin{figure}[!b]
	\centering
	\subfigure[Bifurcation diagram]{\includegraphics[width=0.335\textwidth,height=4.5cm]{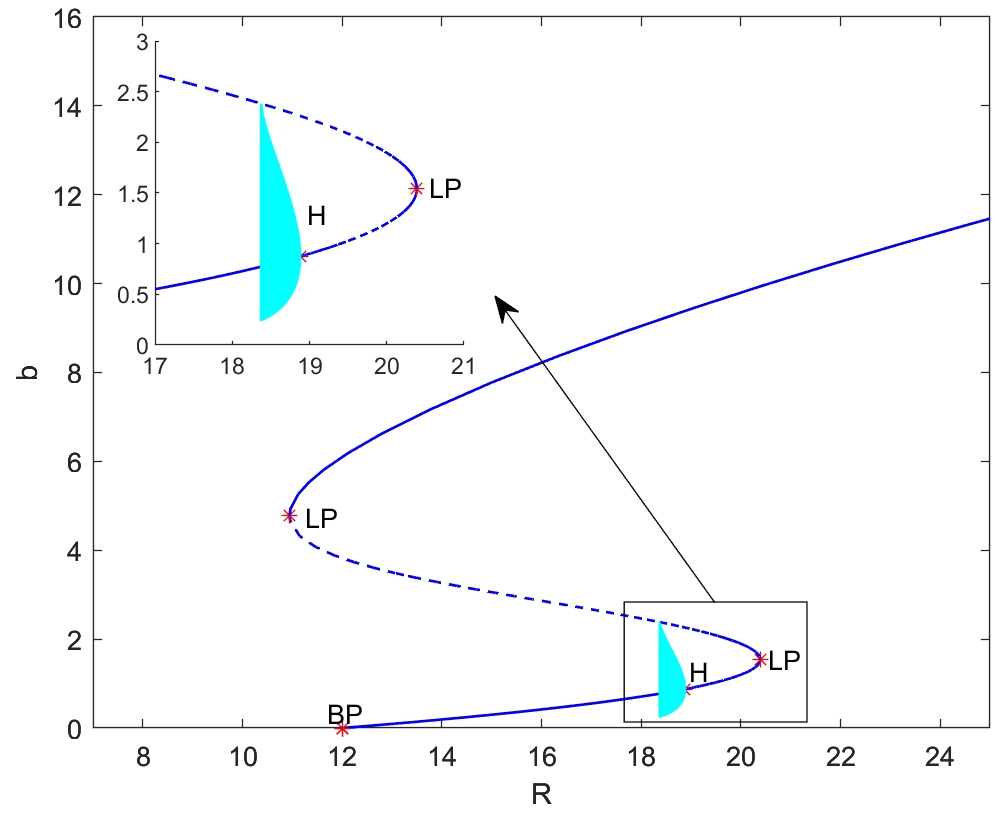}} \hspace{0.1pt}	
	\subfigure[$R_{LP}^1<R<R^*$]{\includegraphics[width=0.31\textwidth]{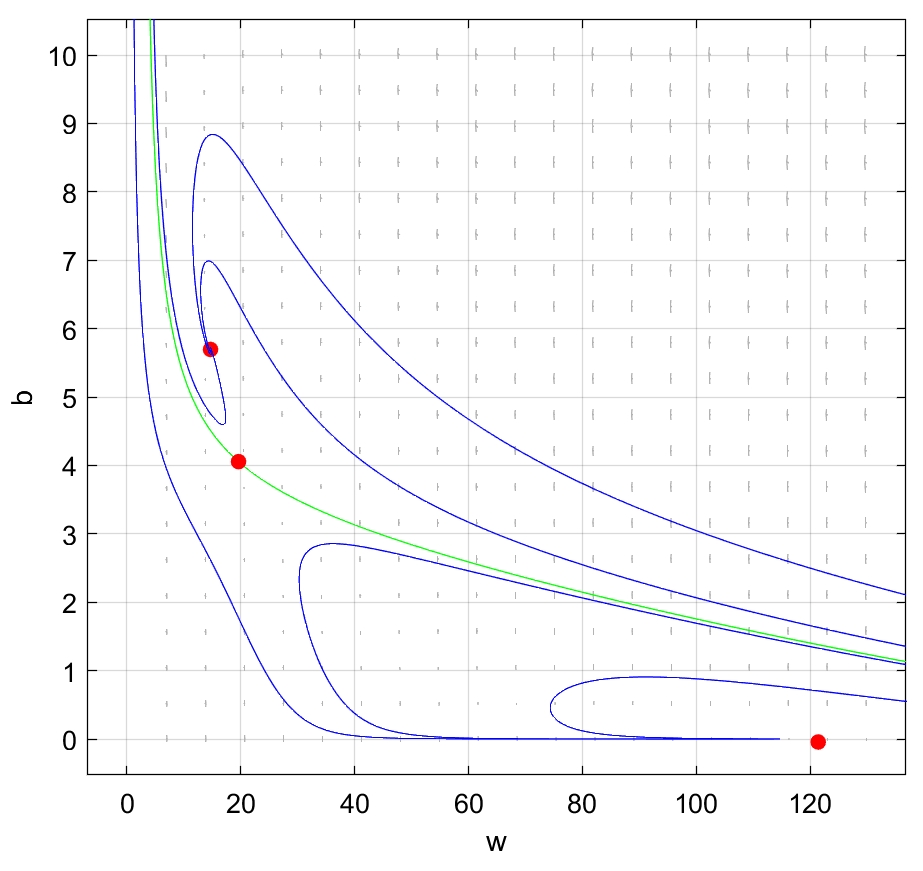}} \hspace{0.5pt}
	\subfigure[$R^*<R<R_{HL}$]{\includegraphics[width=0.314\textwidth]{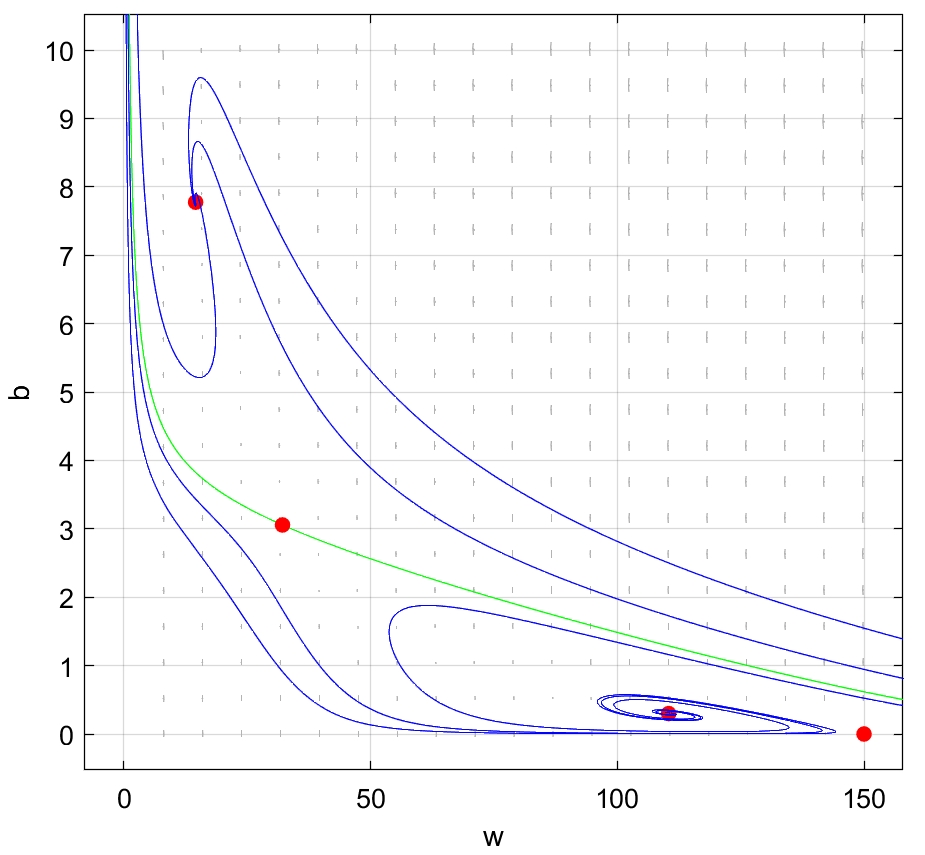}} \hspace{0.5pt}
	
	\subfigure[$R=R_{HL}$]{\includegraphics[width=0.32\textwidth]{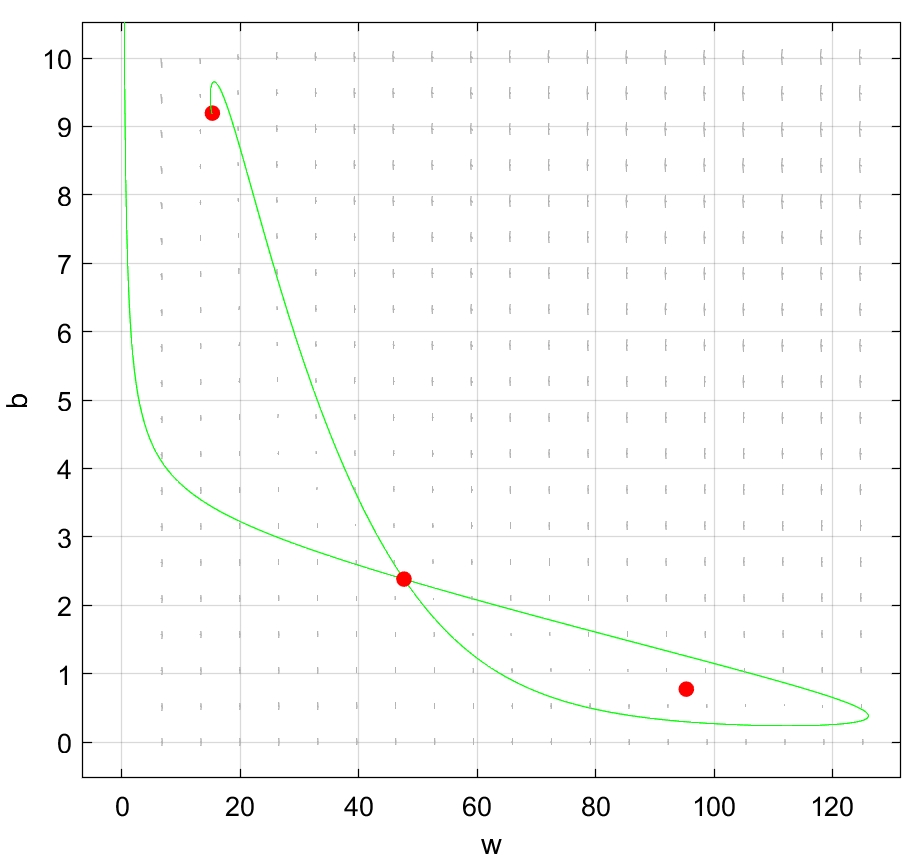}}\hspace{0.5pt}
	\subfigure[$R_{HL}<R<R_2$]{\includegraphics[width=0.325\textwidth]{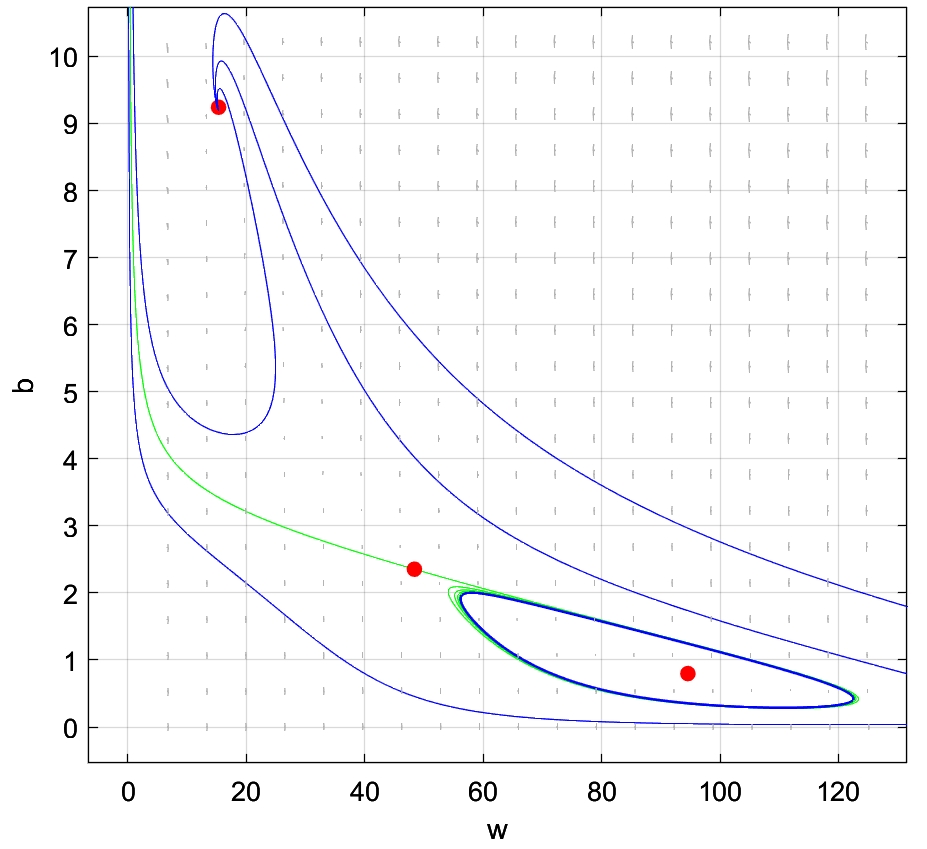}}\hspace{0.5pt}
	\subfigure[$R_2<R<R_{LP}^2$]{\includegraphics[width=0.32\textwidth]{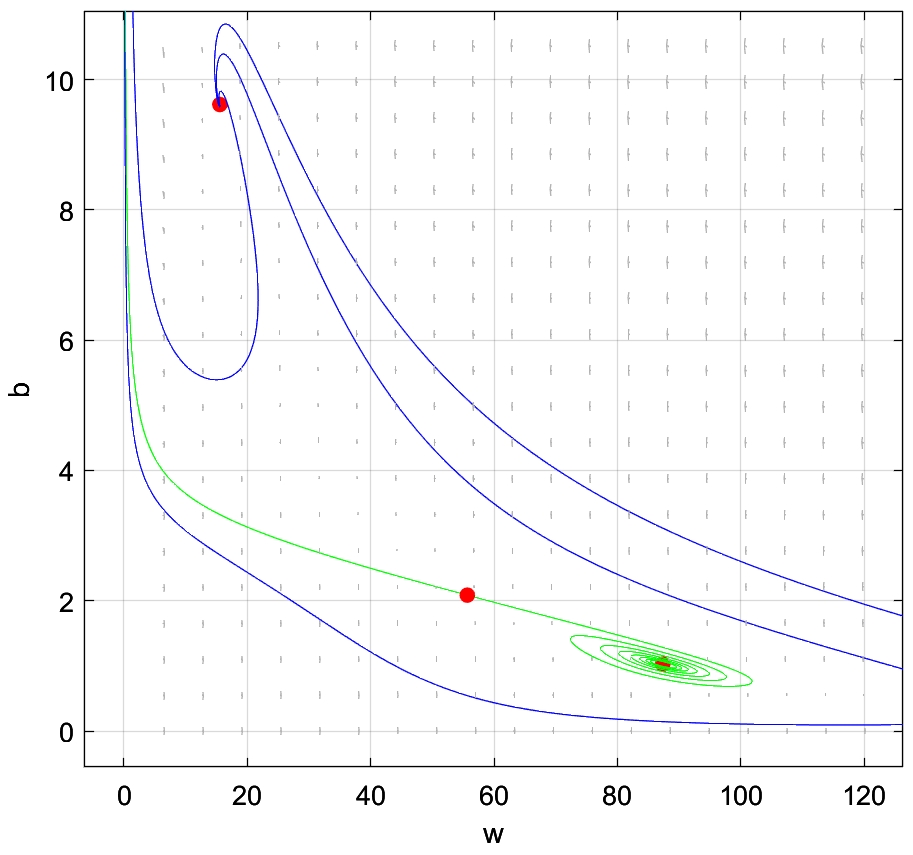}}
	\caption{Bifurcation diagram and the corresponding evolution of phase portraits for  $\theta_2 = 0.075$. In (a), the blue curve represents the biomass $b$ as a function of the rainfall rate $R$, while the cyan curve denotes the unstable limit cycle. In (b), $R = 11.5$; (c) $R = 15$; (d) $ R = 18.375122$; (e) $R = 18.5$; (f) $ R = 19.5 $. Other parameters: $
a=0.1,\ \delta=0.12,\ \rho=1,\ \mu=10,\ \theta_1=2.5$.}
	\label{fig:5.1.6}
\end{figure}

For $\theta_2 = 0.075$, the bifurcation diagram and the evolution of phase portraits as $ R $ increases are shown in Fig. \ref{fig:5.1.6}.
The bifurcation values are $ R^* = 12 $, $ R_2 = 18.877941 $, $ R_{LP}^1 = 10.943899 $, $ R_{LP}^2 = 20.391003 $, and $ R_{HL} = 18.375122 $, with a subcritical Hopf bifurcation occurring at $R = R_2$.
Here, $ R^* $ is greater than $ R_{LP}^1 $, leading to a scenario where both the bare-soil equilibrium $ E_0 $ and a positive equilibrium can be stable.
Specifically, when $ R_{LP}^1 < R < R^* $, there exist two positive equilibria $ E_{32}^* $ and $ E_{33}^* $, with $ E_{32}^* $ always being an unstable saddle. The system exhibits bistability between the bare-soil equilibrium $ E_0 $ and the positive equilibrium $ E_{33}^* $ with the stable manifold of the saddle acting as the boundary between their attraction basins (see Fig. \ref{fig:5.1.6}(b)). As $ R $ increases to $ R^* < R < R_{HL} $, two stable positive equilibria $ E_{31}^*, E_{33}^* $ and one unstable equilibrium $ E_{32}^* $ emerge, and the  basins of attraction for the two locally stable equilibria are separated by the stable manifold of the saddle equilibrium $ E_{32}^* $ (see Fig. \ref{fig:5.1.6}(c)). At $ R = R_{HL} $, a homoclinic orbit appears and it is the boundary between the attraction basins of $ E_{31}^* $ and $ E_{33}^* $ (see Fig. \ref{fig:5.1.6}(d)). As $ R $ increases to $ R_{HL} < R < R_2 $, an unstable limit cycle emerges from the subcritical Hopf bifurcation. In this case, the unstable cycle serves as the separatrix between the attraction basins of $ E_{31}^* $ and $ E_{33}^* $ (see Fig. \ref{fig:5.1.6}(e)). When $ R_2 < R < R_{LP}^2 $, the system has three positive equilibria, but only $ E_{33}^* $ is stable (see Fig. \ref{fig:5.1.6}(f)).
When $ R > R_{LP}^2 $, the system has a unique stable equilibrium $ E_{33}^* $.
In each of Figs. \ref{fig:5.1.6}(c), \ref{fig:5.1.6}(d), and \ref{fig:5.1.6}(e), system has a bistability of $ E_{31}^* $ and $ E_{33}^* $, but the separatrix is different.

\begin{figure}[!h]
	\centering
	\subfigure[Bifurcation diagram]{\includegraphics[width=0.335\textwidth,height=4.4cm]{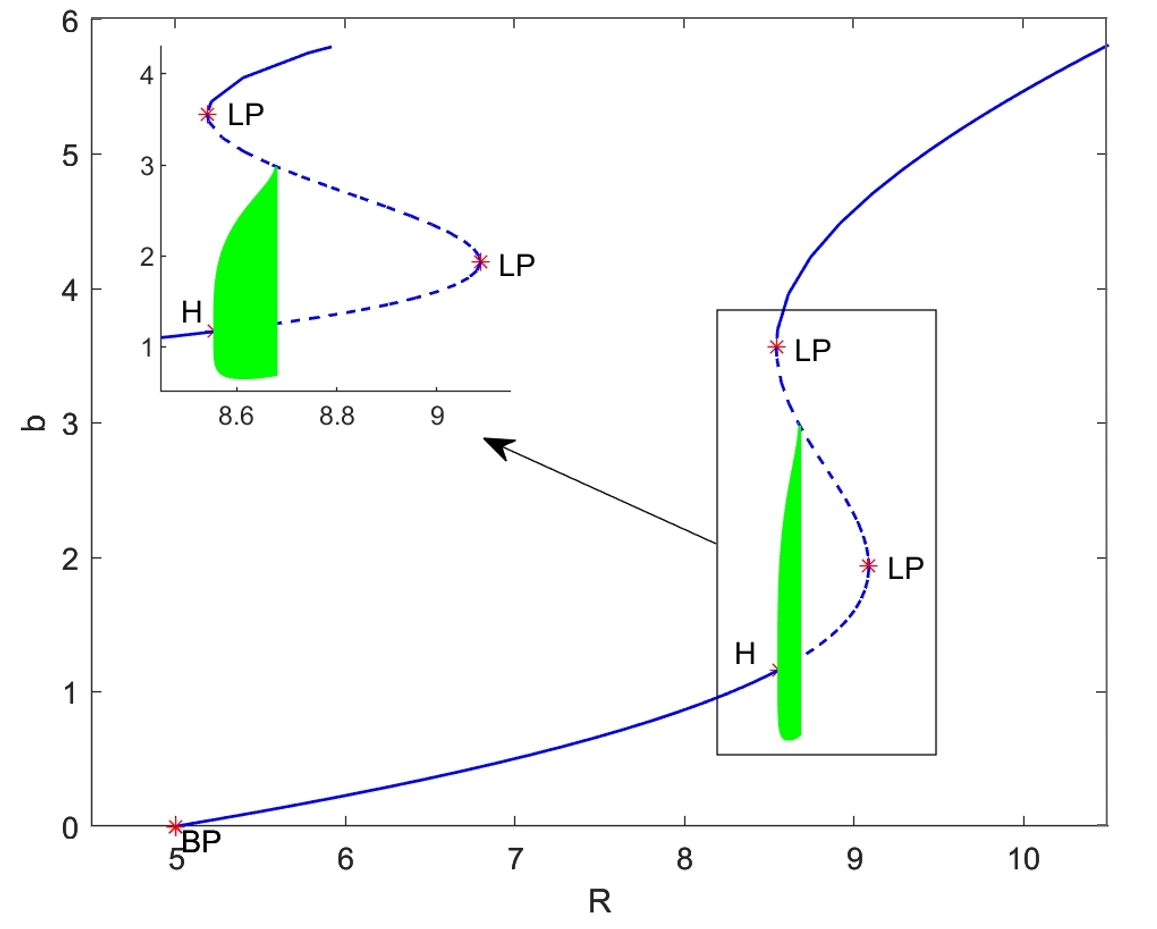}} \hspace{0.1pt}	
	\subfigure[$R^*<R<R_{LP}^1$]{\includegraphics[width=0.31\textwidth]{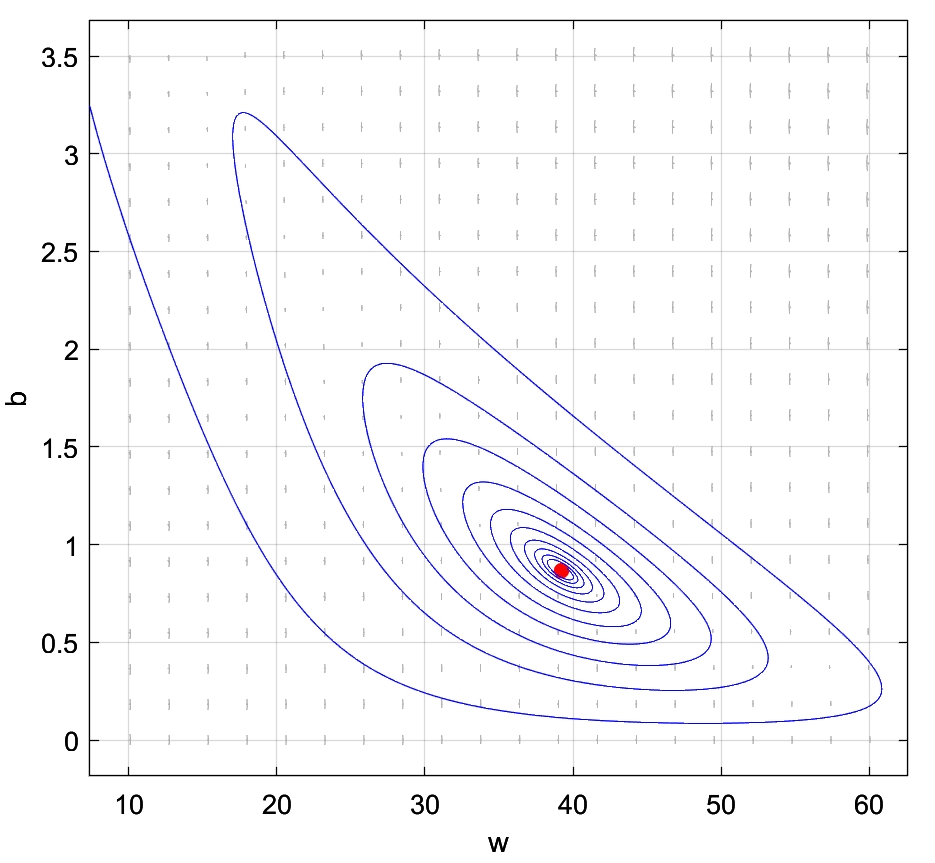}} \hspace{0.5pt}
	\subfigure[$R_{LP}^1<R<R_1$]{\includegraphics[width=0.308\textwidth]{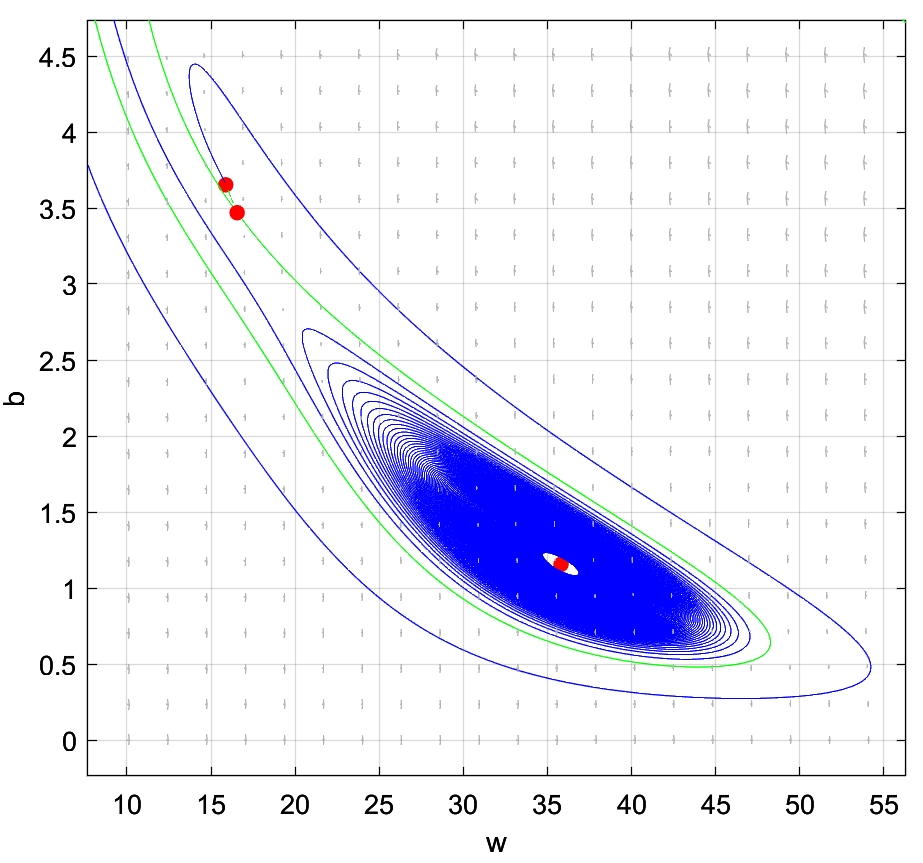}} \hspace{0.5pt}
	
	\subfigure[$R_1<R<R_{HL}$]{\includegraphics[width=0.32\textwidth]{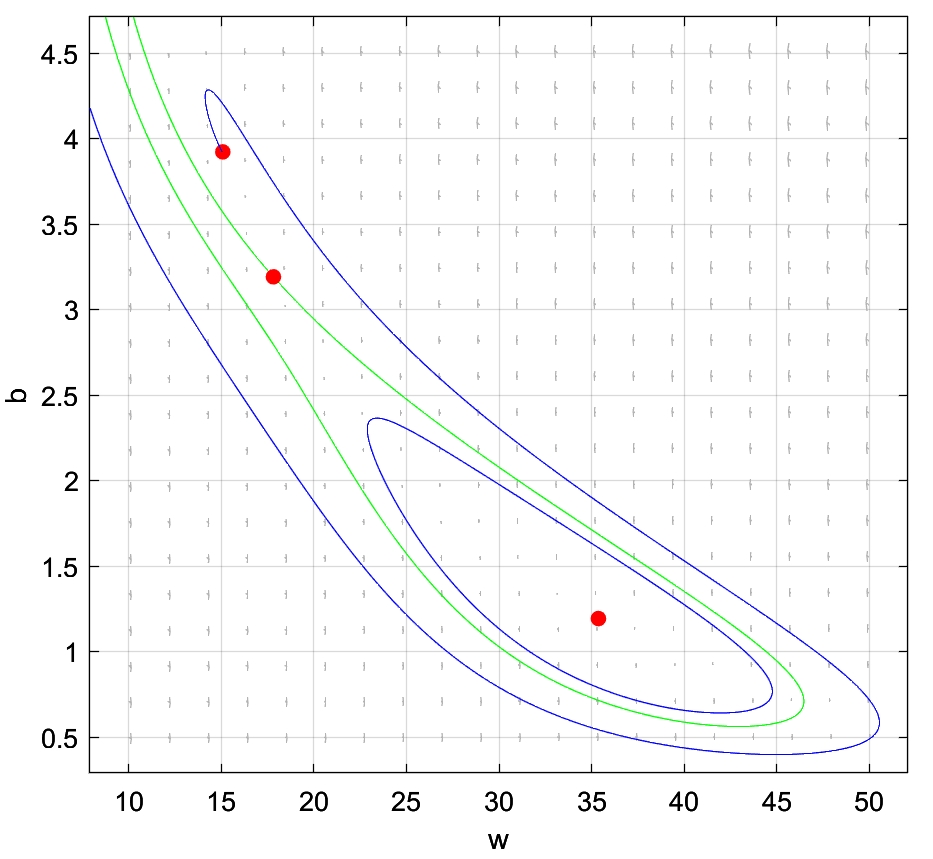}}\hspace{0.5pt}
	\subfigure[$R=R_{HL}$]{\includegraphics[width=0.315\textwidth]{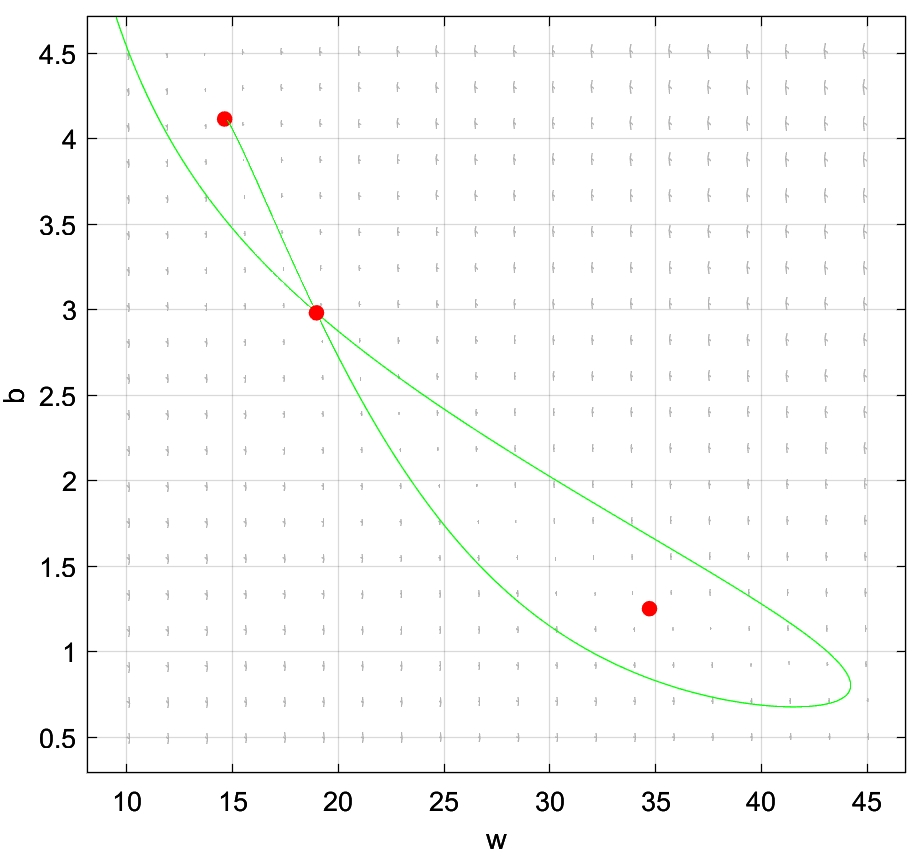}}\hspace{0.5pt}
	\subfigure[$R_{HL}<R<R_{LP}^2$]{\includegraphics[width=0.32\textwidth]{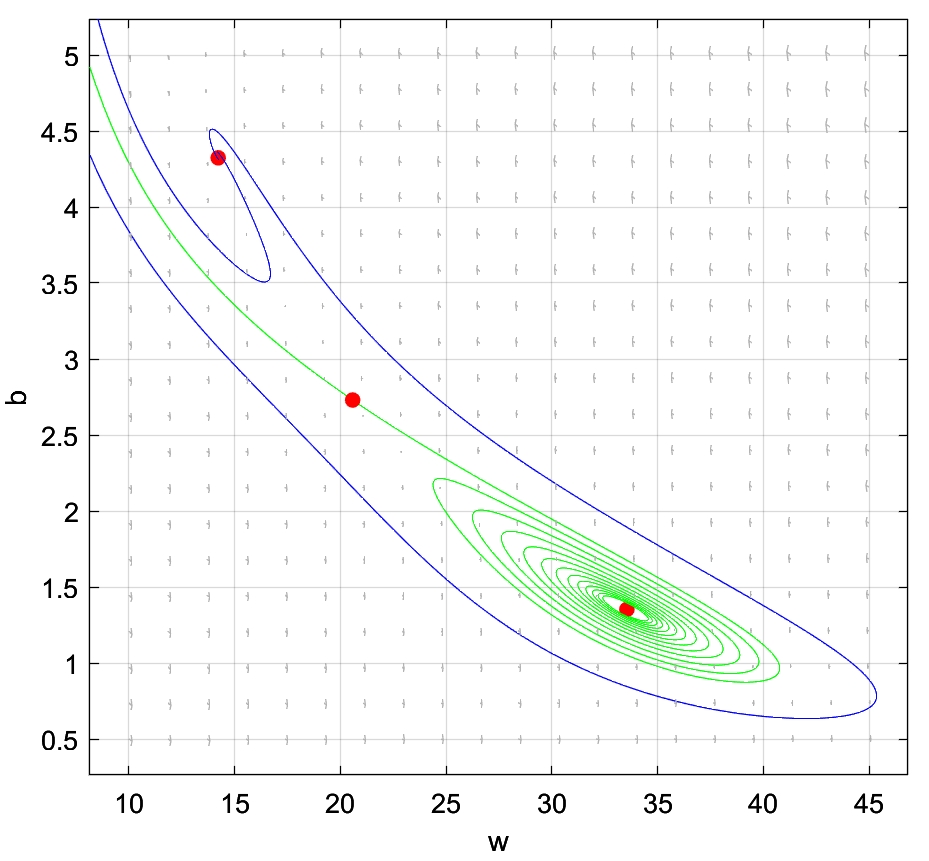}}
	\caption{Bifurcation diagram and the corresponding evolution of phase portraits for  $\theta_2 = 0.18$. In (a), the blue curve represents the biomass $b$ as a function of the rainfall rate $R$, while the green curve denotes the stable limit cycle. In (b), $R = 8$; (c) $R = 8.545$; (d) $ R = 8.6$; (e) $R = 8.6789505$; (f) $ R = 8.8 $. Other parameters: $
a=0.1,\ \delta=0.12,\ \rho=1,\ \mu=10,\ \theta_1=2.5$.}
	\label{fig:5.1.5}
\end{figure}

For $\theta_2 = 0.18$, the bifurcation diagram and the evolution of phase portraits as $ R $ increases are shown in Fig. \ref{fig:5.1.5}. Now the bifurcation values are $R^* = 5$, $R_1 = 8.5563186$, $R_{LP}^1 = 8.5412358$, and $R_{LP}^2 = 9.0870371$, with a supercritical Hopf bifurcation occurring at $R = R_1$. Additionally, there exists a homoclinic bifurcation point at $R_{HL} = 8.6789505$.
When $0 < R < R^*$, the equilibrium $E_0$ is locally asymptotically stable.
When $R^* < R < R_{LP}^1$, the system (2.1) has a unique positive equilibrium $E_{31}^*$, which is stable, while $E_0$ becomes unstable (see Fig. \ref{fig:5.1.5} (b)).
As $ R $ increases to $R_{LP}^1 < R < R_1$, three positive equilibria appear: $E_{31}^*$, $E_{32}^*$, and $E_{33}^*$, with $E_{32}^*$ always being an unstable saddle. The system (\ref{eq2.1}) admits the bistability of the positive equilibria $ E_{31}^* $ and $ E_{33}^* $ with the stable manifolds (the green orbits) of the saddle $ E_{32}^* $ as the separatrix of attraction (see Fig. \ref{fig:5.1.5} (c)).
When $R_1 < R < R_{HL}$, $E_{31}^*$ becomes unstable, and a stable limit cycle emerges from the supercritical Hopf bifurcation. In this case, the system exhibits bistability between the positive equilibrium $ E_{33}^* $ and a stable limit cycle, with the boundary still determined by the stable manifold of the saddle(see Fig. \ref{fig:5.1.5} (d)).
When R increases to $R = R_{HL}$, the period of the limit cycle tends to infinity and a homoclinic orbit emerges in Fig. \ref{fig:5.1.5} (e).
When $R_{HL} < R < R_{LP}^2$, the homoclinic orbit is broken, and among the three positive equilibria, only $E_{33}^*$ remains stable (see Fig. \ref{fig:5.1.5} (f)).
When $R > R_{LP}^2$, The system has only a unique stable equilibrium $E_{33}^*$.

\begin{figure}[!h]
	\centering
	\subfigure[Bifurcation diagram]{\includegraphics[width=0.335\textwidth,height=4.4cm]{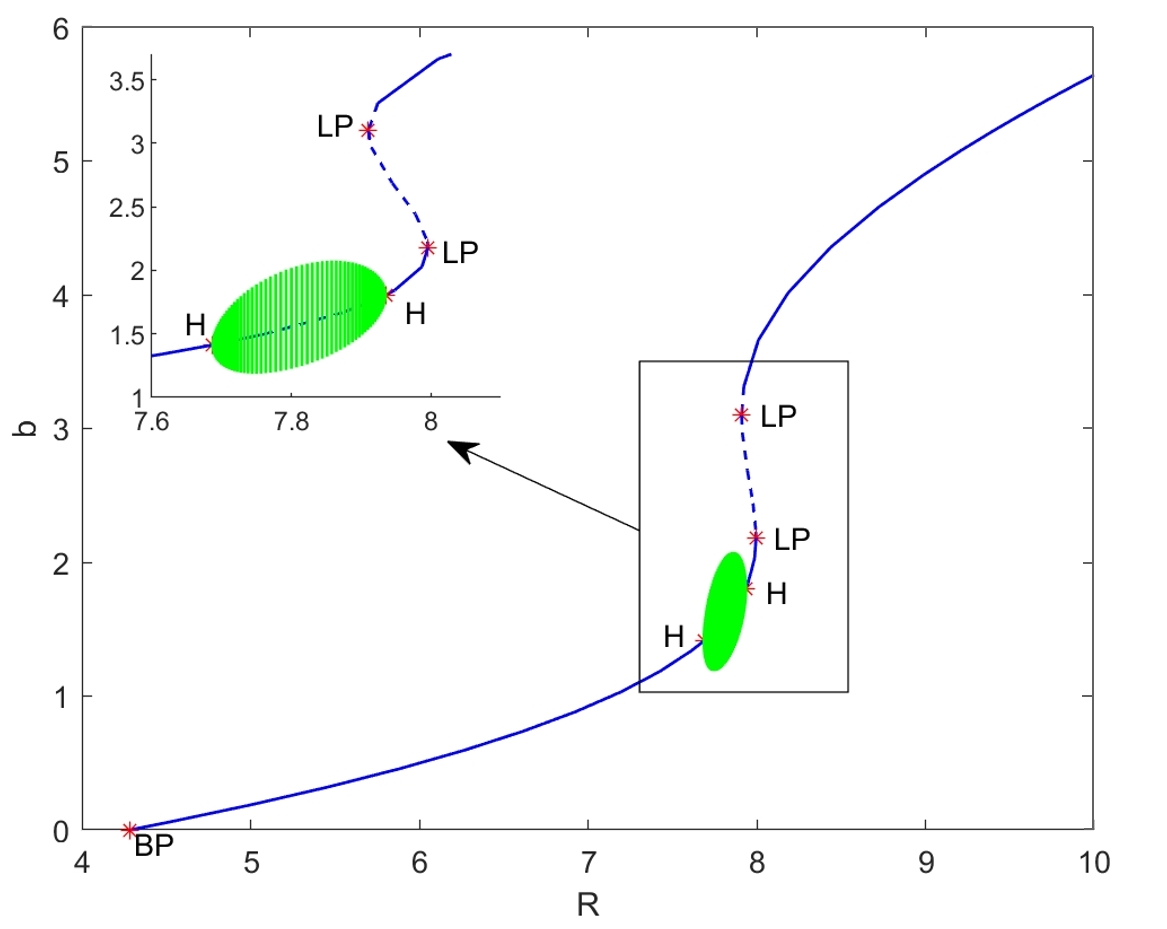}} \hspace{0.1pt}
	\subfigure[$R^*<R<R_1$]{\includegraphics[width=0.31\textwidth]{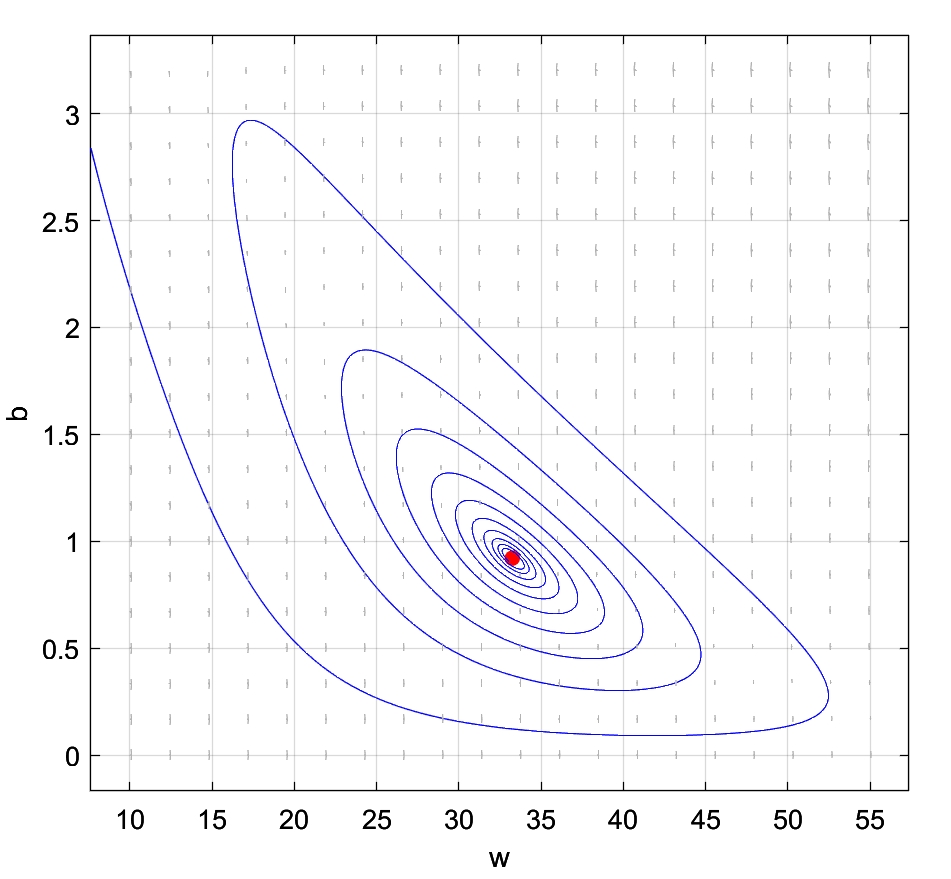}} \hspace{0.5pt}
	\subfigure[$R_1<R<R_{LP}^1$]{\includegraphics[width=0.305\textwidth]{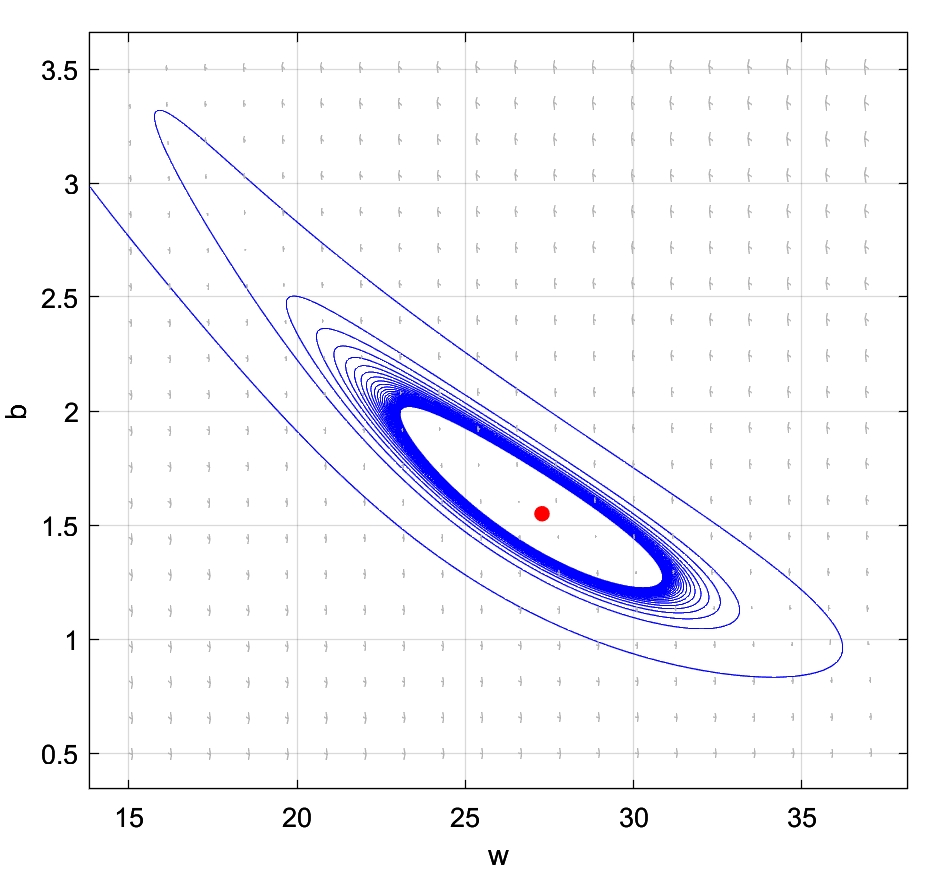}} \hspace{0.5pt}
	
	\subfigure[$R_{LP}^1<R<R_2$]{\includegraphics[width=0.32\textwidth]{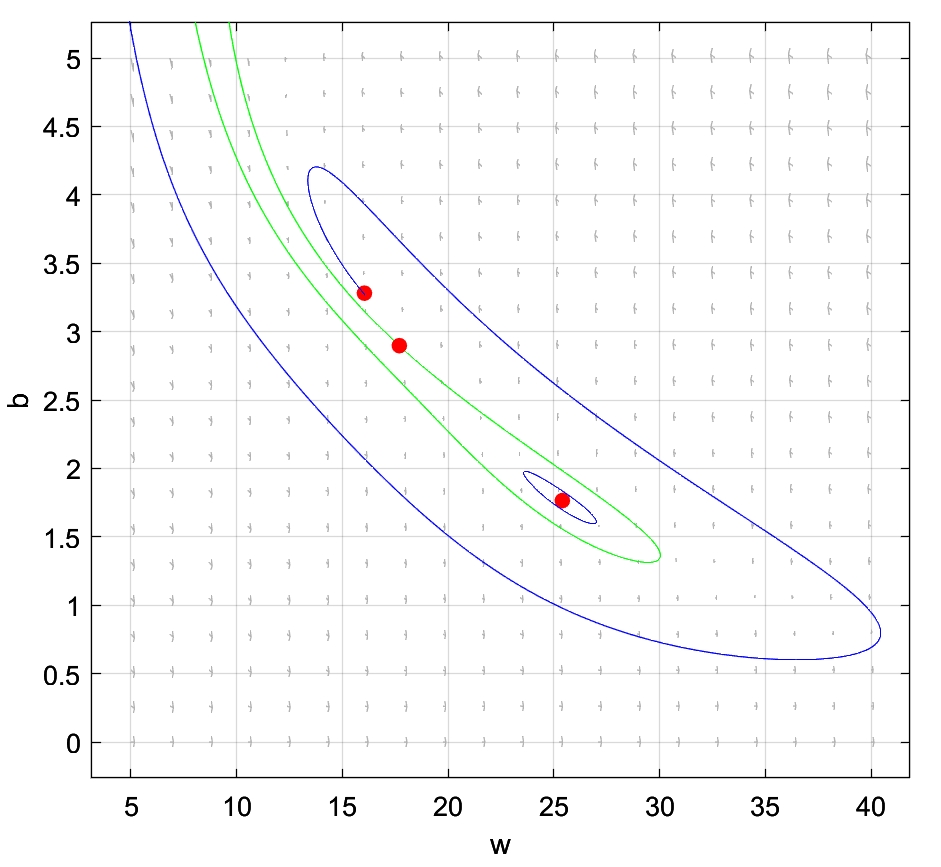}}\hspace{0.5pt}
	\subfigure[$R_2<R<R_{LP}^2$]{\includegraphics[width=0.32\textwidth]{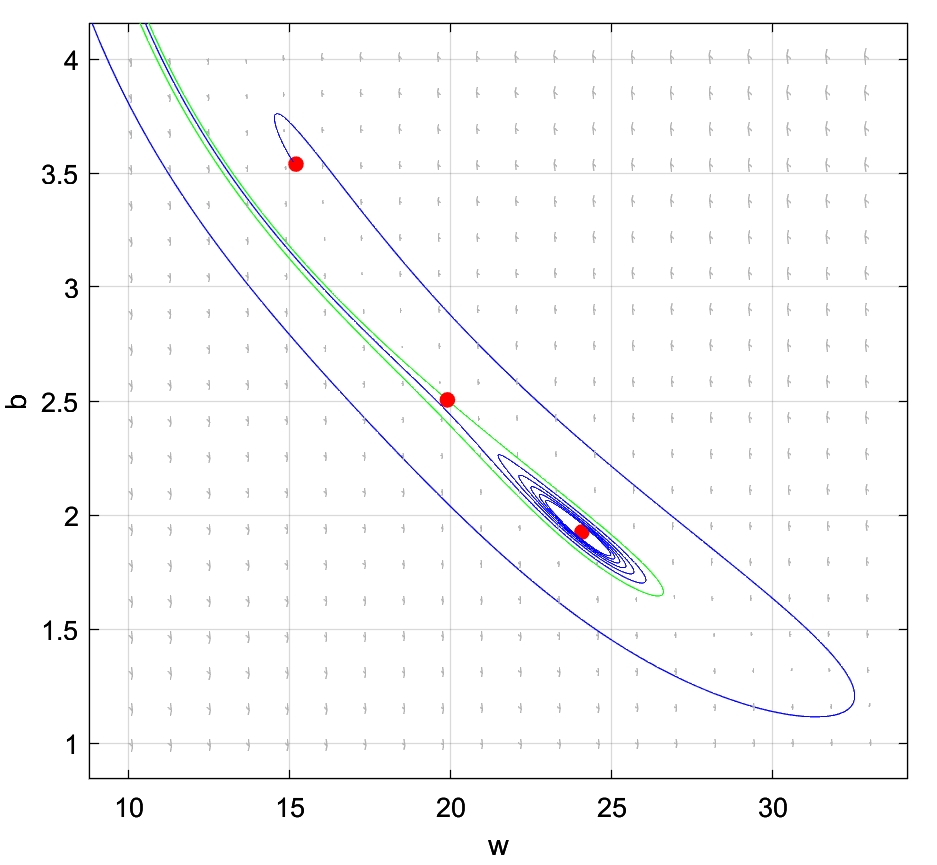}}\hspace{0.5pt}
	\subfigure[$R>R_{LP}^2$]{\includegraphics[width=0.32\textwidth]{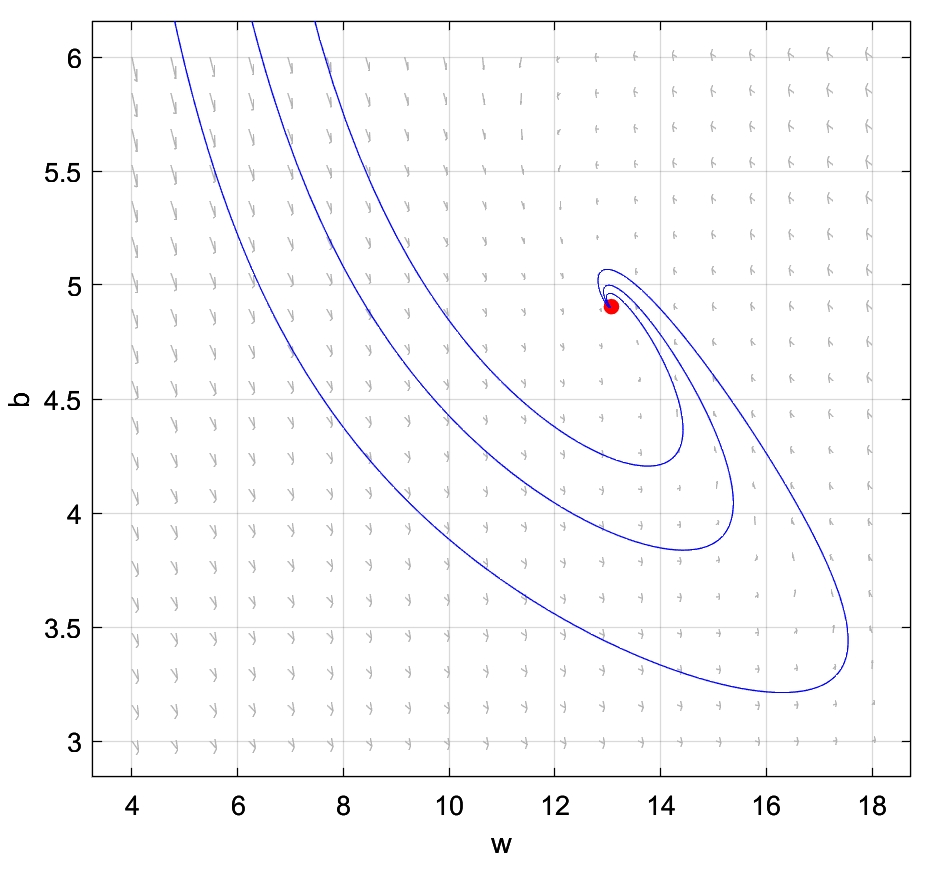}}
	\caption{Bifurcation diagram and the corresponding evolution of phase portraits for $\theta_2 = 0.21$. In (a), the blue curve represents the biomass $b$ as a function of the rainfall rate $R$, while the green curve denotes the stable limit cycle. In (b), $R = 7$; (c) $R = 7.8$; (d) $ R = 7.92$; (e) $R = 7.97$; (f) $ R = 9$. Other parameters: $
a=0.1,\ \delta=0.12,\ \rho=1,\ \mu=10,\ \theta_1=2.5$.}
	\label{fig:5.1.4}
\end{figure}

For $\theta_2 = 0.21$, the bifurcation diagram and the evolution of phase portraits as $ R $ increases are shown in, the evolution of phase portraits as the rainfall rate $ R $ increases is shown in Fig. \ref{fig:5.1.4}. The bifurcation values are $ R^* = 4.2857143 $, $ R_1 = 7.6871373 $ and $ R_2 = 7.9351011 $, as well as $ R_{LP}^1 = 7.9104673 $ and $ R_{LP}^2 = 7.9959705 $. Numerical calculations show that both Hopf bifurcations are supercritical, so the bifurcating periodic orbits are stable.
Apparently when $ 0 < R < R^* $, the bare-soil equilibrium $ E_0 $ is the attractor;
when $ R^* < R < R_1 $, $ E_0 $ becomes unstable, and the system has a unique positive equilibrium $ E_{31}^* $, which is stable (see Fig. \ref{fig:5.1.4} (b)).
As $ R $ increases to $ R_1 < R < R_{LP}^1 $, $ E_{31}^* $ becomes unstable, and a stable limit cycle appears (see Fig. \ref{fig:5.1.4} (c)).
When $ R_{LP}^1 < R < R_2 $, the system has three positive equilibria: $ E_{31}^* $, $ E_{32}^* $, and $ E_{33}^* $, where $ E_{31}^* $ and $ E_{32}^* $ are unstable.
The system (\ref{eq2.1}) admits the bistability of the positive equilibrium $ E_{33}^* $ and a stable limit cycle with the stable manifolds (green orbits) of the saddle $ E_{32}^* $ as the separatrix of attraction (see Fig. \ref{fig:5.1.4} (d)).
As $ R $ increases to $R_2  < R < R_{LP}^2 $, the system shows bistability between $ E_{31}^* $ and $ E_{33}^* $, and the boundary is still determined by the stable manifold of the saddle (see Fig. \ref{fig:5.1.4} (e)).
When $ R > R_{LP}^2 $, the system is monostable, with all solutions tending to the positive equilibrium $ E_{33}^* $ (see Fig. \ref{fig:5.1.4} (f)).



\begin{figure}[!h]
	\centering
	\subfigure[]{\includegraphics[width=0.35\textwidth]{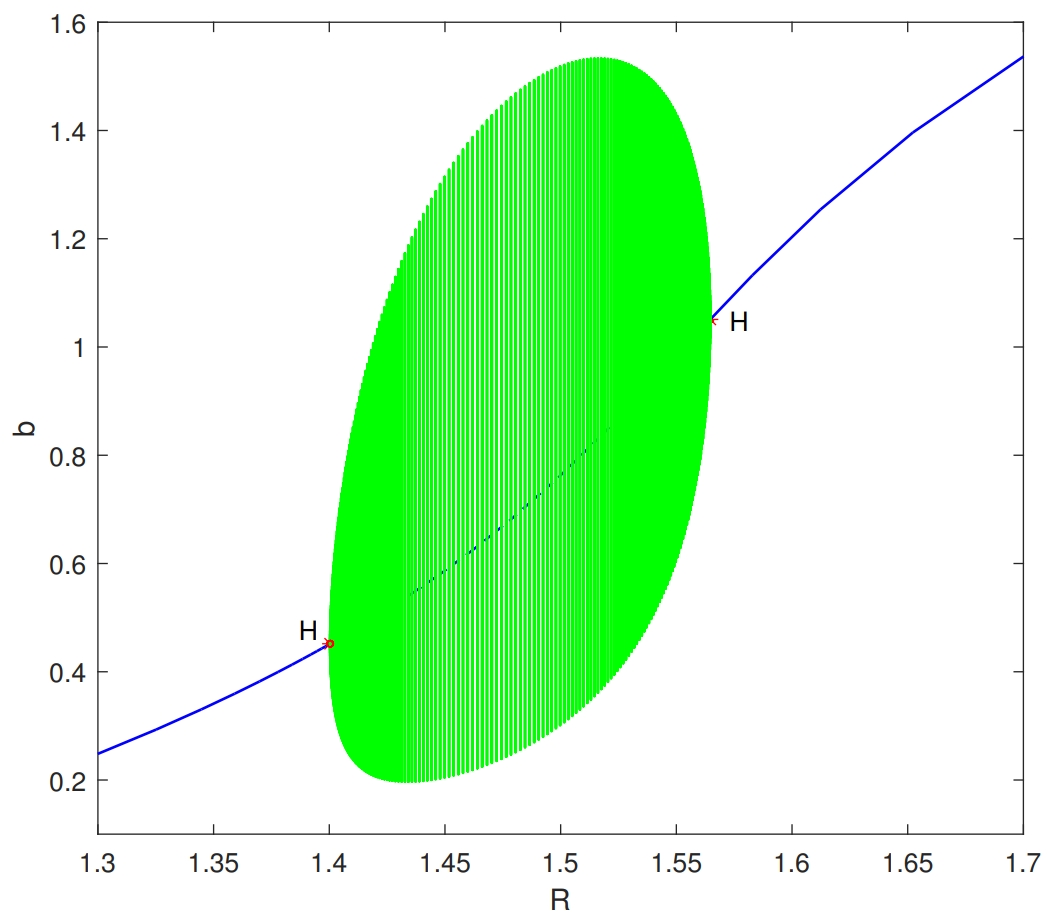}} \hspace{5pt}
	\subfigure[]{\includegraphics[width=0.34\textwidth]{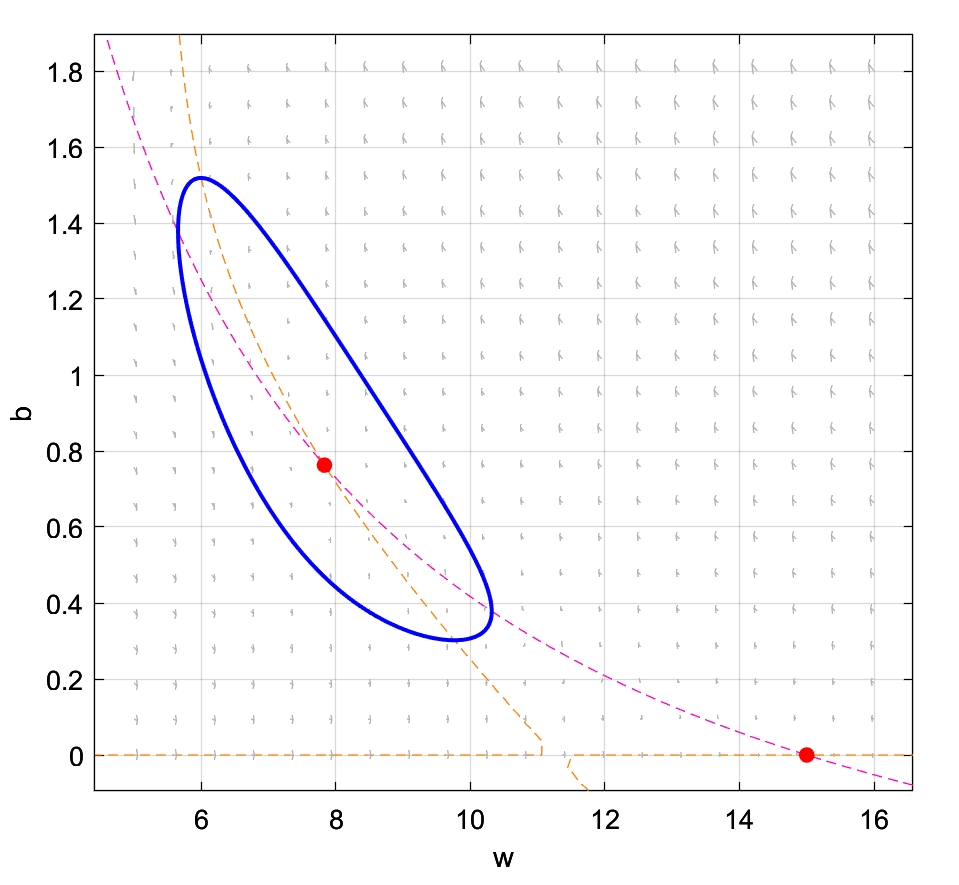}}
	\subfigure[]{\includegraphics[width=0.35\textwidth]{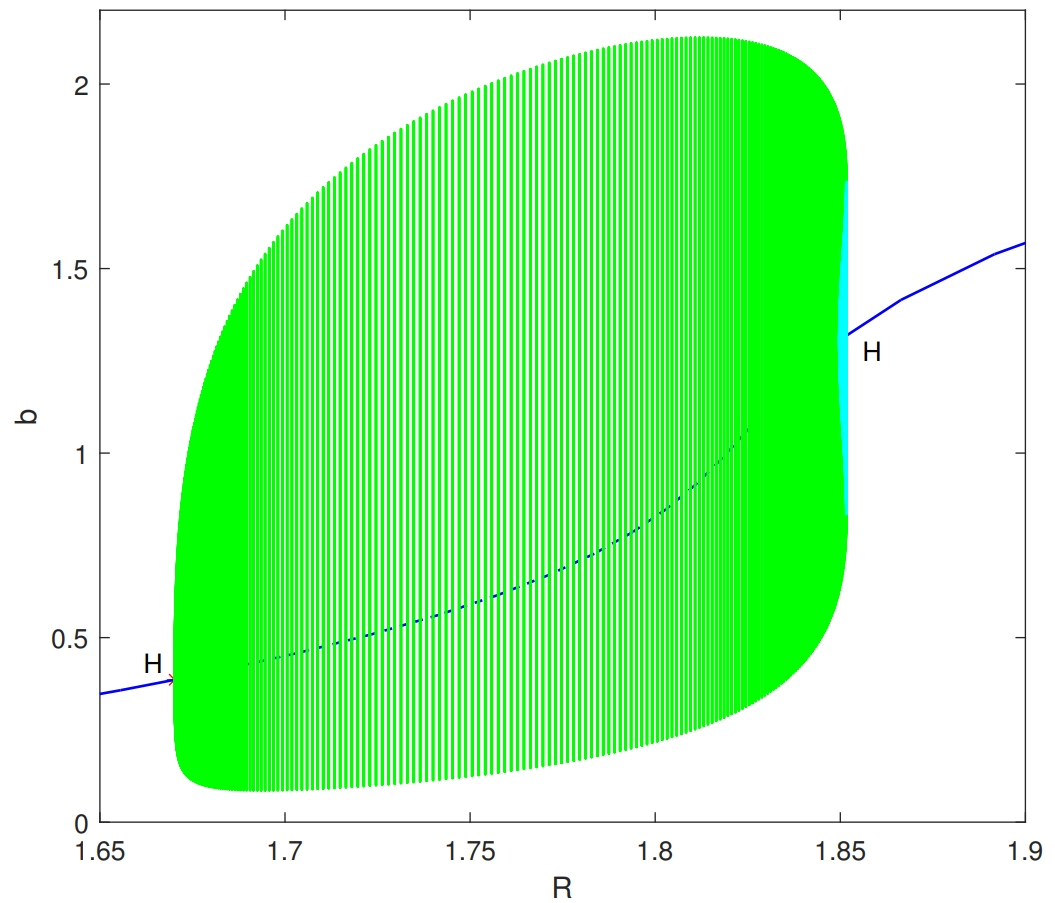}} \hspace{5pt}
	\subfigure[]{\includegraphics[width=0.34\textwidth]{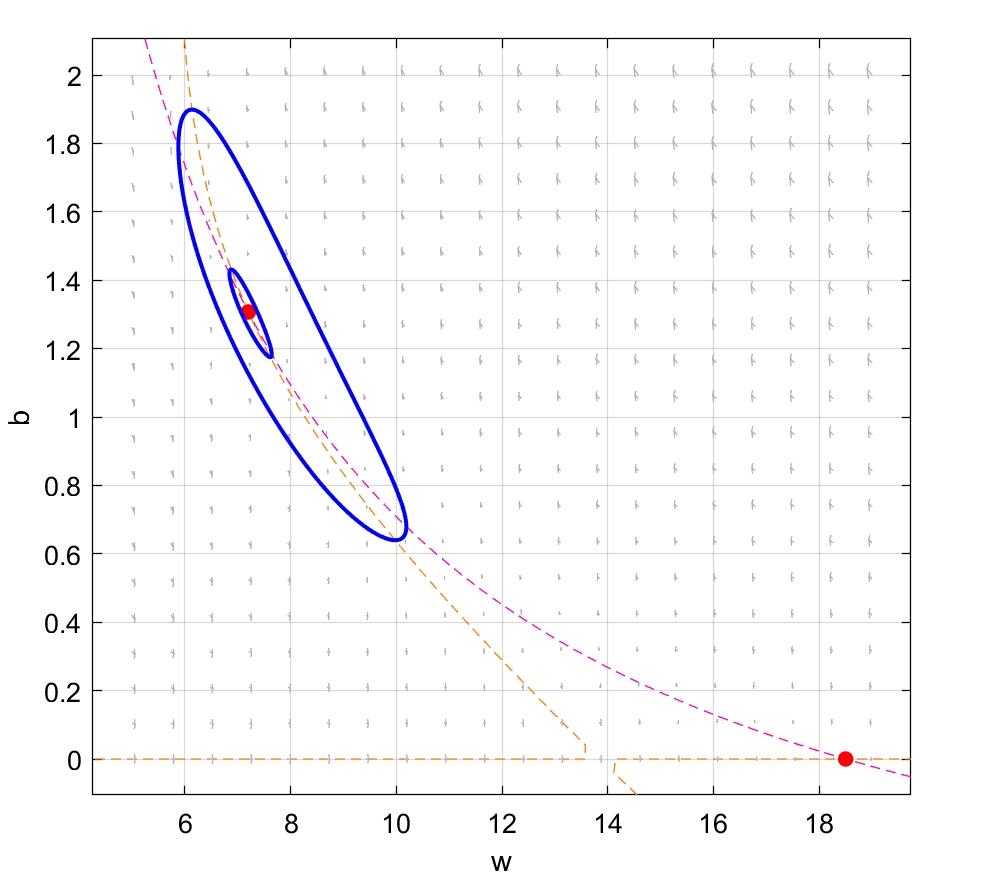}}
	\caption{Bifurcation diagrams and phase portraits of limit cycles. In (a) and (c), the blue curve represents the biomass $b$ as a function of the rainfall rate $R$, while the green curve denotes the stable limit cycle. The corresponding phase portraits are shown in (b) and (d). In (a) and (b), $\theta_2=0.8 $ with $ R = 1.5 $ in (b); in (c) and (d), $ \theta_2=0.65 $ with $ R = 1.85 $ in (d). Other parameters: $
a=0.1,\ \delta=0.12,\ \rho=1,\ \mu=10,\ \theta_1=5$.}
	\label{fig:5.1.1}
\end{figure}

Next,  we  fix the parameters as
$
a=0.1,\ \delta=0.12,\ \rho=1,\ \mu=10,\ \theta_1=5$ and vary $\theta_2, R$
 to investigate  the dynamics of  system (\ref{eq2.1}). For this case, system (\ref{eq2.1}) has a unique positive equilibrium $E_1^*$ and two Hopf bifurcation points  may exist.

For $\theta_2=0.8$, two Hopf bifurcation points occur at $ R = R_1 $ and $ R = R_2 $.
A bounded branch of limit cycles emerges from the equilibrium $ E^* $ at $ R = R_1 $ and eventually vanishes at
$ R = R_2 $, forming a closed-loop structure, or a \lq\lq bubble".
The bifurcation diagram and the corresponding phase portrait of the limit cycles are presented in Fig. \ref{fig:5.1.1} (a) and (b). Moreover, numerical calculations show that the first Lyapunov coefficients at both bifurcation points are negative, indicating that the bifurcating periodic orbits are stable.

For $\theta_2=0.65$, the system also exhibits two Hopf bifurcation points, as illustrated in Fig. \ref{fig:5.1.1} (c). However, unlike in Fig. \ref{fig:5.1.1} (a), the Hopf bifurcation at $ R = R_1 $ is supercritical, whereas at $ R = R_2 $, it is subcritical. The branch of limit cycles still forms a closed loop, but its shape resembles a \lq\lq heart'' due to a saddle-node bifurcation of periodic orbits occurring at $ R = R_3 > R_2 $. This bifurcation results in the coexistence of two periodic orbits within the range $ R_2 < R < R_3 $.
The corresponding phase portrait is shown in Fig. \ref{fig:5.1.1} (d), where the larger limit cycle is stable, while the smaller one is unstable.
Consequently, the system exhibits bistability, with both a stable equilibrium $ E_1^* $ and a stable limit cycle coexisting for $ R_2 < R < R_3 $.

\begin{figure}[!h]
	\centering
	\subfigure[$\theta_1=15,\, \theta_2=3.95$]{\includegraphics[width=0.32\textwidth]{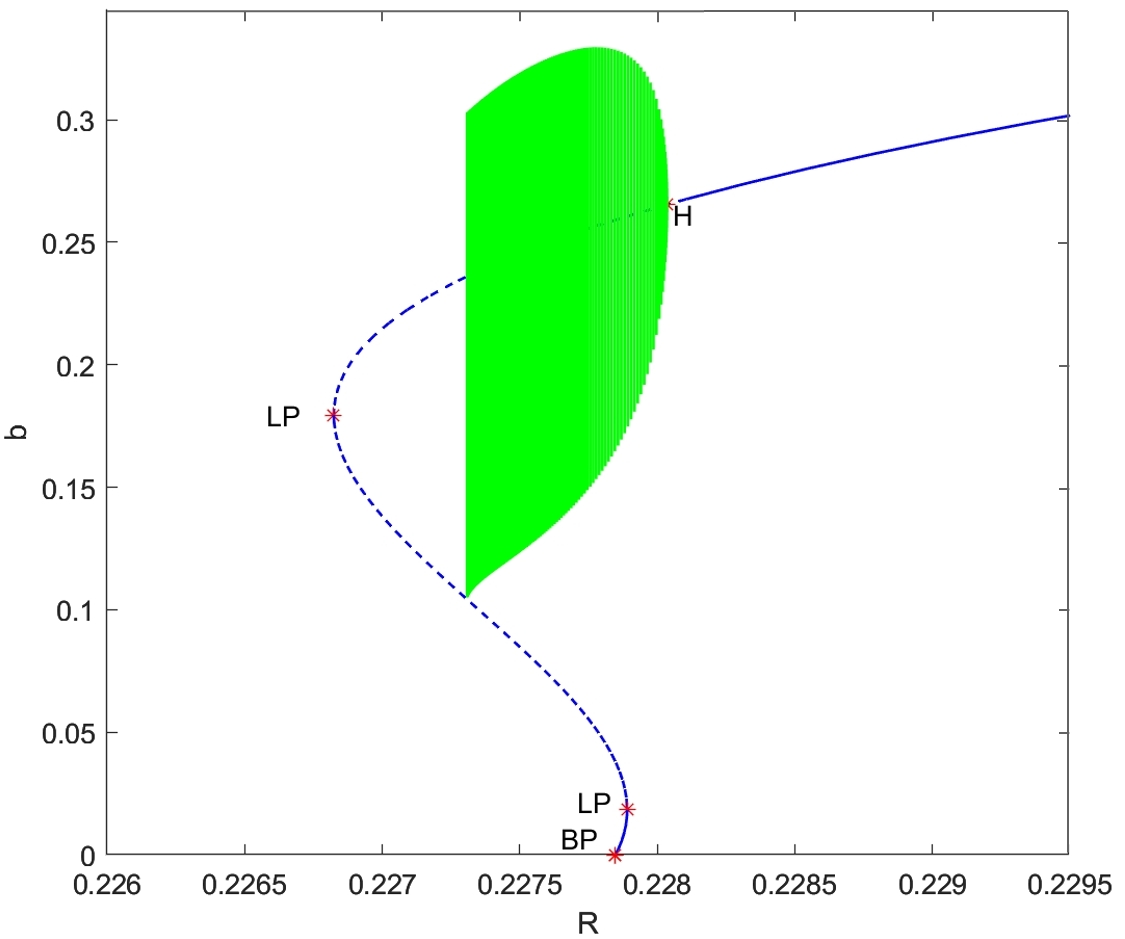}} \hspace{0.5pt}
\subfigure[$\theta_1=15,\, \theta_2=4$]{\includegraphics[width=0.32\textwidth]{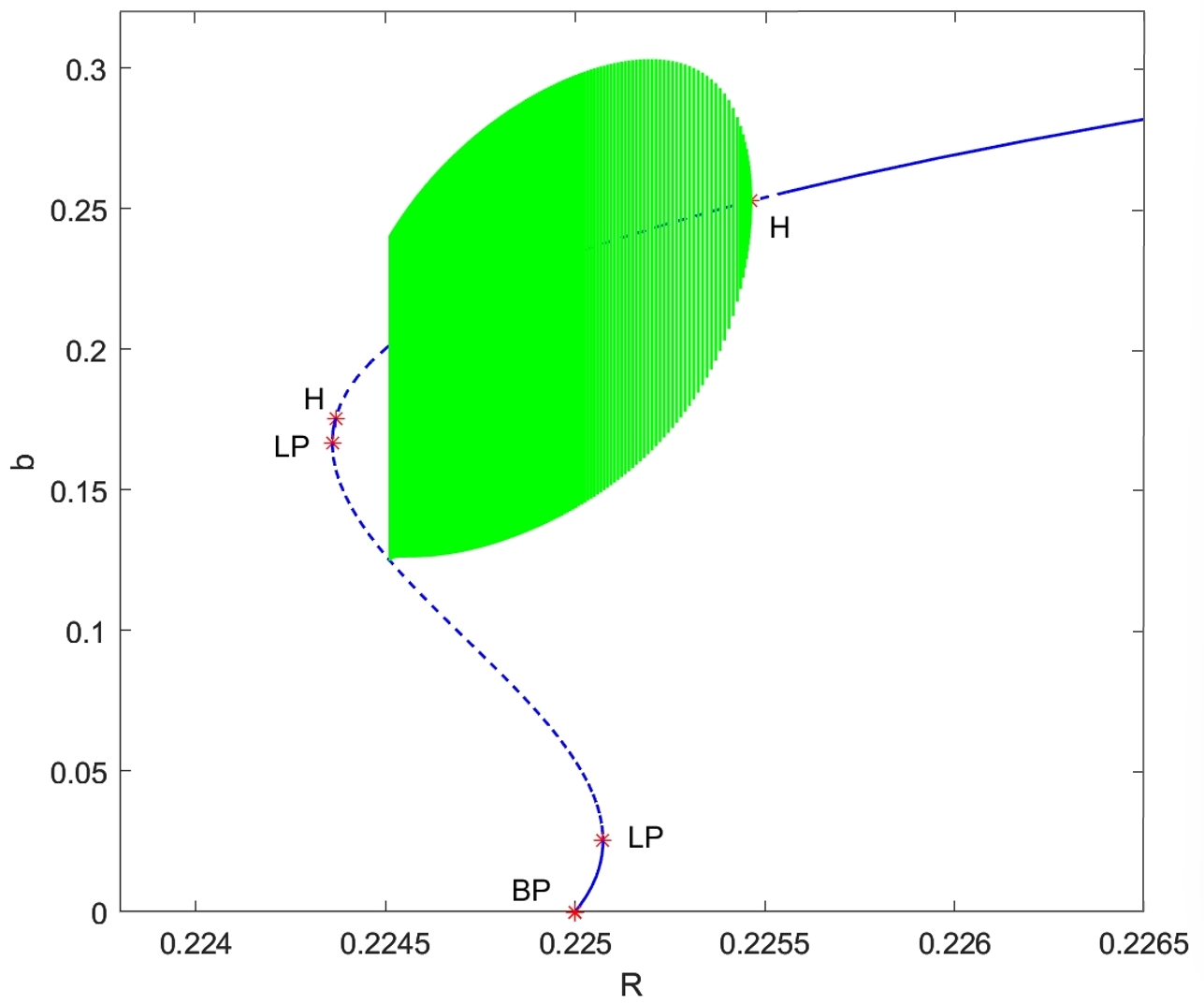}} \hspace{0.5pt}
	\subfigure[$\theta_1=5,\, \theta_2=0.55$]{\includegraphics[width=0.308\textwidth]{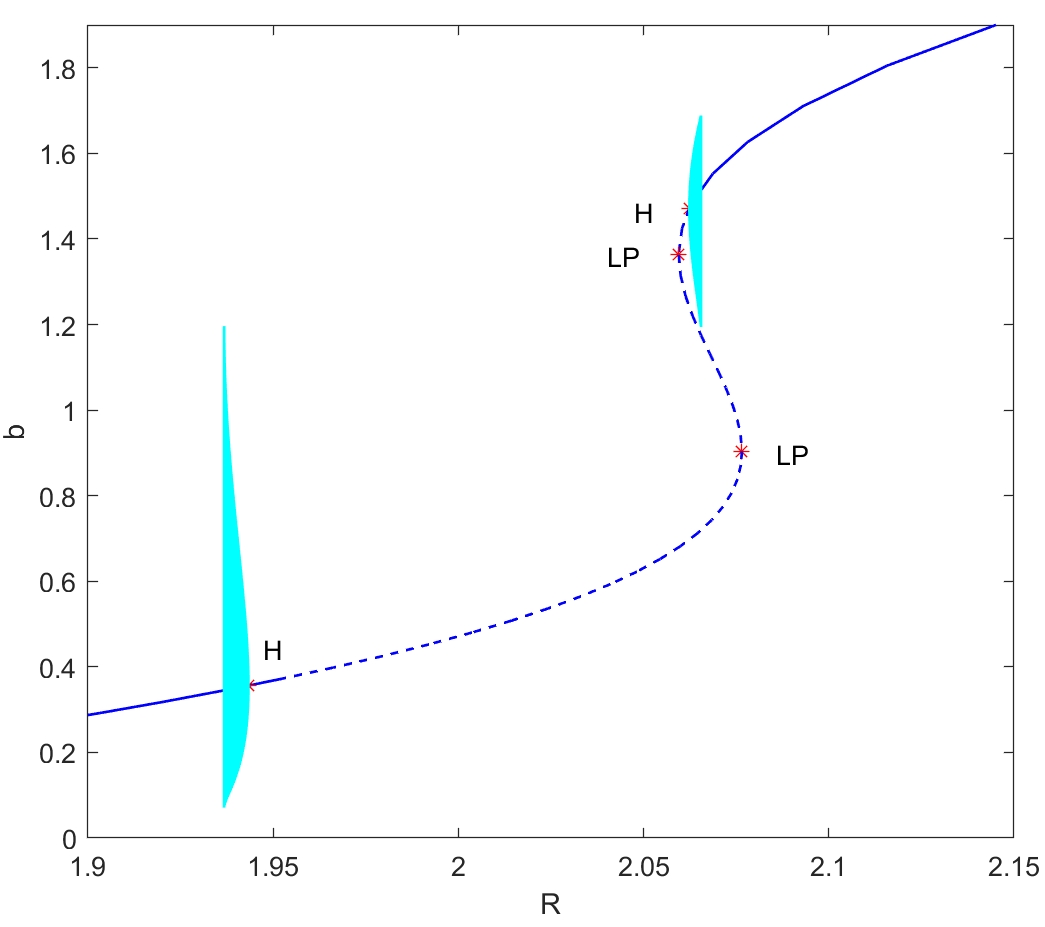}}
	\caption{Bifurcation diagrams of Hopf bifurcation at  $E_{33}^*$. Other parameters: $
a=0.1,\ \delta=0.12,\ \rho=1,\ \mu=10$.}
	\label{fig:5.1.7}
\end{figure}

Finally, we note that,  when system (\ref{eq2.1}) admits three positive equilibria, Hopf bifurcations could also occur  at $ E_{33}^* $, and there may exist one or two Hopf bifurcation points which may be subcritical or
supercritical resulting in   stable or unstable limit cycles, as shown in Fig. \ref{fig:5.1.7}.

\subsection{The spatiotemporal dynamics near the Turing$-$Hopf bifurcation point}
To analyze the effect of the diffusion coefficient $d_1$ on spatiotemporal patterns, we fix the parameters $a, \delta, \rho, \mu, \theta_1, \theta_2, l,$  $d_2$, and treat $R$ and $d_1$ as bifurcation parameters  for numerical simulations. Setting $\theta_1=5$ and $\theta_2=0.8$, we obtain the corresponding bifurcation diagram shown in Fig. \ref{fig:5.1.1} (a). As discussed earlier,
a Hopf bifurcation occurs at $R_1 = 1.400362$ and $R_2=1.564687$ such that $E^*$ is stable in $R_0<R<R_1$ or $R>R_2$ and unstable in $R_1<R<R_2$. In the following, we focus on the range $R\in(1.125, 2.07019)$, as a necessary condition for Turing instability is $a_{22}>0$, which holds for $R$ within this interval. Given $d_2=0.01$ and $l=3$, it follows from (\ref{eq3.5}) that $k_0 = 7$ for $1.37783 \leq R < 1.48432$ and $k_0 = 8$ for $1.48432 \leq R < 1.68646$. Consequently, the critical Turing bifurcation curve is one of $J_7 = 0, J_8 = 0$ and $J_9 = 0$. Further computations indicate that $d_1(R,7^2) \leq d_1(R,8^2)$ for $1.37783 \leq R \leq 1.42198$, $d_1(R,7^2) > d_1(R,8^2)$ for $1.42198 < R < 1.48432$, and $d_1(R,8^2) < d_1(R,9^2)$ for $1.48432 \leq R < 1.68646$. Thus, we obtain $k_*=7$ for $1.37783 \leq R \leq 1.42198$ and $k_*=8$ for $1.42198 < R < 1.68646$. The critical Turing bifurcation curve is defined as:
$$
\mathrm{T} :=
\begin{cases}
	J_7 = 0, & 1.37783 \leq R \leq 1.42198, \\
	J_8 = 0, & 1.42198 < R < 1.68646.
\end{cases}
$$

The Turing bifurcation curve $J_7 = 0$ intersects the Hopf bifurcation curve $R = R_1$ at $(R, d_1)=(1.400362, 0.04211)$, while $J_8 = 0$ intersects $R = R_2$ at $(R, d_1)=(1.564687, 0.03114)$. From Theorem \ref{thm3.3}, these intersections define two TH bifurcation points. Since their theoretical properties are similar, we focus on the TH bifurcation point $(R, d_1)=(R^*, d_1^*)$ with $R^*=R_2=1.564687$ and $d_1^*=0.03114$.

Combining the above analysis, we present these bifurcation curves in Fig. \ref{fig:5.3.1}(b). The solid red curve (T) represents the Turing bifurcation $J_8 = 0$, while the dotted black vertical lines $\mathrm{H}_1$ and $\mathrm{H}_2$ indicate the two Hopf bifurcation curves. These curves partition the plane into different regions, where in region $S$, the positive steady state is asymptotically stable, and the shadowed area represents the Turing instability region.

\begin{figure}[!h]
	\centering
	\centering
	\includegraphics[width=5cm,height=4.5cm]{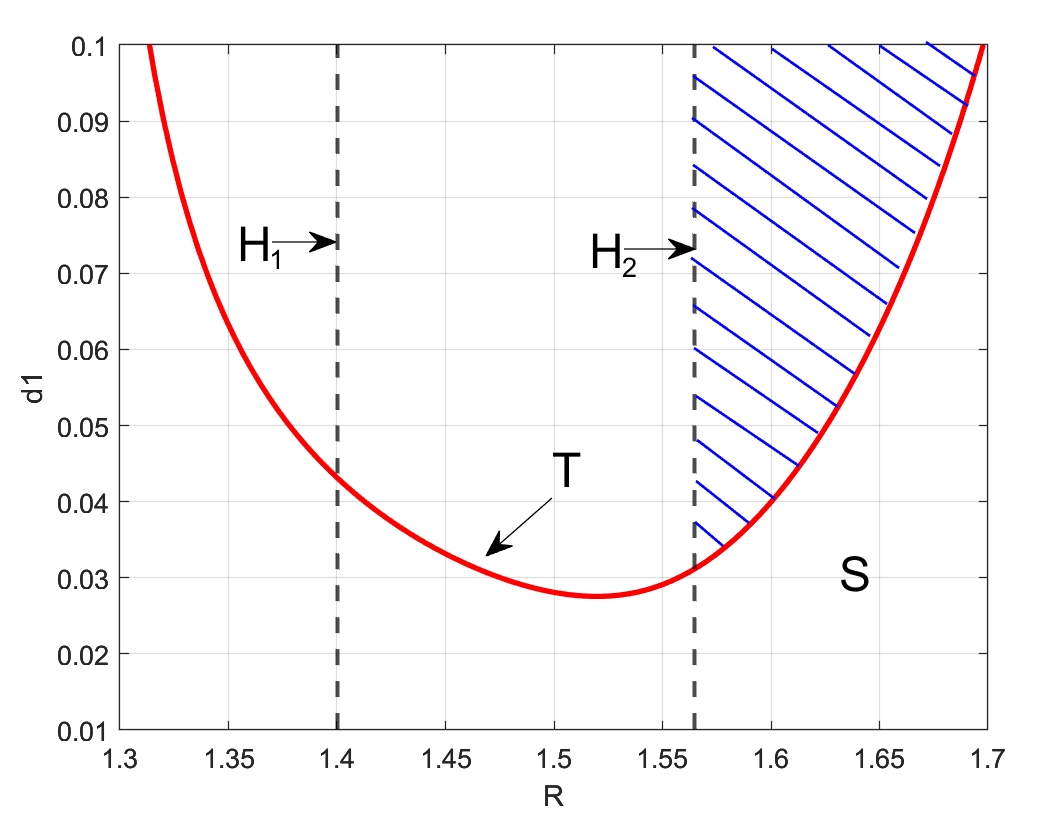}
	\caption{ Bifurcation diagram for the diffusive water-vegetation model (\ref{eq1.1}) with $\theta_1=5, \theta_2=0.8, d_2=0.01$ and $l=3$ in the $R-d_1$ plane. The curves
		marked by $\mathrm{H}_1$ and $\mathrm{H}_2$ are Hopf bifurcation curves, and the curve marked by $\mathrm{T}$ is the Turing bifurcation curve.
		The shadow region is the Turing instability region, and the region marked by S is the stable region with respect to $E^*$.}
	\label{fig:5.3.1}
\end{figure}

Next, we are curious about the dynamics of the system near the TH bifurcation point. When $(R, d_1)$ is chosen at the TH singularity $(R^*, d_1^*)$, we have $(w^*, b^*)=(6.91998, 1.05093)$. By direct calculation, the normal form truncated to
the third-order term is
\begin{equation}\label{eq5.1}
	\begin{cases}
		\dot{\rho}=-\rho (0.33423 \varepsilon_1 + \rho^2+1.11749 s^2),\\
		\dot{s}=s(-0.64121 \varepsilon_1+3.76777 \varepsilon_2-7.43065 \rho^2-s^2).
	\end{cases}
\end{equation}

Note that $\rho>0$. System (\ref{eq5.1}) exhibits the following equilibria
\begin{align*}
	&\tilde{A}_0 = (0,\,0),\\
	&\tilde{A}_1 = (\sqrt{-0.33423 \varepsilon_1}, \,0),\,\, \text{ for }\varepsilon_1 < 0,\\
	&\tilde{A}_2^{\pm} = (0,\, \pm\sqrt{-0.64121\varepsilon_1+3.76777\varepsilon_2}),\,\, \text{ for }-0.64121\varepsilon_1+3.76777\varepsilon_2 > 0,\\
	&\tilde{A}_3^{\pm} = (\sqrt{-0.05234 \varepsilon_1+0.57648 \varepsilon_2},\, \pm\sqrt{-(0.25225 \varepsilon_1+0.51587 \varepsilon_2)}), \\
	&\ \ \ \ \ \ \ \ \ \ \  \text{for} -0.05234 \varepsilon_1+0.57648 \varepsilon_2>0 \,\,\text{and} \,\,0.25225 \varepsilon_1+0.51587 \varepsilon_2<0.
\end{align*}

According to the existence and stability of these equilibria, the critical bifurcation line is
obtained as follows:
\begin{align*}
	\mathcal{H}_0 &: \varepsilon_1 = 0;&
	\mathcal{T} &: \varepsilon_2 = 0.17018\varepsilon_1;\\
	\mathcal{T}_1 &: \varepsilon_2 = 0.09079\varepsilon_1,\quad\varepsilon_1 < 0;&
	\mathcal{T}_2 &: \varepsilon_2 = -0.48898\varepsilon_1,\quad\varepsilon_1 < 0.
\end{align*}
These bifurcation lines divide the $\varepsilon_1- \varepsilon_2$ plane into six regions,  denoted as $D_i
, i = 1, 2\cdots\ , 6$ (see Fig. \ref{fig:5.3.2}(a)), and the dynamical
behavior of the system in each region is different (see Fig. \ref{fig:5.3.2}(b)). Moreover, the existence and stability properties of the steady states in the six regions are listed in Table \ref{tab1}, which is consistent with Fig.  \ref{fig:5.3.2}(b).
\begin{figure}[!b]
	\centering
	\subfigure[]{\includegraphics[width=0.44\textwidth]{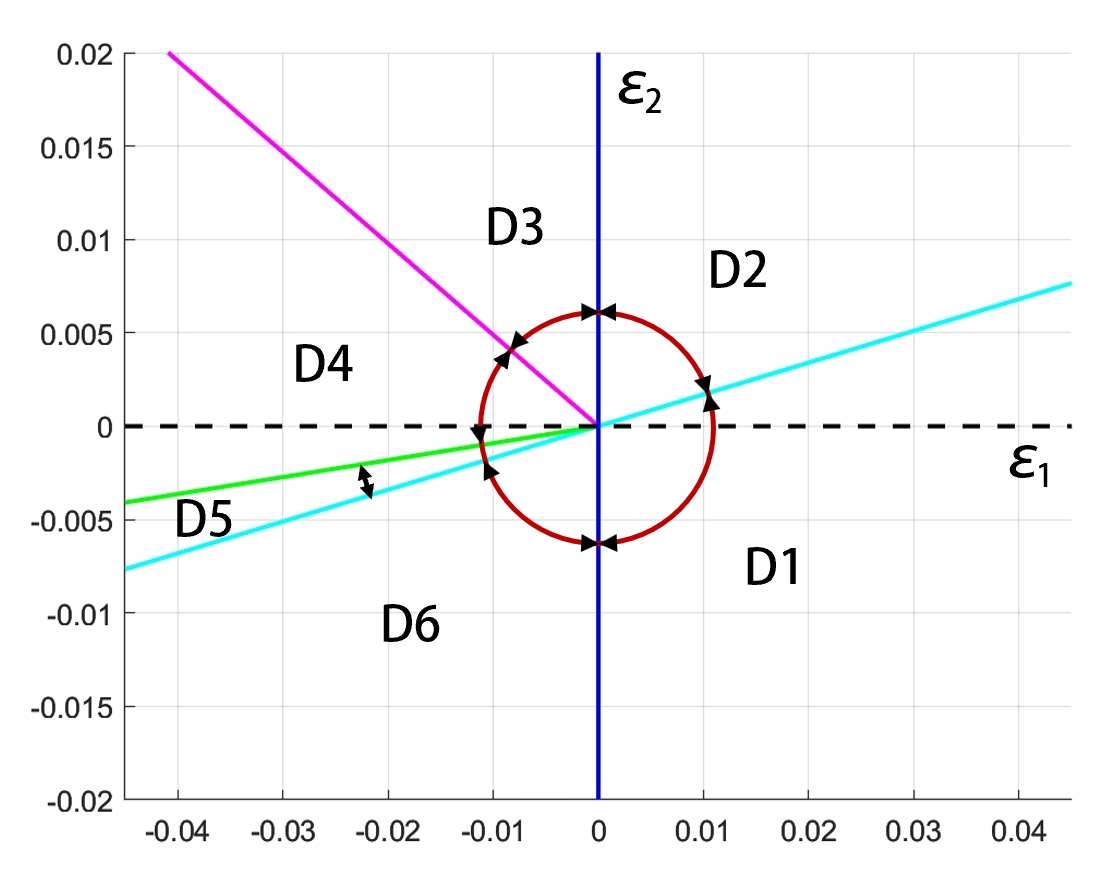}} \hfill
	\subfigure[]{\includegraphics[width=0.48\textwidth]{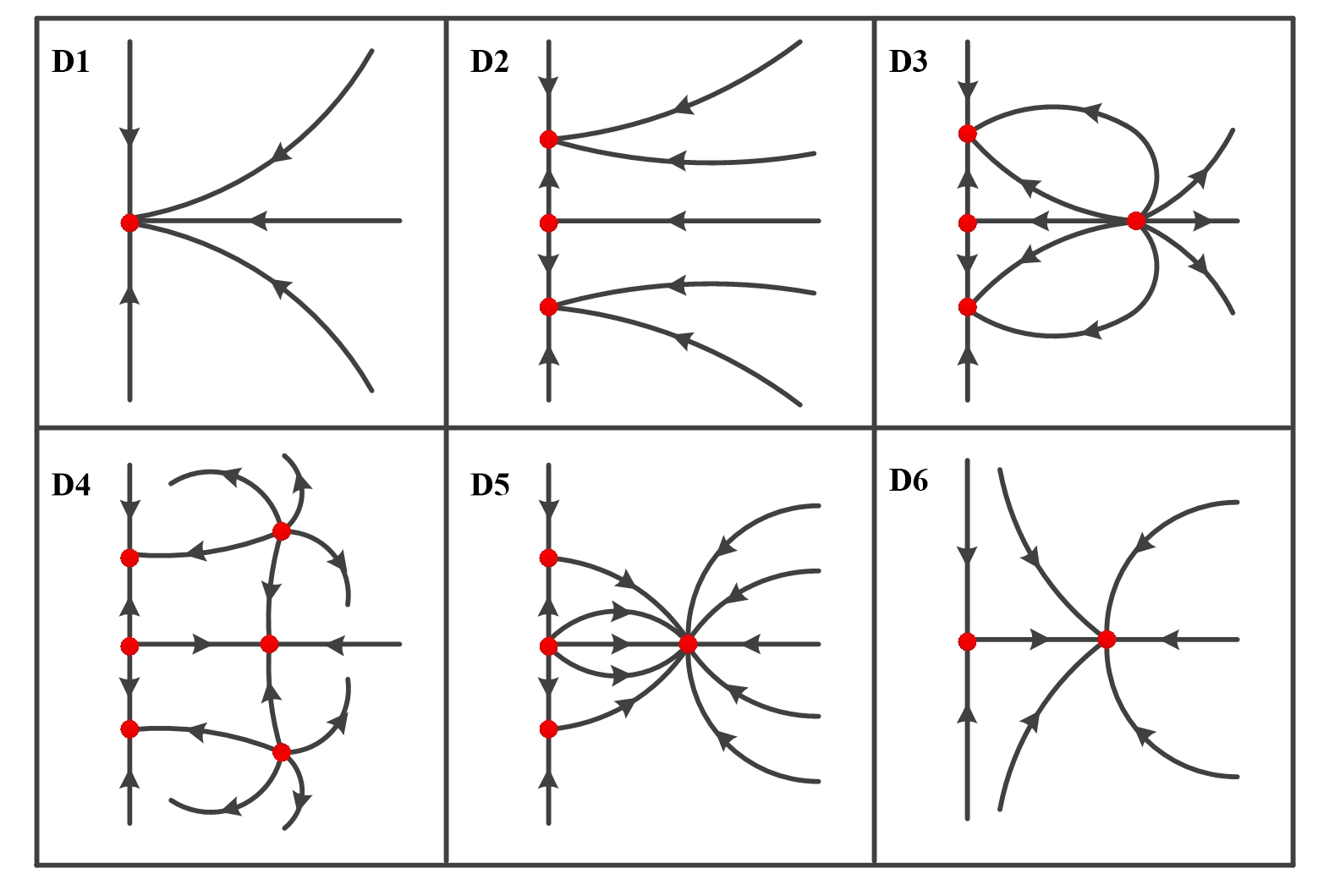}}
	\caption{Bifurcation diagram of the normal form system (\ref{eq5.1}) in the $\varepsilon_1-\varepsilon_2$ plane and the corresponding phase portraits.}
	\label{fig:5.3.2}
\end{figure}
\begin{table}[!h]
	\centering
	\caption{The stability of equilibria in each region.}
	\label{tab1}
	\begin{adjustbox}{width=0.8\textwidth} 
		\small 
		\begin{tabular}{lcccccc}
			\toprule
			\diagbox{regions}{equilibria} & $\tilde{A}_0$ & $\tilde{A}_1$ & $\tilde{A}_2^+$ & $\tilde{A}_2^-$ & $\tilde{A}_3^+$ & $\tilde{A}_3^-$ \\
			\midrule
			$D_1$ & stable & --- & --- & --- & --- & --- \\
			$D_2$ & unstable & --- & stable & stable & --- & --- \\
			$D_3$ & unstable & unstable & stable & stable & --- & --- \\
			$D_4$ & unstable & stable & stable & stable & unstable & unstable \\
			$D_5$ & unstable & stable & unstable & unstable & --- & --- \\
			$D_6$ & unstable & stable & --- & --- & --- & --- \\
			\bottomrule
			\multicolumn{7}{l}{--- indicates that the equilibrium point does not exist.} \\
		\end{tabular}
	\end{adjustbox}
\end{table}

As is well known, the equilibria $\tilde{A}_0, \tilde{A}_1, \tilde{A}_2^{\pm}$, and $\tilde{A}_3^{\pm}$ of system (\ref{eq5.1}) correspond to the constant equilibria, the spatially homogeneous periodic solution, the spatially inhomogeneous steady states, and the spatially inhomogeneous periodic solution of system (\ref{eq1.1}), respectively. Therefore, the dynamical behavior of system (\ref{eq1.1}) near the TH singularity in the $R-d_1$ plane can be determined by the dynamical behavior of system (\ref{eq5.1}).

\begin{figure}[!b]
	\centering
	\subfigure[]{\includegraphics[width=0.49\textwidth]{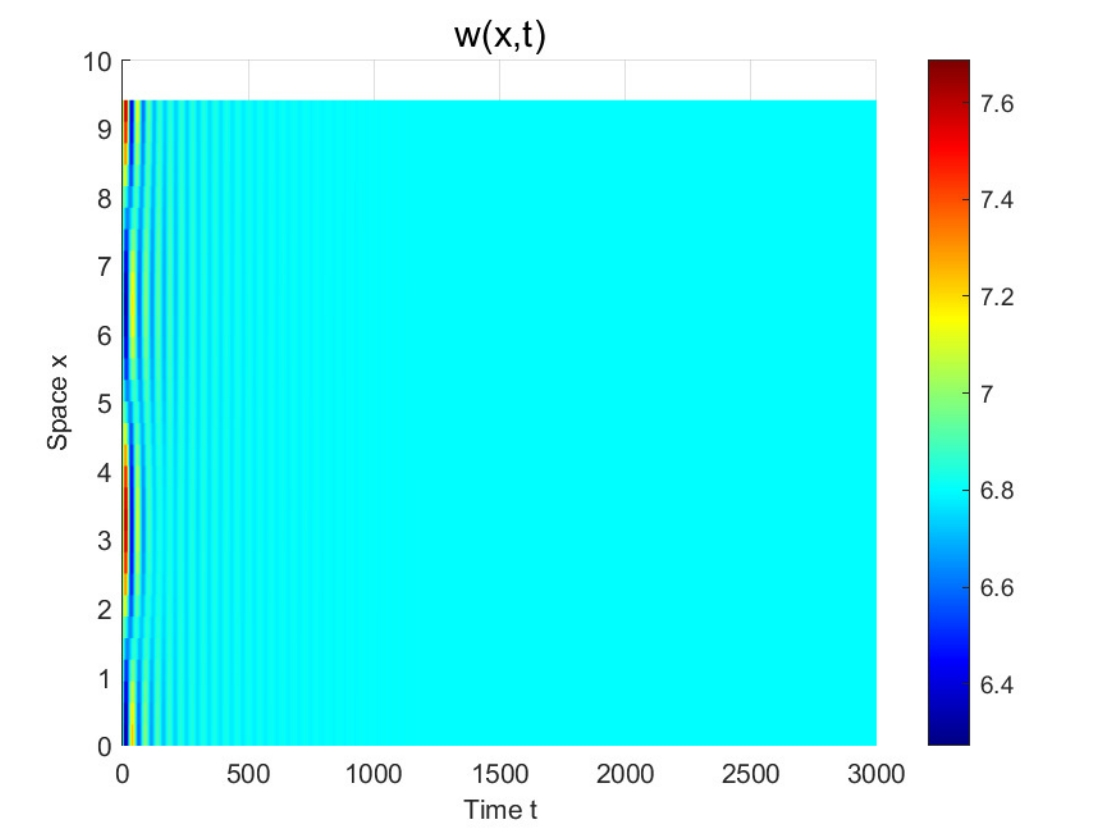}} \hfill
	\subfigure[]{\includegraphics[width=0.49\textwidth]{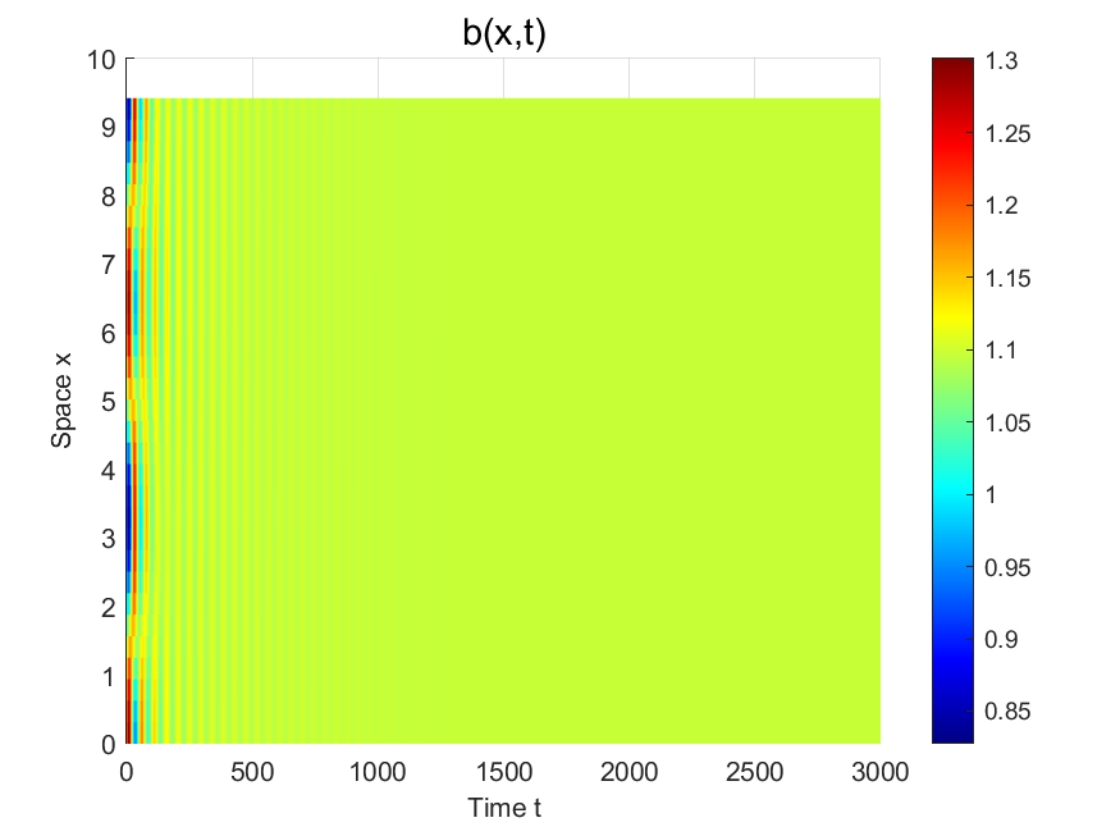}}
	\caption{For $(\varepsilon_1, \varepsilon_2)= (0.01, -0.005) \in D_1$, the positive constant steady state $E^*$ in system (\ref{eq1.1}) is asymptotically stable. The initial conditions are set as $w(x, 0)=6.91998+0.1\text{cos} x,\, b(x, 0)=1.05093+0.1\text{cos} x$.}
	\label{fig:5.3.3}
\end{figure}
When $(\varepsilon_1, \varepsilon_2) \in D_1$, the normal form system (\ref{eq5.1}) has a stable equilibrium $\tilde{A}_0$, indicating that the spatially homogeneous steady state $E^*$ of the diffusive water-vegetation model system (\ref{eq1.1}) is asymptotically stable, as illustrated in Fig. \ref{fig:5.3.3}.

\begin{figure}[!h]
	\centering
	\subfigure[]{\includegraphics[width=0.49\textwidth]{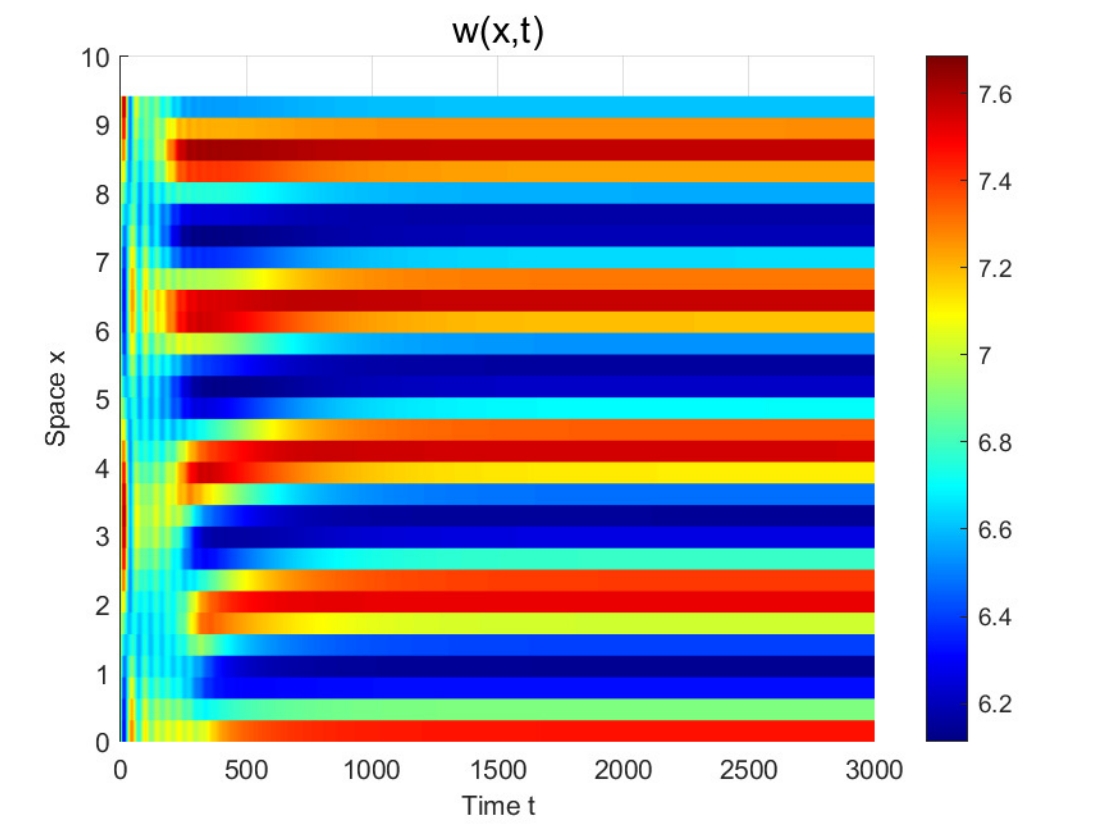}} \hfill
	\subfigure[]{\includegraphics[width=0.49\textwidth]{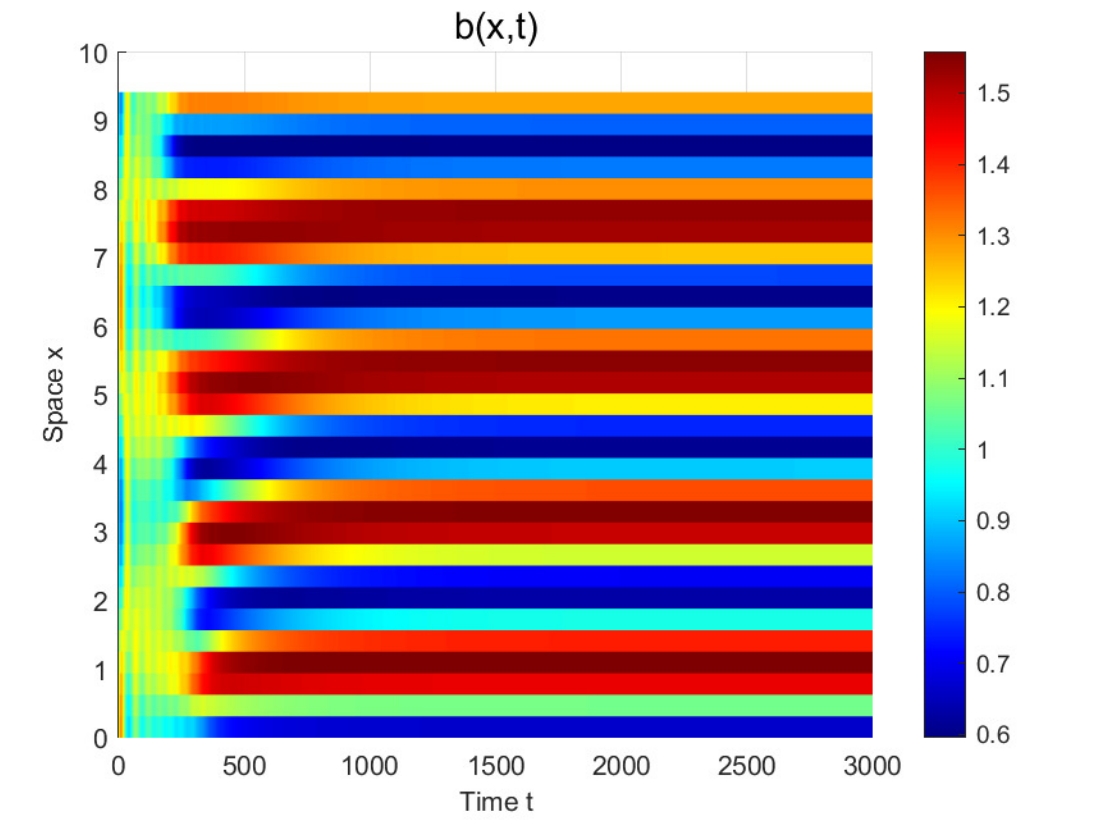}}
	\caption{When $(\varepsilon_1, \varepsilon_2)= (0.01, 0.01) \in D_2$, the spatially inhomogeneous steady states are stable. The initial conditions are set as $w(x, 0)=6.91998+0.1\text{cos} x,\, b(x, 0)=1.05093+0.1\text{cos} x$.}
	\label{fig:5.3.4}
\end{figure}
When $(\varepsilon_1, \varepsilon_2) \in D_2$, system (\ref{eq5.1}) has three equilibria: $\tilde{A}_0$ and $\tilde{A}_2^{\pm}$. Among them, $\tilde{A}_2^+$ and $\tilde{A}_2^-$ are stable, indicating that system (\ref{eq1.1}) possesses two stable spatially inhomogeneous steady states (called \lq\lq Turing pattern"), as illustrated in Fig. \ref{fig:5.3.4}.

\begin{figure}[!h]
	\centering
	\subfigure[]{\includegraphics[width=0.49\textwidth]{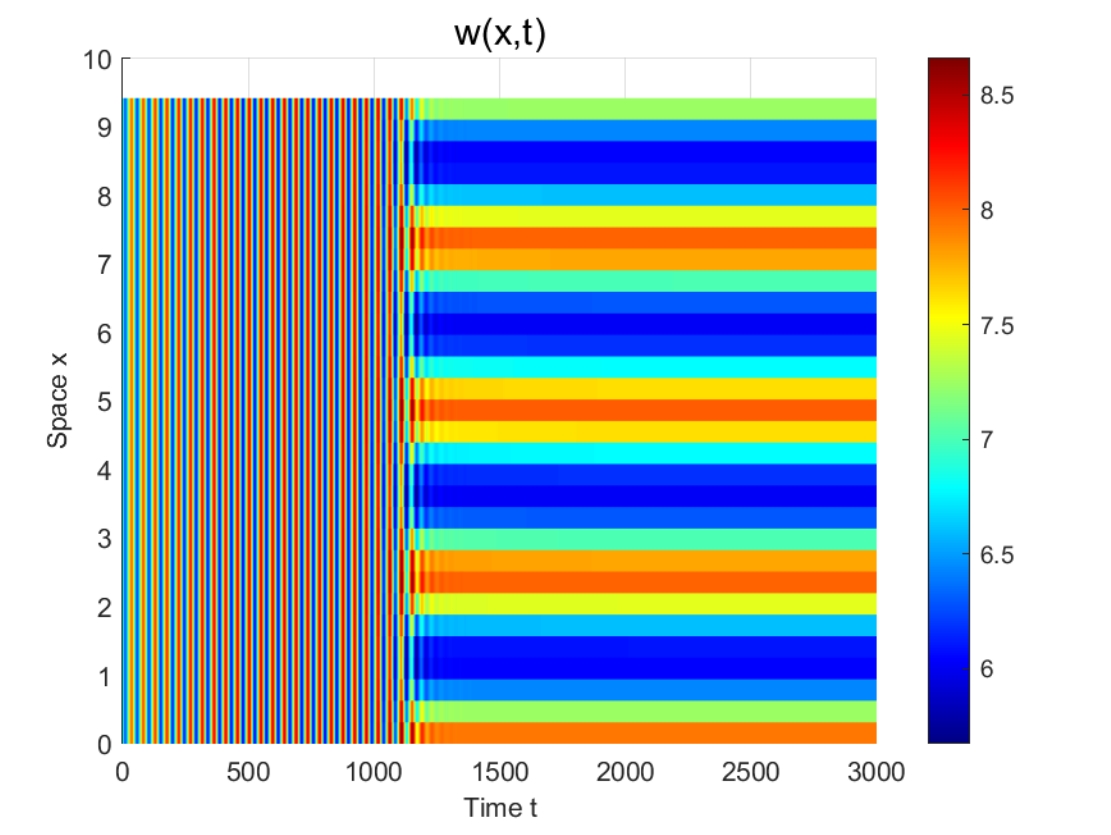}} \hfill
	\subfigure[]{\includegraphics[width=0.49\textwidth]{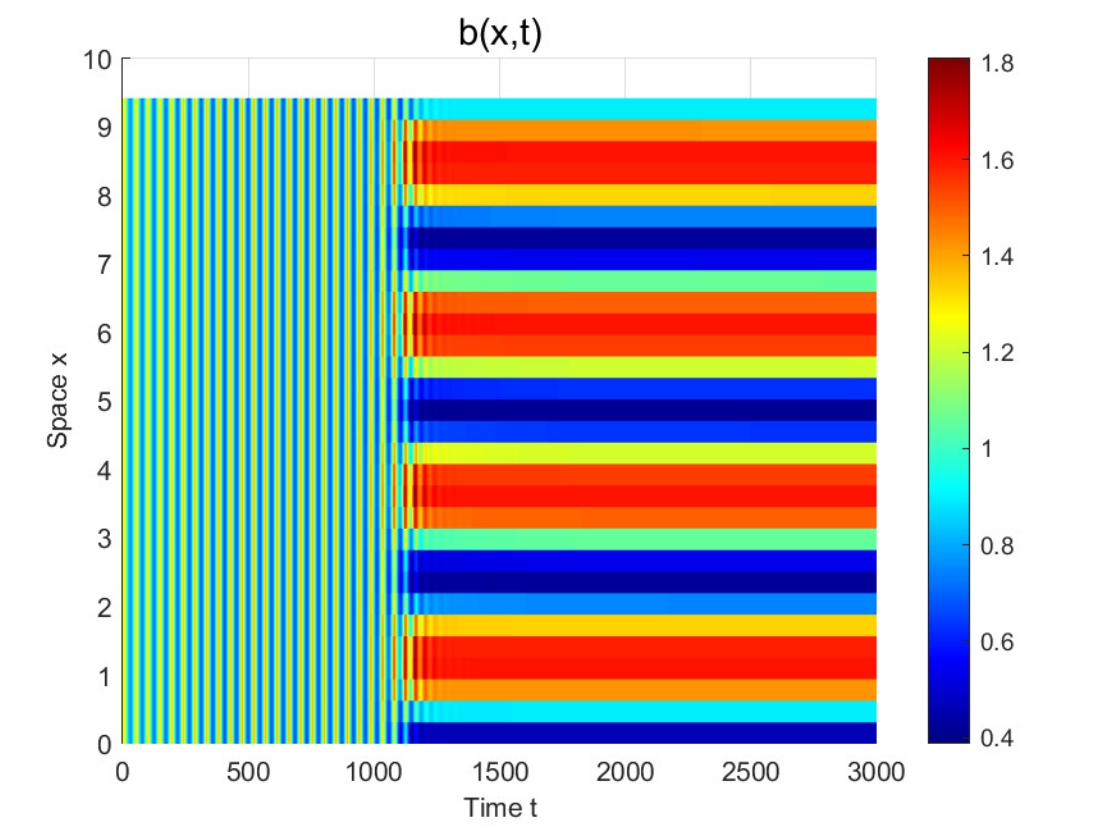}}
	\caption{When $(\varepsilon_1, \varepsilon_2)= (-0.01, 0.015) \in D_3$, system (\ref{eq1.1}) exhibits a pair of stable spatially inhomogeneous steady states alongside an unstable spatially homogeneous periodic solution. The initial values are set as $w(x, 0) = 6.91998 + 0.1$, $b(x, 0) = 1.05093 + 0.1$.}
	\label{fig:5.3.5}
\end{figure}
When $(\varepsilon_1, \varepsilon_2) \in D_3$, system (\ref{eq5.1}) has four equilibria: $\tilde{A}_0$, $\tilde{A}_1$, and $\tilde{A}_2^{\pm}$. While $\tilde{A}_2^+$ and $\tilde{A}_2^-$ are stable, $\tilde{A}_1$ is unstable. Consequently, system (\ref{eq1.1}) exhibits a pair of stable spatially inhomogeneous steady states along with an unstable spatially homogeneous periodic solution, as illustrated in Fig. \ref{fig:5.3.5}. This figure highlights a pattern transition, where unstable spatially homogeneous periodic solutions evolve into stable spatially inhomogeneous steady states.

\begin{figure}[!h]
	\centering
	\subfigure[]{\includegraphics[width=0.49\textwidth]{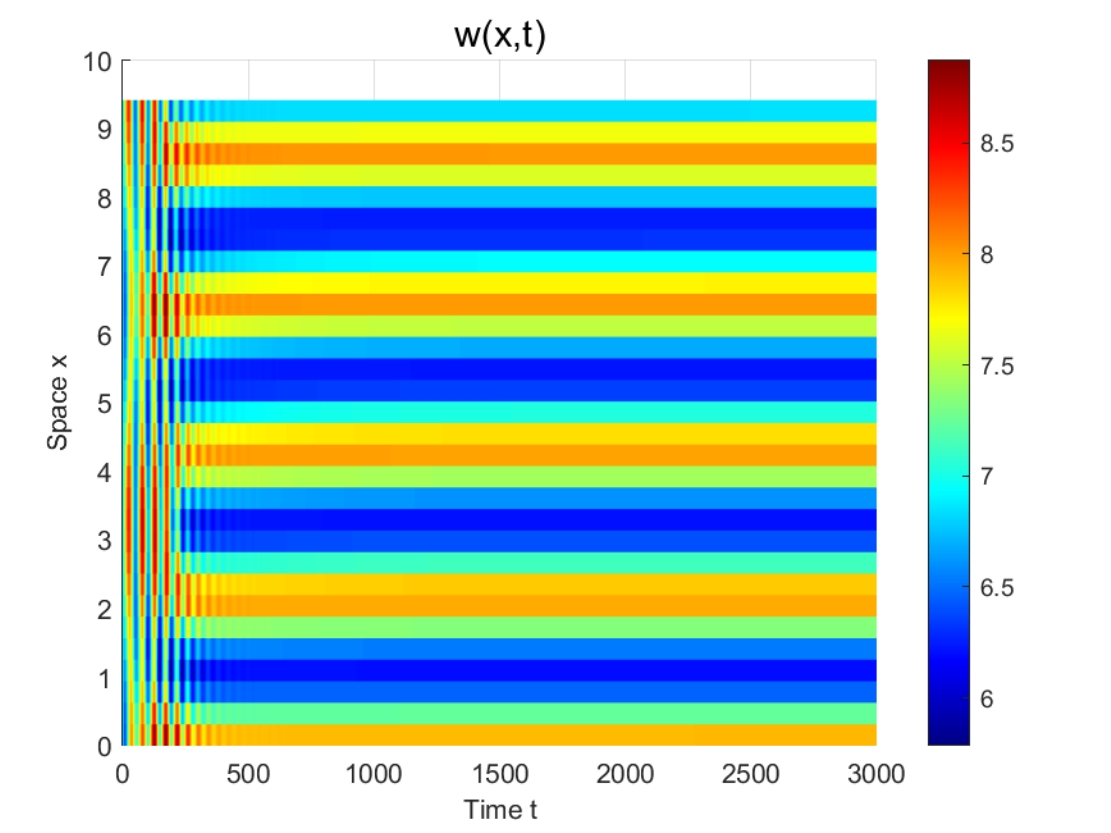}} \hfill
	\subfigure[]{\includegraphics[width=0.49\textwidth]{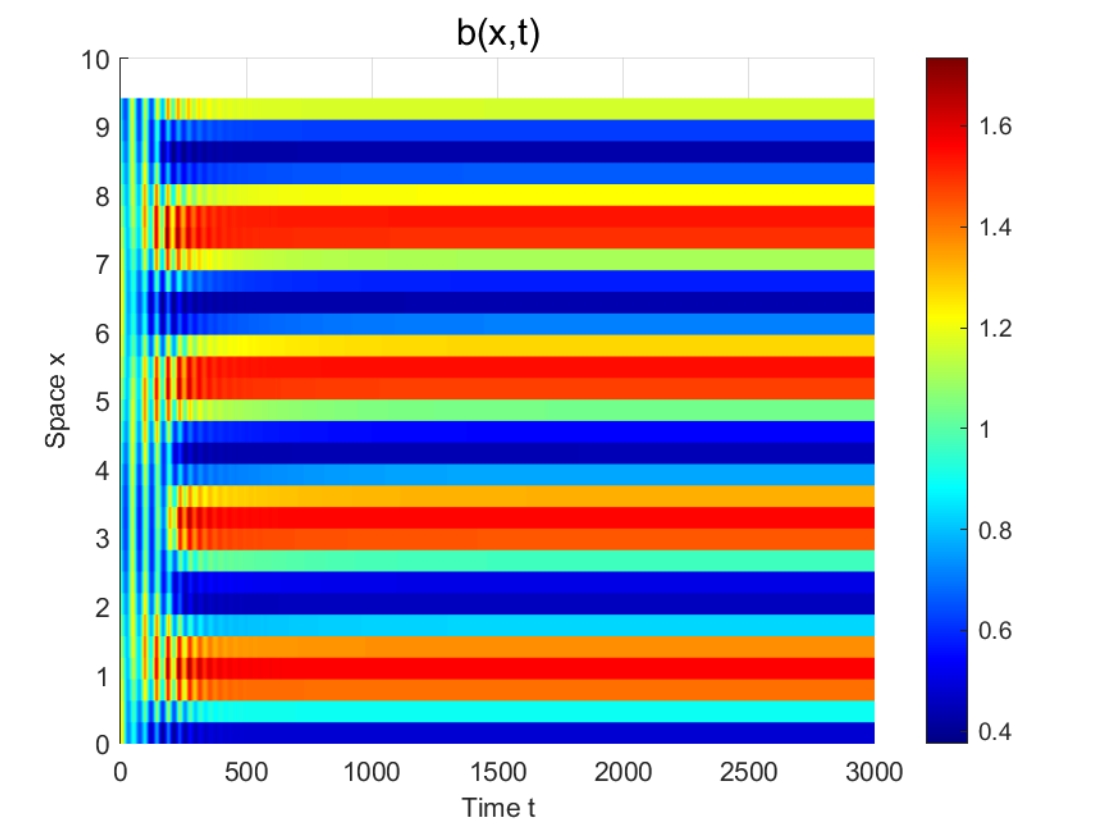}}
	\subfigure[]{\includegraphics[width=0.49\textwidth]{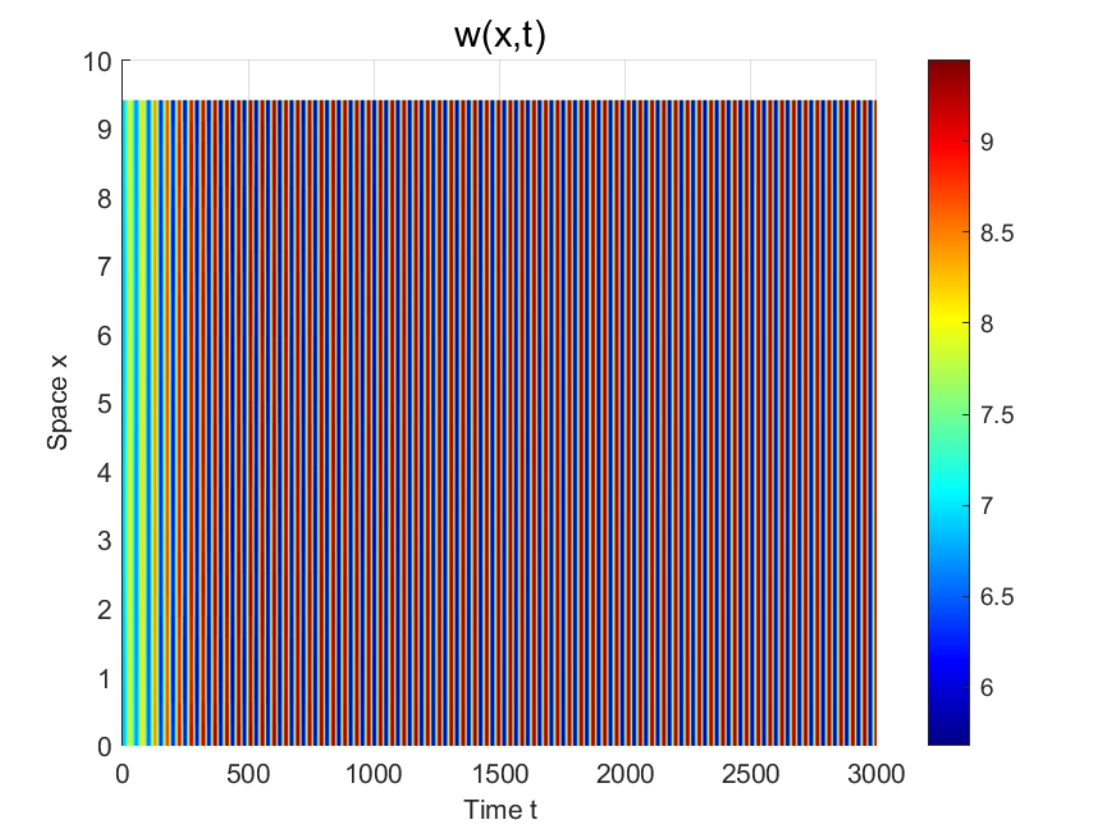}} \hfill
	\subfigure[]{\includegraphics[width=0.49\textwidth]{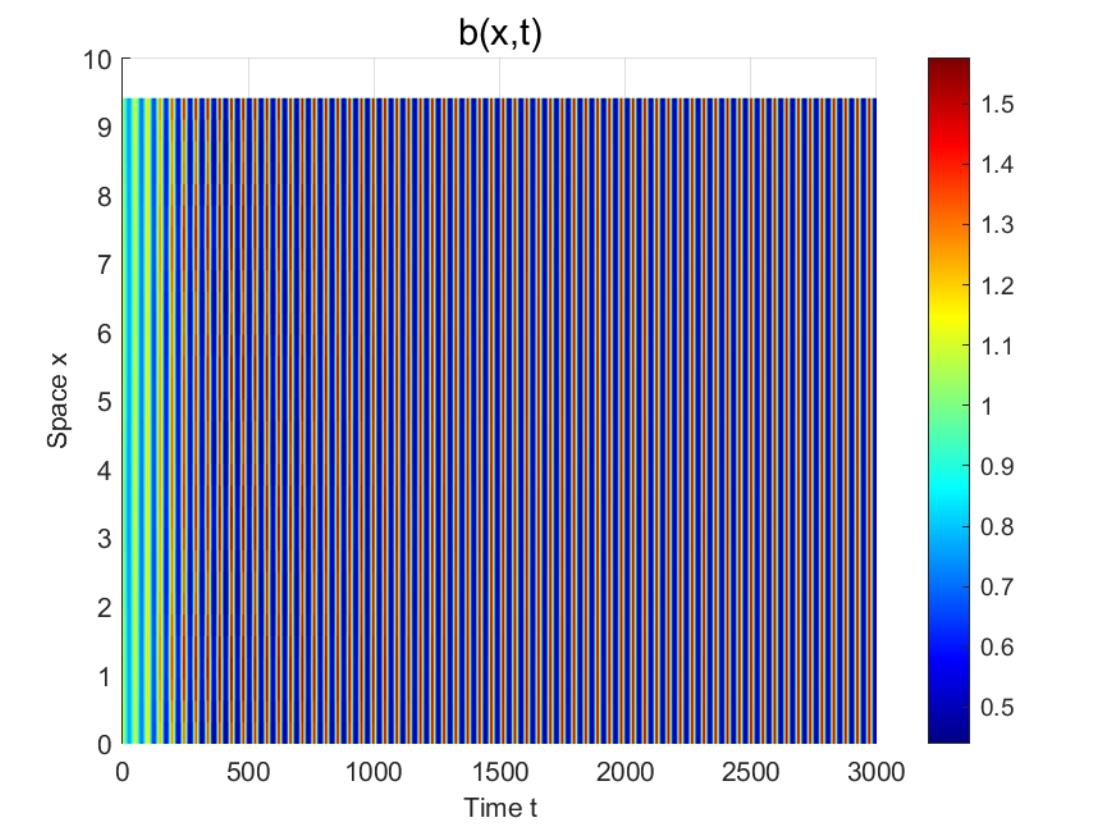}}
	\caption{When $(\varepsilon_1, \varepsilon_2)= (-0.03, 0.01) \in D_4$, 	there exists a stable spatially homogeneous periodic solution and a pair of stable spatially inhomogeneous steady states, as well as a pair of unstable spatially inhomogeneous periodic solutions. The initial values are: in (a) and (b), $w(x, 0) = 6.91998 + 0.1\cos x$, $b(x, 0) = 1.05093 + 0.1\cos x$; in (c) and (d), $w(x, 0) = 6.91998 + 0.01\cos x$, $b(x, 0) = 1.05093 + 0.01\cos x$.}
	\label{fig:5.3.6}
\end{figure}
When $(\varepsilon_1, \varepsilon_2) \in D_4$, system (\ref{eq5.1}) has six equilibria: $\tilde{A}_0, \tilde{A}_1, \tilde{A}_2^{\pm}$ and $\tilde{A}_3^{\pm}$. While $\tilde{A}_1, \tilde{A}_2^+$ and $\tilde{A}_2^-$ are stable, $\tilde{A}_3^+$ and $\tilde{A}_3^-$ are unstable. Consequently, system (\ref{eq1.1}) has a stable spatially homogeneous
periodic solution, a pair of stable spatially inhomogeneous steady states, and a pair of unstable spatially inhomogeneous periodic solutions, as illustrated in Fig. \ref{fig:5.3.6}.

\begin{figure}[!h]
	\centering
	\subfigure[]{\includegraphics[width=0.49\textwidth]{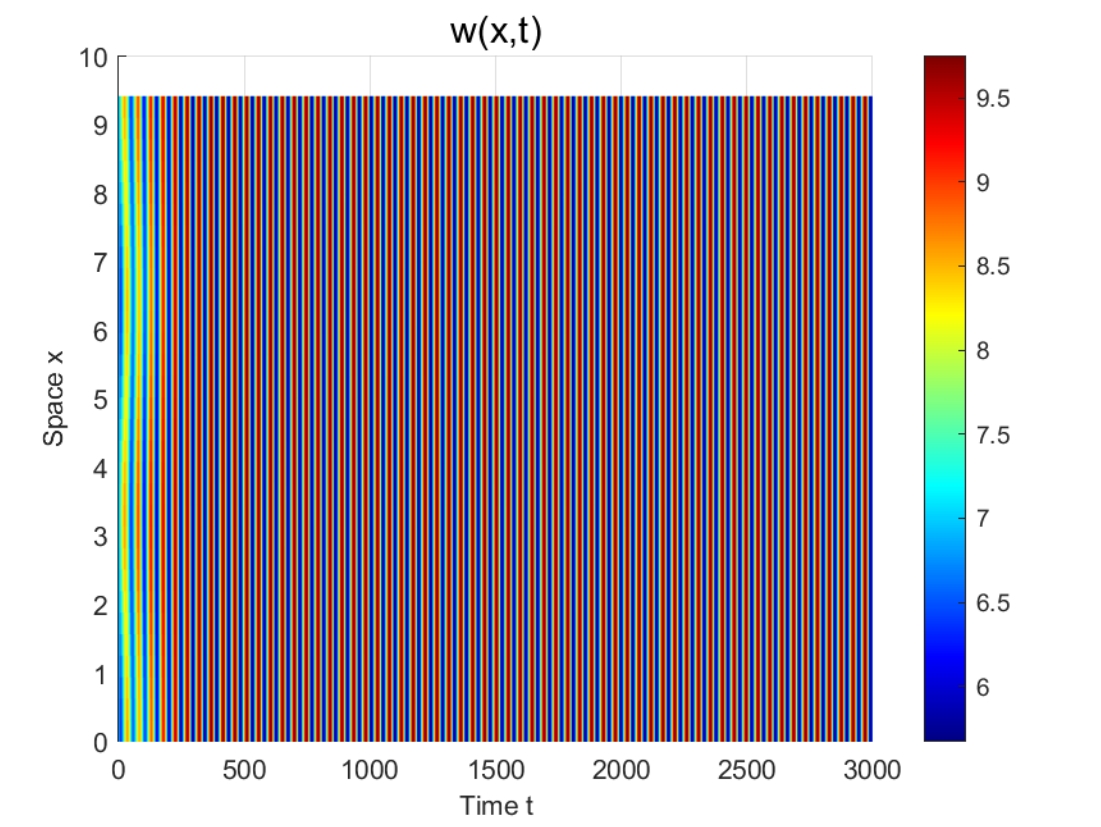}} \hfill
	\subfigure[]{\includegraphics[width=0.49\textwidth]{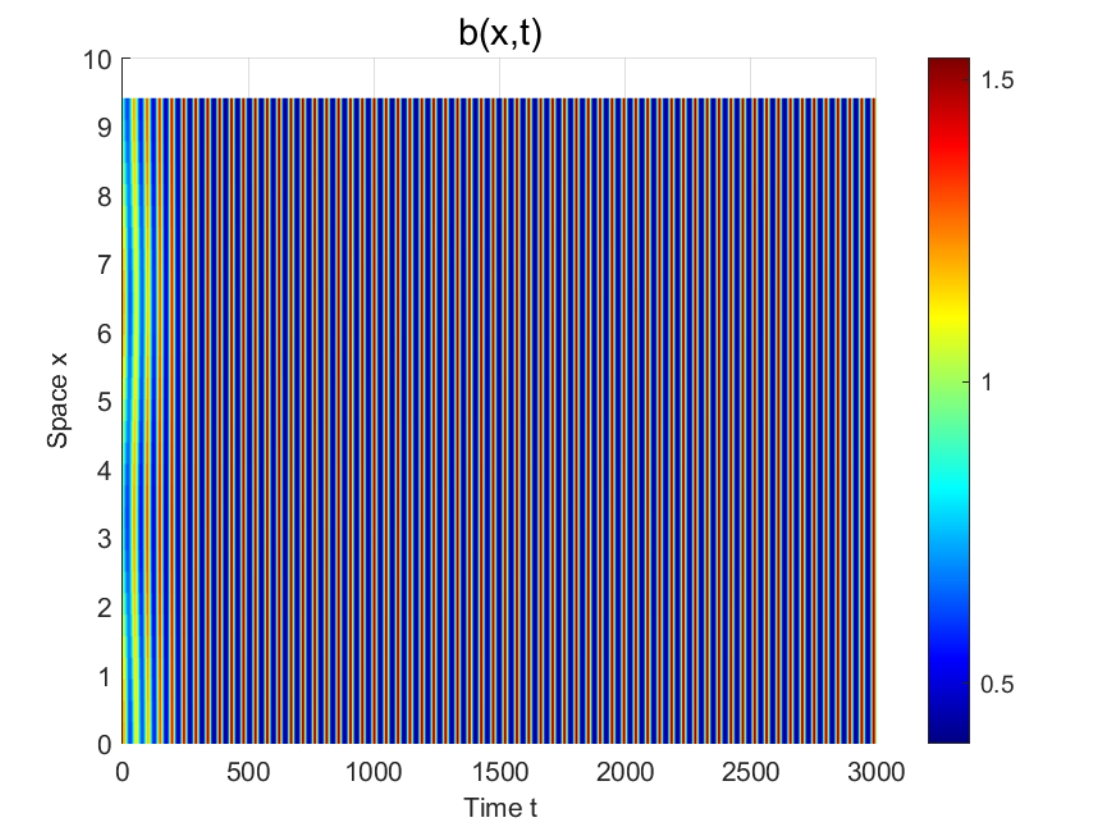}}
	\caption{When $(\varepsilon_1, \varepsilon_2)= (-0.04, -0.005) \in D_5$, the spatially homogeneous periodic solution in system (\ref{eq1.1}) is stable. The initial conditions are set as $w(x, 0)=6.91998+0.1\text{cos} x$ and $b(x, 0)=1.05093+0.1\text{cos} x$.}
	\label{fig:5.3.7}
\end{figure}
When $(\varepsilon_1, \varepsilon_2) \in D_5$, system (\ref{eq5.1}) has four equilibria: $\tilde{A}_0, \tilde{A}_1$ and $\tilde{A}_2^{\pm}$. Since $\tilde{A}_1$ is stable while $\tilde{A}_2^+$ and $\tilde{A}_2^-$ are unstable, system (\ref{eq1.1}) has a stable spatially homogeneous periodic solution, and also has a pair of unstable spatially inhomogeneous steady states, as illustrated in Fig. \ref{fig:5.3.7}.

\begin{figure}[!h]
	\centering
	\subfigure[]{\includegraphics[width=0.49\textwidth]{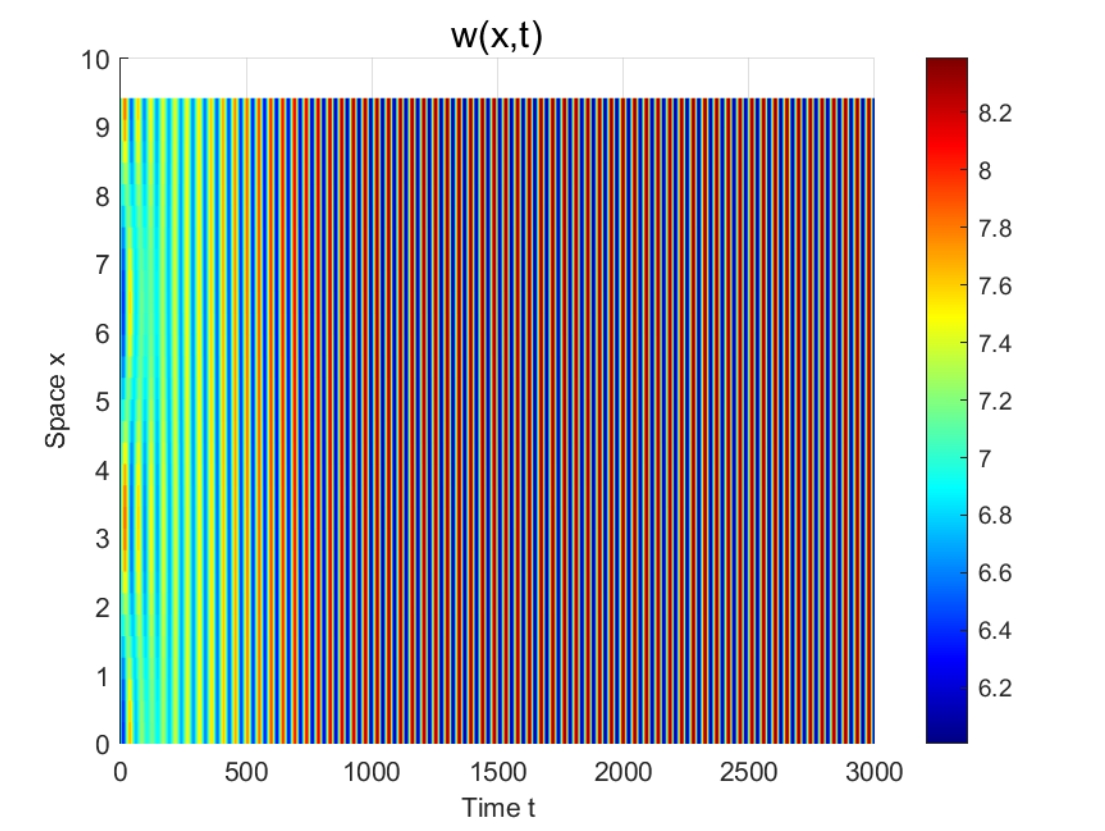}} \hfill
	\subfigure[]{\includegraphics[width=0.49\textwidth]{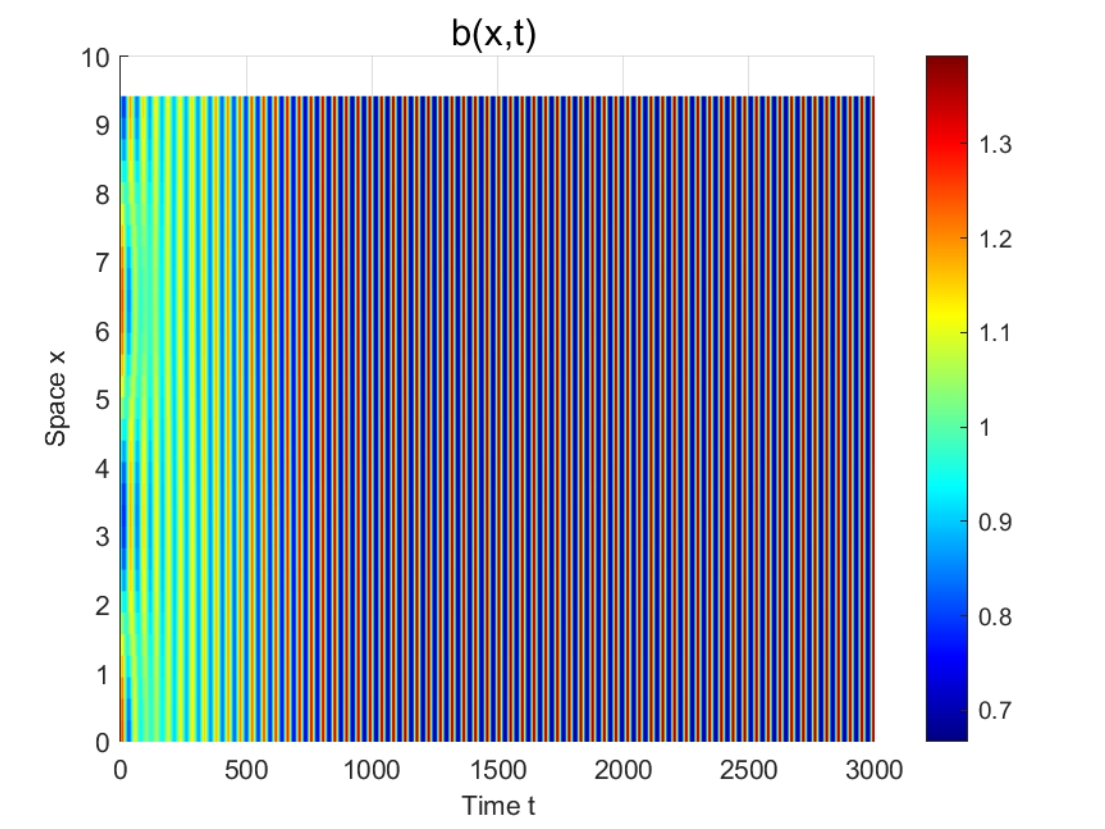}}
	\caption{When $(\varepsilon_1, \varepsilon_2)= (-0.01, -0.005) \in D_6$, the spatially homogeneous periodic solution in system (\ref{eq1.1}) is stable. The initial conditions are set as $w(x, 0)=6.91998+0.1\text{cos} x$ and $b(x, 0)=1.05093+0.1\text{cos} x$.}
	\label{fig:5.3.71}
\end{figure}
When $(\varepsilon_1, \varepsilon_2) \in D_6$, system (\ref{eq5.1}) has an unstable equilibrium $\tilde{A}_0$ and a stable equilibrium $\tilde{A}_1$, which implies that system (\ref{eq1.1}) has a spatially homogeneous periodic solution, as illustrated in Fig. \ref{fig:5.3.71}.

Furthermore, numerical results  show that the spatial distribution of the two substances (the density
of vegetation biomass and the density of water) is in antiphase state, with regions of higher vegetation biomass exhibiting lower water density, and areas of lower vegetation biomass showing higher water density.

\subsection{Effect of soil moisture $\theta_2$  on pattern formation}
In this section, we consider the effect of soil moisture $\theta_2$ on pattern formation. According to Theorem \ref{thm3.3}, system (\ref{eq1.1}) admits pattern formation when $d_1>d_1(R,k_*^2)$ with $k_*^2<\frac{a_{22} l^2}{d_2}$, where $d_1(R,k_*^2)$ is defined as in (\ref{eq3.4}).

\begin{figure}[!h]
	\centering
	\subfigure[]{\includegraphics[width=0.445\textwidth]{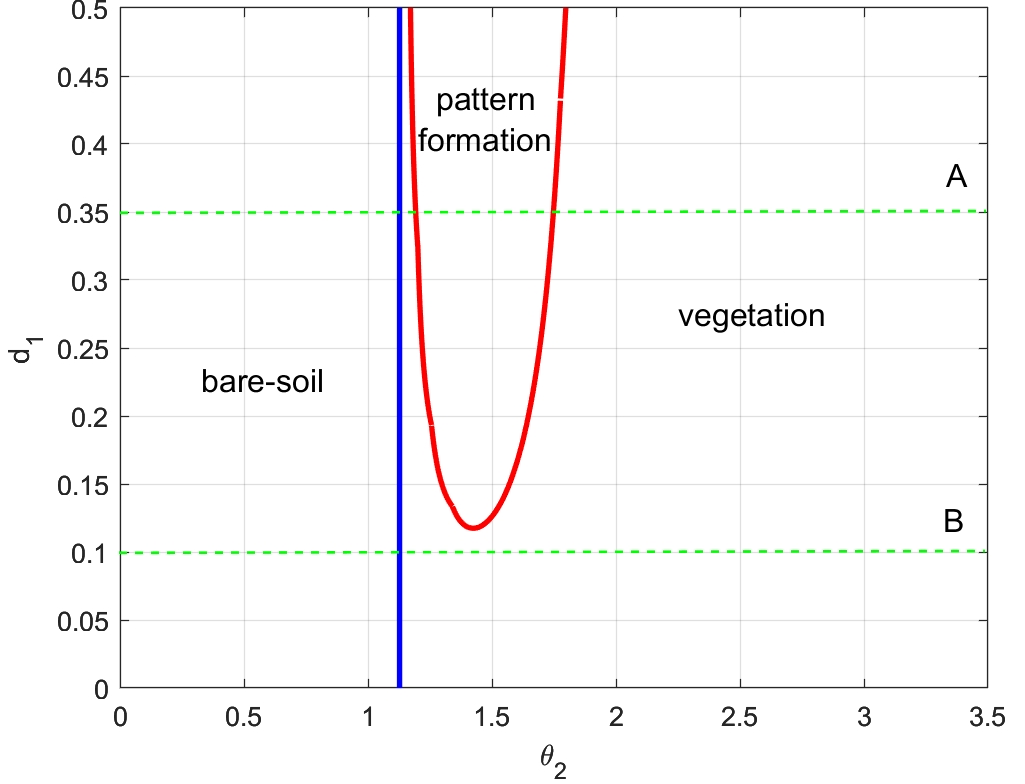}} \hspace{0.5pt}
	\subfigure[]{\includegraphics[width=0.435\textwidth]{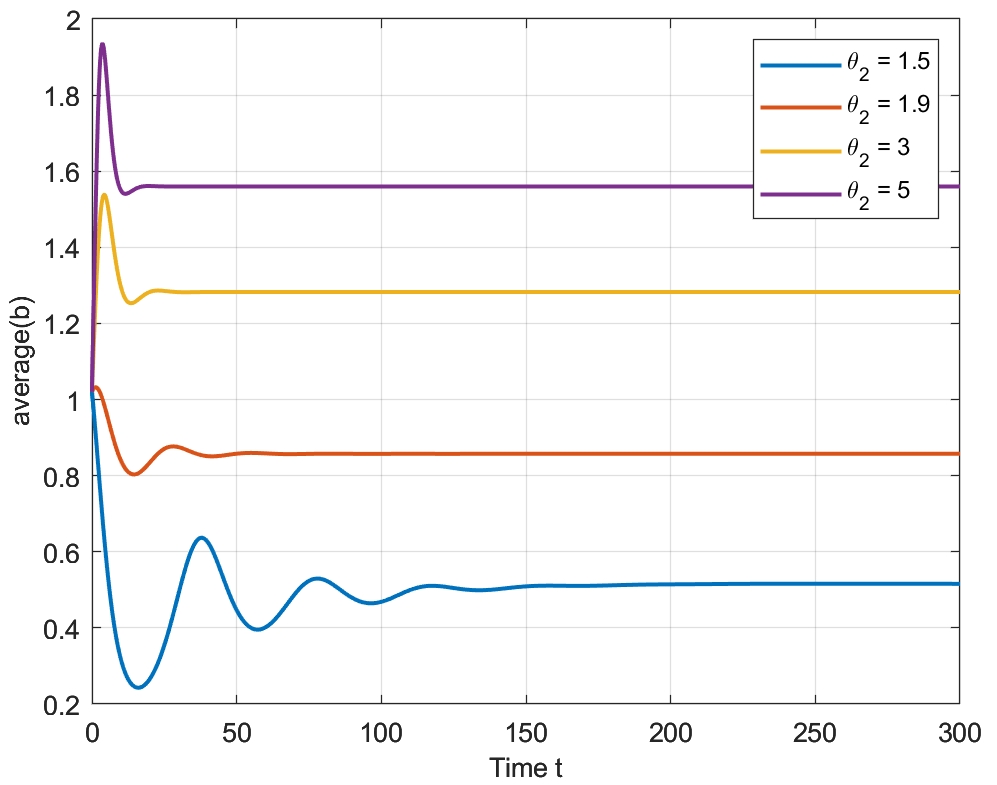}}
	\caption{(a) Effects of soil moisture on the transition between  three states: bare soil, pattern formation, and uniform vegetation,  with $R=0.8,\theta_1=5,d_2=0.01$ and $l=3$. (b) Relationship between soil moisture and the average biomass density of vegetation. }
	\label{fig:5.4.1}
\end{figure}

Note that at the bare-soil equilibrium $E_0=(\frac{R}{a},0)$, we have
$$
\begin{aligned}
	T_k &= -(d_1+d_2)\frac{k^2}{l^2} -a+\rho-\frac{\mu a}{a+\theta_2 R}, \\
	J_k &= d_1 d_2 \frac{k^4}{l^4} - \left[d_1 \left(\rho-\frac{\mu a}{a+\theta_2 R} \right) - d_2 a\right] \frac{k^2}{l^2} -a \left(\rho-\frac{\mu a}{a+\theta_2 R} \right) .
\end{aligned}
$$
Clearly, if $\rho-\frac{\mu a}{a+\theta_2 R}<0$, i.e., $\theta_2<\frac{a(\mu-\rho)}{\rho R}$, then for all $k \geq 0$, we obtain $T_k<0$ and $J_k>0$.
This indicates that $E_0$ is locally stable when $\theta_2<\frac{a(\mu-\rho)}{\rho R}$.
In the following simulations, we demonstrate that pattern formation may occur when $\theta_2>\frac{a(\mu-\rho)}{\rho R}$.

Fix the parameters $R=0.8$ and $\theta_1=5$.  It is straightforward to show that system (\ref{eq2.1}) has no positive equilibrium when $\theta_2<\frac{a(\mu-\rho)}{\rho R}$, while it has a unique, stable positive equilibrium $E_1^*$  when  $\theta_2>\frac{a(\mu-\rho)}{\rho R}$.
For $d_2=0.01$ and $l = 3$, numerical computations reveal the effect of soil moisture $\theta_2$ on pattern formation, as illustrated in Fig. \ref{fig:5.4.1}(a).
Notably, the vertical line $\theta_2=\frac{a(\mu-\rho)}{\rho R}$ divides the $\theta_2-d_1$ parameter plane into two
states where biomass is able to exist or not: the bare-soil equilibrium $E_0$ is locally stable when $\theta_2<\frac{a(\mu-\rho)}{\rho R}$ , and pattern
formation or vegetation exists when $\theta_2>\frac{a(\mu-\rho)}{\rho R}$.
More specifically, pattern formation occurs when $d_1$ exceeds the red curve $d_1(\theta_2)$, while below this curve, no patterns emerge.
Consequently, these boundaries partition the parameter plane into three distinct states: bare soil, pattern formation, and uniform vegetation.
Furthermore, for a lower water diffusion rate $d_1$ (dashed line B), decreasing soil moisture $\theta_2$ results in a direct transition from the vegetation state to the bare-soil state. In contrast, for a higher water diffusion rate $d_1$ (dashed line A), decreasing $\theta_2$ leads to multiple transitions among the vegetation, pattern formation, and bare-soil states. These findings suggest that soil moisture ultimately determines the final state of vegetation.

The relationship between soil moisture and the average biomass density of vegetation is shown in Fig. \ref{fig:5.4.1}(b). The results indicate a positive correlation between average vegetation density and $\theta_2$, suggesting that higher soil moisture levels promote increased vegetation density.

\section{Conclusion}\label{6}

To investigate the effect of soil moisture on the evolution of vegetation ecosystems in arid and semi-arid areas, we establish an improved water-vegetation model. We approach this study from the perspective of Turing patterns, discussing the conditions for diffusion-induced instability and studying the spatiotemporal dynamics of the model near the TH bifurcation point.

Specifically, we first analyze the corresponding ODE system of the model and find that the system may have up to three positive equilibria. Through theoretical analysis and numerical simulations, we uncover a rich bifurcation structure, including forward/backward transcritical bifurcations, saddle-node bifurcations, subcritical/supercritical Hopf bifurcations, and homoclinic bifurcations. Moreover, the system (2.1) exhibits both monostable and bistable dynamics, with a bistable region where bare soil equilibrium coexists with the positive equilibrium, i.e., biomass extinction coexists with biomass “steady-state persistence,” or the positive equilibrium point coexists with a stable limit cycle, or two positive equilibrium points coexist. When the rainfall rate is high, the system is in a higher vegetation equilibrium state; as the rainfall rate drops, the vegetation equilibrium state decreases and may transition into an oscillatory state. When the rainfall rate decreases further, the oscillatory state may contract back to the equilibrium state or may suddenly lead to complete extinction. In both cases, the system will exhibit periodic patterns, which may serve as early indicators of catastrophic transitions.

For the diffusion system, we use the diffusion coefficient of water ($d_1$) and the rainfall rate ($R$) as the bifurcation parameters for Turing and Hopf bifurcations, respectively, and determine the TH bifurcation point through bifurcation analysis. Using the normal form theory of reaction-diffusion equations, we derive the normal form for the TH bifurcation. Through theoretical analysis and numerical simulations, we examine the system's dynamics in different parameter regions and highlight the significant impact of small perturbations in the system parameters on the its stability. Studying the normal form of the TH bifurcation point helps us gain a deeper understanding of the system's dynamics and provides theoretical guidance for addressing potential ecological disasters, such as desertification.

In addition, we also consider the effect of soil moisture ($\theta_2$) on vegetation pattern formation. The study finds that when the diffusion coefficient of water ($d_1$) is small, as $\theta_2$ decreases, a sudden transition from a vegetation state to a bare soil state occurs. However, for larger $d_1$, there are multiple transitions among vegetation state, pattern formation, and bare soil state. This suggests that diffusion may be the primary driving force for vegetation pattern formation, while soil moisture may influence the final state of vegetation evolution.

Finally, we point out that many factors influence vegetation pattern structure. How to establish a model that comprehensively considers various influencing factors remains an important issue worth further exploration. Addressing this problem will help better protect the ecological environment.

\appendix
\section*{Appendix A}
In this section, we provide the calculation formulas for $C_{ijk}, D_{ijk}$ and $E_{ijk}$. Firstly, denote
$$
G_{j_1j_2}=(G_{j_1j_2}^{(1)},G_{j_1j_2}^{(2)})^T,
$$
with
$$
G_{j_1j_2}^{(k)}=\frac{\partial G^{(k)}(0,0,0,0)}{\partial w^{j_1}\partial b^{j_2}},\, k = 1,2,\, j_1 + j_2=2,3.
$$
Then, we can figure out by calculation that
\begin{align*}
	A_{200}&=G_{20} + 2p_{02}G_{11}+p_{02}^2G_{02}=\overline{A}_{020},\\
	A_{002}&=G_{20} + 2p_{k_*2}G_{11}+p_{k_*2}^2G_{02},\\
	A_{110}&=2(G_{20}+2Re(p_{02})G_{11}+|p_{02}|^2G_{02}),\\
	A_{101}&=2(G_{20}+(p_{02}+p_{k_*2})G_{11}+p_{02}p_{k_*2}G_{02})=\overline{A}_{011},\\
	A_{210}&=3(G_{30}+(2p_{02}+\overline{p}_{02})G_{21}+(p_{02}^2 + 2|p_{02}|^2)G_{12}+|p_{02}|^2p_{02}G_{03}),\\
	A_{102}&=3(G_{30}+(p_{02}+2p_{k_*2})G_{21}+(p_{k_*2}^2+2p_{02}p_{k_*2})G_{12}+p_{02}p_{k_*2}^2G_{03}),\\
	A_{111}&=6(G_{30}+(p_{k_*2}+2Re(p_{02}))G_{21}+(|p_{02}|^2 + 2p_{k_*2}Re(p_{02}))G_{12}+|p_{02}|^2p_{k_*2}G_{03}),\\
	A_{003}&=G_{30}+3(p_{k_*2}G_{21}+p_{k_*2}^2G_{12})+p_{k_*2}^3G_{03}.
\end{align*}
Therefore, we can obtain
\begin{align*}
	C_{210}&=\frac{1}{6l\pi}q_0^T A_{210}, \ C_{102}=\frac{1}{6l\pi}q_0^T A_{102},\\
	C_{111}&=\frac{1}{6l\pi}q_{k_*}^T A_{111}, \ C_{003}=\frac{1}{6l\pi}q_{k_*}^T A_{003},
\end{align*}
\begin{align*}
	D_{210}&=\frac{1}{6l\pi i\omega_0}\left[ - (q_0^T A_{200})(q_0^T A_{110})+(q_0^T A_{110})(\overline{q_0^T} A_{110})+\frac{2}{3}(q_0^T A_{020})(\overline{q_0^T} A_{200})\right],\\
	D_{102}&=\frac{1}{6l\pi i\omega_0}\left[ - 2(q_0^T A_{200})(q_0^T A_{002})+(q_0^T A_{110})(\overline{q_0^T} A_{002})+2(q_0^T A_{002})(\overline{q_{k_*}^T} A_{101})\right],\\
	D_{111}&=\frac{1}{6l\pi i\omega_0}\left[ (q_{k_*}^T A_{011})(\overline{q_0^T} A_{110})-(q_{k_*}^T A_{101})(q_0^T A_{110})\right],\\
	D_{003}&=\frac{1}{6l\pi i\omega_0}\left[ (q_{k_*}^T A_{011})(\overline{q_0^T} A_{002})-(q_{k_*}^T A_{101})(q_0^T A_{002})\right],\\
	E_{210}&=\frac{1}{6}q_0^T\left[ S_{yz_1}\langle h_{00110}\rangle+S_{yz_2}\langle h_{00200}\rangle\right],\\
	E_{102}&=\frac{1}{6}q_0^T\left[ S_{yz_1}\langle h_{00002}\rangle+S_{yz_2}\langle h_{k_*0101}\rangle\right],\\
	E_{111}&=\frac{1}{6}q_{k_*}^T\left[ S_{yz_1}\langle h_{0k_*011}\rangle+S_{yz_2}\langle h_{0k_*101}\rangle+S_{yz_3}\langle h_{k_*k_*110}\rangle\right],\\
	E_{003}&=\frac{1}{6}q_{k_*}^T\left[ S_{yz_3}\langle h_{k_*k_*002}\rangle\right],
\end{align*}
with
\begin{align*}
	S_{yz_1}&=\begin{pmatrix}
		(2A_{13}+A_{14}p_{02})|_{\varepsilon_1=0} & (A_{14}+2A_{15}p_{02})|_{\varepsilon_1=0}\\
		(2A_{23}+A_{24}p_{02})|_{\varepsilon_1=0} & (A_{24}+2A_{25}p_{02})|_{\varepsilon_1=0}
	\end{pmatrix},
\end{align*}
\begin{align*}
	S_{yz_2}&=\begin{pmatrix}
		(2A_{13}+A_{14}\overline{p}_{02})|_{\varepsilon_1=0} & (A_{14}+2A_{15}\overline{p}_{02})|_{\varepsilon_1=0}\\
		(2A_{23}+A_{24}\overline{p}_{02})|_{\varepsilon_1=0} & (A_{24}+2A_{25}\overline{p}_{02})|_{\varepsilon_1=0}
	\end{pmatrix},\\
	S_{yz_3}&=\begin{pmatrix}
		(2A_{13}+A_{14}\overline{p}_{k_* 1})|_{\varepsilon_1=0} & (A_{14}+2A_{15}\overline{p}_{k_*1})|_{\varepsilon_1=0}\\
		(2A_{23}+A_{24}\overline{p}_{k_*1})|_{\varepsilon_1=0} & (A_{24}+2A_{25}\overline{p}_{k_*1})|_{\varepsilon_1=0}
	\end{pmatrix},
\end{align*}
and
\begin{align*}
	&h_{00110}=\frac{1}{l\pi}\left\{ -L_0^{-1}A_{110}+\frac{1}{i\omega_0}\left[q_0^T A_{110}p_0+\frac{1}{3}\overline{q_0^T}A_{110}\overline{p_0}\right]\right\},\\
	&h_{00200}=\frac{1}{l\pi}\left\{(2i\omega_0 - L_0)^{-1}A_{200}-\frac{1}{i\omega_0}\left[q_0^T A_{200}p_0+\frac{1}{3}\overline{q_0^T}A_{200}\overline{p_0}\right]\right\},\\
	&h_{00002}=-\frac{1}{l\pi}L_0^{-1}A_{002}+\frac{1}{i\omega_0 l\pi}\left[q_0^T A_{002}p_0-\overline{q_0^T}A_{002}\overline{p_0}\right],\\
	&h_{k_*0101}=\frac{1}{l\pi}\left[i\omega_0+\frac{k_*^2}{l^2}D - L_0\right]^{-1}A_{101}-\frac{1}{i\omega_0 l\pi}q_{k_*}^T A_{101}p_{k_*},\\
	&h_{0k_*011}=\frac{1}{l\pi}\left[-i\omega_0+\frac{k_*^2}{l^2}D - L_0\right]^{-1}A_{011}+\frac{1}{i\omega_0 l\pi}q_{k_*}^T A_{011}p_{k_*},\\
	&h_{k_*k_*002}=\frac{1}{2l\pi}\left[\frac{(2k_*)^2}{l^2}D - L_0\right]^{-1}A_{002}+h_{00002},\\
	&h_{0k_*101}=h_{k_*0101},\ h_{k_*k_*110}=h_{00110}.
\end{align*}

\bibliographystyle{plain}
\bibliography{references} 

\end{document}